\pdfoutput=1
\documentclass[11pt,a4paper,reqno]{amsart}

\usepackage[british]{babel}

\usepackage[DIV=9,oneside,BCOR=0mm]{typearea}
\usepackage{microtype}
\usepackage[T1]{fontenc}
\usepackage{setspace}
\usepackage{amssymb}
\usepackage{amsmath}
\usepackage{amsthm}
\usepackage{enumitem}
\usepackage{mathrsfs}
\usepackage[colorlinks,citecolor=blue,urlcolor=blue,linkcolor=blue,linktocpage]{hyperref}
\usepackage{mathrsfs}
\usepackage{mathtools}
\usepackage{xcolor}
\usepackage{xfrac}

\newtheorem{innercustomgeneric}{\customgenericname}
\providecommand{\customgenericname}{}
\newcommand{\newcustomtheorem}[2]{%
	\newenvironment{#1}[1]
	{%
		\renewcommand\customgenericname{#2}%
		\renewcommand\theinnercustomgeneric{##1}%
		\innercustomgeneric
	}
	{\endinnercustomgeneric}
}

\newcustomtheorem{customthm}{Theorem}
\newcustomtheorem{customcor}{Corollary}

\newtheorem{thm}{Theorem}[section]

\newtheorem{cor}[thm]{Corollary}
\newtheorem{lem}[thm]{Lemma}
\newtheorem{prop}[thm]{Proposition}

\theoremstyle{definition}

\newtheorem{exmpl}[thm]{Example}

\newtheorem{definition}[thm]{Definition}

\newtheorem{remark}[thm]{Remark}

\renewcommand{\epsilon}{\varepsilon}
\renewcommand{\phi}{\varphi}
\newcommand{\defeq}{\mathrel{\mathop:}=}                         
\newcommand{\eqdef}{\mathrel{\mathopen={\mathclose:}}}

\newcommand{\less}{\mathrel{\leq_{\rm{ess}}}}   

\renewcommand{\Re}{\operatorname{Re}}

\DeclareMathOperator{\Fix}{Fix}
\DeclareMathOperator{\dist}{dist}
\DeclareMathOperator{\Sym}{Sym}
\DeclareMathOperator{\Alt}{Alt}
\DeclareMathOperator{\free}{F}
\DeclareMathOperator{\cf}{cf}
\DeclareMathOperator{\scf}{scf}
\DeclareMathOperator{\End}{End}
\DeclareMathOperator{\Hom}{Hom}
\DeclareMathOperator{\M}{M}
\DeclareMathOperator{\I}{I}
\DeclareMathOperator{\IM}{I_{mo}}
\DeclareMathOperator{\IME}{I_{mo}^{\ast}}
\DeclareMathOperator{\lat}{L}

\DeclareMathOperator{\E}{E}
\DeclareMathOperator{\Emo}{E_{mo}}
\DeclareMathOperator{\U}{U}
\DeclareMathOperator{\GLM}{GL_{mo}}
\DeclareMathOperator{\GL}{GL}
\DeclareMathOperator{\EL}{EL}
\DeclareMathOperator{\cent}{Z}
\DeclareMathOperator{\Cl}{Cl}

\DeclareMathOperator{\im}{im}

\DeclareMathOperator{\C}{\mathbb{C}}
\DeclareMathOperator{\R}{\mathbb{R}}

\DeclareMathOperator{\Z}{\mathbb{Z}}
\DeclareMathOperator{\N}{\mathbb{N}}

\DeclareMathOperator{\Pow}{\mathcal{P}}
\DeclareMathOperator{\Pfin}{\mathcal{P}_{fin}}
\DeclareMathOperator{\Neigh}{\mathcal{U}}
\DeclareMathOperator{\Sub}{Sub}

\makeatletter
\def\moverlay{\mathpalette\mov@rlay}
\def\mov@rlay#1#2{\leavevmode\vtop{%
		\baselineskip\z@skip \lineskiplimit-\maxdimen
		\ialign{\hfil$\m@th#1##$\hfil\cr#2\crcr}}}
\newcommand{\charfusion}[3][\mathord]{
	#1{\ifx#1\mathop\vphantom{#2}\fi
		\mathpalette\mov@rlay{#2\cr#3}
	}
	\ifx#1\mathop\expandafter\displaylimits\fi}
\makeatother

\newcommand{\cupdot}{\charfusion[\mathbin]{\cup}{\cdot}}
\newcommand{\bigcupdot}{\charfusion[\mathop]{\bigcup}{\boldsymbol{\cdot}}}

\begin{document}
%%%%%%%%%%%%%%%%%%%%%%%%%%%%%%%%

\setlist{noitemsep}

\author{Friedrich Martin Schneider}
\address{F.M.~Schneider, Institute of Discrete Mathematics and Algebra, TU Bergakademie Freiberg, 09596 Freiberg, Germany}
\email{martin.schneider@math.tu-freiberg.de}
\thanks{This research is funded by the Deutsche Forschungsgemeinschaft (DFG, German Research Foundation) -- Projektnummer 561178190}

\title[On unit groups of purely infinite rings]{On unit groups of purely infinite rings}
\date{\today}

\begin{abstract} 
	We prove that the unit group of any purely infinite, regular, right self-injective ring admits a universal sequence of words (thus has strong distortion), is minimally almost periodic, has finite involution width, is verbally simple and verbally elliptic, and has uncountable uniform strong cofinality. These findings are consequences of more refined structural results, some of which provide concrete cardinal estimates involving the Goodearl--Boyle dimension of the ring. A central ingredient in our results is a comprehensive study of bounded normal generation, in particular with respect to certain classes of involutions. Along the way, we develop some abstract methods and concepts pertaining to geometric group theory.
\end{abstract}

\subjclass[2020]{16D50, 16E50, 20E45, 20F65}

\keywords{Regular ring, self-injectivity, unit group, bounded normal generation, verbal simplicity, width, minimal almost periodicity, Bergman property, cofinality.}

\maketitle

\allowdisplaybreaks

%%%%%%%%%%%%%%%%%%%%%%%%%%%%%%%%%%%%%%%%%%
%%%%%%%%%%%%%%%%%%%%%%%%%%%%%%%%%%%%%%%%%%

\tableofcontents

\section{Introduction}

This work is a contribution to the understanding of unit groups of rings. Every ring $R$ embeds naturally into its maximal right quotient ring~\cite[\S2.C]{GoodearlRingTheory}, which is right self-injective and regular if $R$ is right non-singular~\cite[Corollary~2.31]{GoodearlRingTheory}. Furthermore, every regular, right self-injective ring is isomorphic to a direct product of a directly finite and a purely infinite ring~\cite[Proposition~10.21]{GoodearlBook}, and such an isomorphism in turn gives rise to a direct product decomposition of the associated unit group. The present manuscript focuses on the latter factor in this decomposition, i.e., it studies unit groups of purely infinite, regular, right self-injective rings. Our results pertain to several different aspects of geometric group theory.

One of our main insights offers a deeper understanding of conjugacy classes and the phenomenon of bounded normal generation in the considered unit groups. A regular, right self-injective ring $R$ is purely infinite if and only if $R$ contains a \emph{moiety}\footnote{Our choice of terminology is inspired by an analogous set-theoretic notion (cf.~Section~\ref{section:moieties}).}, i.e., an idempotent element $e \in R$ such that ${eR_{R}} \cong {R_{R}} \cong {(1-e)R_{R}}$ (Proposition~\ref{proposition:pure.infinity.via.moieties}). We analyze elements of the unit group $\GL(R)$ based on their behavior with respect to such moieties (Definition~\ref{definition:moietious.involutions}) and, specifically, study the set $\GLM(R)$ of those units $g \in \GL(R)$ which satisfy $e \perp {geg^{-1}}$ for some moiety $e \in R$. Building on ideas of Menal and Moncasi~\cite{MenalMoncasi} (see Section~\ref{section:matrix.decomposition}), we establish the following.

\begin{customthm}{A}[Theorem~\ref{theorem:generation.by.moietious.units}]\label{theorem:a} Let $R$ be a regular, right self-injective ring. Then, for every $g \in \GLM(R)$, \begin{displaymath}
	\GL(R) \, = \, \Cl_{\GL(R)}\!\left( \left\{ g,g^{-1}\right\}\right)^{128} .
\end{displaymath} \end{customthm}

One application of Theorem~\ref{theorem:a} yields minimal almost periodicity (Corollary~\ref{corollary:unit.group.minap}), even a qualitative version of this property (Proposition~\ref{proposition:uniformly.bounded.normal.generation.over.large.subgroup}), for the considered family of unit groups. The insights gained in the course of the proof of Theorem~\ref{theorem:a}, in particular those concerning a certain type of involutions, have further structural ramifications, as outlined below.

Menal and Moncasi have shown that the unit group of any purely infinite, regular, right self-injective ring is perfect~\cite[Corollary~2.8]{MenalMoncasi}. We strengthen and extend this result through Theorem~\ref{theorem:b}. Given a group $G$ and considering the free group $\free (X)$ over a set $X$, let us recall that any map $g \in G^{X}$ naturally extends to a unique homomorphism $\free (X) \to G, \, w \mapsto w(g)$, and define \begin{displaymath}
	w(G) \, \defeq \, \left\{ w(g) \left\vert \, g \in G^{X} \right\}\! \right. \qquad (w \in \free(X)) .
\end{displaymath} A group $G$ is said to be \emph{verbally simple} if, for every $w \in \free (\N)$, the subgroup\footnote{Note that $\langle w(G) \rangle \unlhd G$, as $w(G)$ is closed under conjugation in $G$.} $\langle w(G) \rangle$ equals $G$ or $\{ 1\}$. Furthermore, a group $G$ is called \emph{verbally elliptic}~\cite[\S1.1, p.~2]{segal} if, for every $w \in \free(\N)$, there exists $m \in \N$ such that \begin{displaymath}
	\langle w(G) \rangle \, = \, \left( w(G) \cup w(G)^{-1} \right)^{m} .
\end{displaymath} Combining our own structural results with work of Larsen and Shalev~\cite{LarsenShalev08,LarsenShalev09}, we establish the following theorem.
	
\begin{customthm}{B}[Theorem~\ref{theorem:verbally.elliptic}]\label{theorem:b} Let $R$ be a purely infinite, regular, right self-injective ring. For every $w \in \free(\N)\setminus \{ 1 \}$, \begin{displaymath}
		\GL(R) \, = \, w(\GL(R))^{128} .
\end{displaymath} In particular, $\GL(R)$ is verbally simple and verbally elliptic. \end{customthm}

Our work provides further information concerning word images in the considered unit groups. Recall that a \emph{universal sequence} for a group $G$ is a sequence $w \in \free(m)^{\N}$ with $m \in \N$ such that \begin{displaymath}
	\forall g \in G^{\N} \ \exists h \in G^{m} \ \forall n \in \N \colon \qquad g_{n} \, = \, w_{n}(h) .
\end{displaymath} This notion has received substantial interest in recent years (see~\cite{HydeJonusasMitchellPeresse} for references). The existence of a universal sequence has non-trivial ramifications for a group, notably \emph{strong distortion}, as discussed in Section~\ref{section:strong.distortion.and.cofinality}. Revisiting the work of Galvin on infinite symmetric groups~\cite{galvin}, we isolate an abstract method for establishing the existence of universal sequences using wreath products (Lemma~\ref{lemma:galvin}\ref{lemma:galvin.2}). In combination with additional ring-theoretic insights, this criterion is then used to prove the following.

\begin{customthm}{C}[Theorem~\ref{theorem:universal.sequence}]\label{theorem:c} Let $R$ be a purely infinite, regular, right self-injective ring. Then $\GL(R)$ admits a universal sequence. In particular, $\GL(R)$ has strong distortion. \end{customthm}

Strong distortion of a group implies \emph{uncountable strong cofinality} and thus \emph{Bergman's property} (see Remark~\ref{remark:strong.distortion.implies.uscf}), phenomena originally discovered in infinite symmetric groups by Bergman~\cite{bergman} and subsequently found in many other natural examples (see, e.g.,~\cite{DrosteGoebel,DrosteHolland,tolstykh,TolstykhFree,Cornulier06,DrosteHollandUlbrich,KechrisRosendal,RicardRosendal,dowerk}). In particular, the groups considered in Theorem~\ref{theorem:c} possess these properties (Corollary~\ref{corollary:scf.for.rings}), which readily strengthens and generalizes a result by Tolstykh~\cite[Theorem~2.1]{tolstykh} about endomorphism rings of infinite-dimensional right vector spaces over division rings.

Exhibiting a proper strengthening of the conclusion pointed out in the preceding paragraph, we introduce the following cardinal invariant for groups (Definition~\ref{definition:uniform.strong.cofinality}): given any positive integer $k$, we define the \emph{$k$-strong cofinality} $\scf_{k}(G)$ of an infinite group $G$ to be the smallest cardinality of a directed subset $\mathcal{D}$ of $(\Pow(G),{\subseteq})$ such that $\bigcup \mathcal{D} = G \notin \{ D^{k} \mid D \in \mathcal{D} \}$. This definition is inspired by and modeled after a uniform version of uncountable strong cofinality studied in a number of earlier works~\cite{MillerThesis,KechrisRosendal,RicardRosendal,rosendal,dowerk}. We provide an abstract criterion for proving lower bounds for uniform strong cofinalities (Theorem~\ref{theorem:abstract.2}). This method is then applied to the unit group of a regular, right self-injective ring $R$ to establish the following connection with the ring's \emph{Goodearl--Boyle dimension} $\mu(R)$ (Definition~\ref{definition:goodearl.boyle}).

\begin{customthm}{D}[Theorem~\ref{theorem:uniform.bergman.unit.group}]\label{theorem:d} Let $R$ be a non-zero, purely infinite, regular, right self-injective ring. Then \begin{displaymath}
	\aleph_{0} \, < \, \mu(R) \, \leq \, \scf_{6144}(\GL(R)) .
\end{displaymath} \end{customthm}

The conclusion of Theorem~\ref{theorem:d} does not merely imply uncountable strong cofinality, but also entails a representation-theoretic rigidity phenomenon related to Kazhdan's property $(T)$: as established in Proposition~\ref{proposition:restricted.kazhdan.property}, if a group $G$ satisfies $\scf_{k}(G) > \alpha$ for an infinite cardinal $\alpha$ and some positive integer $k$, then $G$ has \emph{property $(T_{\alpha})$}, i.e., every unitary representation of $G$ admitting an almost invariant net of unit vectors of cardinality at most $\alpha$ must actually have an invariant unit vector (Definition~\ref{definition:restricted.property.t}). On the other hand, infinite groups with Bergman's property cannot be finitely generated, thus fail to have property $(T)$. Since $R \cong \M_{n}(R)$ and thus $\GL(R) \cong \GL_{n}(R)$ for every purely infinite, regular, right self-injective ring $R$ and every positive integer $n$ (cf.~Theorem~\ref{theorem:purely.infinite}), the resulting Corollary~\ref{corollary:restricted.kazhdan.property} should be contrasted with the work of Ershov and Jaikin-Zapirain~\cite[Theorem~1.1]{ErshovJaikinZapirain} establishing property $(T)$ for the matrix group $\EL_{n}(S) \leq \GL_{n}(S)$ generated by the set of elementary transvections over any finitely generated unital ring $S$ and any integer $n \geq 3$.

This article is organized as follows. Supplying some background concerning several rigidity phenomena in infinite groups, we revisit an approach to minimal almost periodicity due to Hart and Kunen in Section~\ref{section:minimal.almost.periodicity}, isolate an abstract criterion for the existence of universal sequences inspired by work of Galvin in Section~\ref{section:word.images}, and discuss some topological ramifications of uncountable strong cofinality in Section~\ref{section:strong.distortion.and.cofinality}. In Section~\ref{section:uniform.strong.cofinality} we introduce a family of cardinal invariants for groups, namely uniform strong cofinalities, for which we establish structural lower bounds in Section~\ref{section:general.criterion} as well as a connection with a family of relaxations of Kazhdan's property $(T)$ in Section~\ref{section:kazhdan}. The subsequent Section~\ref{section:idempotents} marks the beginning of our study of unit groups of rings, starting with several basic facts about conjugacy of idempotents. In Section~\ref{section:strcuture.theory}, we recollect some relevant material from the structure theory of regular, right self-injective rings. In Section~\ref{section:matrix.decomposition}, we revisit and refine some work by Menal and Moncasi concerning the decomposition of unit groups of matrix rings. In Section~\ref{section:moieties}, we turn to the study of moieties and several classes of units defined by their interaction with such. The following Section~\ref{section:essential.embeddings} revolves around a particular construction of embeddings and contains the proof of Theorem~\ref{theorem:c}. In Section~\ref{section:bng}, we investigate bounded normal generation with respect to several classes of units and provide proofs for Theorem~\ref{theorem:a} and Theorem~\ref{theorem:b}. The final Section~\ref{section:goodearl.boyle.dimension} is dedicated to a discussion of Goodearl--Boyle dimension, leading to the proof of Theorem~\ref{theorem:d} and a number of corollaries.

\section{Minimal almost periodicity}\label{section:minimal.almost.periodicity}

A group $G$ is said to be \emph{minimally almost periodic} if every homomorphism from $G$ to a compact Hausdorff topological group is trivial. The present section reviews and extends an approach to minimal almost periodicity due to Hart and Kunen~\cite[Section~5]{HartKunen}. The following definition is nothing but a slight variation on their notion of \emph{badness}~\cite[Definition~5.3]{HartKunen}. Here and throughout the entire manuscript, if $G$ is a group, then we let \begin{displaymath}
	\Cl_{G}(S) \, \defeq \, \left. \! \left\{ gsg^{-1} \, \right\vert s \in S, \, g \in G \right\} 
\end{displaymath} for any $S \subseteq G$ and abbreviate $\Cl_{G}(g) \defeq \Cl_{G}(\{ g \})$ for any $g \in G$.

\begin{definition} Let $G$ be a group and let $T \subseteq G$. Then $G$ is said to have \emph{uniformly bounded normal generation over} $T$ if there exists $n \in \N$ such that \begin{displaymath}
	\forall g \in T\setminus \{ 1 \} \colon \qquad G \, = \, \Cl_{G}\!\left( \left\{ g,g^{-1}\right\}\right)^{n} .
\end{displaymath}\end{definition}

We assume familiarity with the notion of a \emph{uniform space}~\cite[II, \S1.1]{bourbakitop1}. Given a topological group $G$, we let $\Neigh(G)$ denote the neighborhood filter at its neutral element, and we view $G$ as a uniform space endowed with its \emph{right uniformity}, i.e., the one generated by the basic entourages of the form \begin{displaymath}
	U_{\Rsh} \, \defeq \, \left\{ (x,y) \in G \times G \left\vert \, xy^{-1} \in U \right\} \! \right. \qquad (U \in \Neigh(G)) .
\end{displaymath} Moreover, recall that a topological group $G$ is said to have \emph{small invariant neighborhoods} if $\Neigh(G)$ admits a filter base consisting of conjugation-invariant sets, i.e., \begin{displaymath}
	\forall U \in \Neigh(G) \ \exists V \in \Neigh(G) \colon \qquad (V \subseteq U) \, \wedge \, \left(\forall g \in G \colon \, gVg^{-1} = V\right) .
\end{displaymath} 

\begin{lem}\label{lemma:rigidity.for.SIN.reps} Let $G$ be a group and let $T \subseteq G$ be such that $G$ has uniformly bounded normal generation over $TT^{-1}$. Moreover, let $H$ be a Hausdorff topological group with small invariant neighborhoods and let $\pi \colon G \to H$ be a non-trivial homomorphism. Then $\pi\vert_{T}$ is injective and $\pi(T)$ is uniformly discrete in $H$. \end{lem}

\begin{proof} Let $n \in \N$ be such that $G = \Cl_{G}\!\left( \left\{ g,g^{-1}\right\}\right)^{n}$ for every $g \in TT^{-1}\setminus \{ 1 \}$. Since $H$ is Hausdorff, there exists $U \in \Neigh(H)$ such that $\pi(G) \nsubseteq U$. As $H$ is a topological group with small invariant neighborhoods, we find some conjugation-invariant $V \in \Neigh(H)$ such that $V^{-1} = V$ and $V^{n} \subseteq U$. Now, if there exist $s,t \in T$ such that $s \ne t$ and $\pi(s)\pi(t)^{-1} \in V$, then also $\pi(t)\pi(s)^{-1} \in V^{-1} = V$ and hence \begin{align*}
	\pi(G) \, = \, \pi\!\left( \Cl_{G}\!\left( \left\{ st^{-1},ts^{-1}\right\}\right)^{n} \right) \, &\subseteq \, \Cl_{H}\!\left( \left\{ \pi(s)\pi(t)^{-1},\pi(t)\pi(s)^{-1}\right\}\right)^{n} \\
		& \subseteq \, \Cl_{H}(V)^{n} \, = \, V^{n} \, \subseteq \, U ,
\end{align*} which contradicts our choice of $U$. Consequently, $\pi(s)\pi(t)^{-1} \notin V$ for any two distinct $s,t \in T$, i.e., $\pi\vert_{T}$ is injective and $\pi(T)$ is $V_{\Rsh}$-discrete in $H$. \end{proof}

Extending the notion for topological groups~\cite[\S5.1, p.~286]{AT} in an obvious way, a uniform space $X$ will be called \emph{$\alpha$-narrow} with respect to an infinite cardinal $\alpha$ if, for every entourage $E$ of $X$, there exists $A \subseteq X$ such that $X=\bigcup_{a \in A} \{ x \in X \mid (x,a) \in E \}$ and $\vert A \vert \leq \alpha$. A standard application of Zorn's lemma proves the two well-known basic facts recorded in the following remark.

\begin{remark}\label{remark:narrow} Let $X$ be a uniform space. \begin{enumerate}
	\item\label{remark:narrow.1} $X$ is precompact if and only if every uniformly discrete subset of $X$ is finite.
	\item\label{remark:narrow.2} Let $\alpha$ be an infinite cardinal. Then $X$ is $\alpha$-narrow if and only if every uniformly discrete subset of $X$ has cardinality at most $\alpha$.
\end{enumerate} \end{remark}

\begin{prop}[cf.~{\cite[Corollary~5.7]{HartKunen}}]\label{proposition:rigidity.for.SIN.reps} Let $G$ be a group. Suppose that $G$ has uniformly bounded normal generation over~$TT^{-1}$ for some subset $T \subseteq G$. \begin{enumerate}
	\item\label{proposition:rigidity.for.SIN.reps.1} If $T$ is infinite, then $G$ is minimally almost periodic.
	\item\label{proposition:rigidity.for.SIN.reps.2} If $\alpha$ is an infinite cardinal less than $\vert T \vert$, then every homomorphism from $G$ to any $\alpha$-narrow Hausdorff topological group with small invariant neighborhoods is trivial.
\end{enumerate} \end{prop}

\begin{proof} \ref{proposition:rigidity.for.SIN.reps.1} Since every compact topological group has small invariant neighborhoods~\cite[Corollary~1.8.7]{AT}, this follows from Lemma~\ref{lemma:rigidity.for.SIN.reps} and Remark~\ref{remark:narrow}\ref{remark:narrow.1}.
	
\ref{proposition:rigidity.for.SIN.reps.2} This is a direct consequence of Lemma~\ref{lemma:rigidity.for.SIN.reps} and Remark~\ref{remark:narrow}\ref{remark:narrow.2}. \end{proof}

\section{Word images}\label{section:word.images}

This section provides some relevant background concerning word images and universal sequences of words for groups, including a technical tool (Lemma~\ref{lemma:galvin}) extracted from Galvin's work~\cite{galvin}.

Let $X$ be a set. Then $\free(X)$ will denote the free group over $X$. Given a group~$G$ and recollecting the fact that any map $g \in G^{X}$ naturally extends to a unique homomorphism $\free (X) \to G, \, w \mapsto w(g)$, we define \begin{displaymath}
	w(G) \, \defeq \, \left\{ w(g) \left\vert \, g \in G^{X} \right\}\! \right. \qquad (w \in \free(X)) 
\end{displaymath} as well as \begin{displaymath}
	W(G) \, \defeq \, \bigcup\nolimits \{ w(G) \mid w \in W \} \qquad (W \subseteq \free(X)) .
\end{displaymath} Moreover, if $Y$ is another set, then considering the disjoint union $X \sqcup Y$ along with the natural injections $\iota_{X} \colon X \to \free (X \sqcup Y)$ and $\iota_{Y} \colon Y \to \free (X \sqcup Y)$, we define \begin{displaymath}
	u \otimes v \, \defeq \, u(\iota_{X})v(\iota_{Y}) \, \in \, \free(X \sqcup Y) 
\end{displaymath} for any $u \in \free (X)$ and $v \in \free(Y)$ and \begin{displaymath}
	U \otimes V \, \defeq \, \{ u \otimes v \mid u \in U, \, v \in V \} \, \subseteq \, \free(X \sqcup Y)
\end{displaymath} for any $U \subseteq \free (X)$ and $V \subseteq \free(Y)$. Finally, given $w \in \free (X)$ and $W \subseteq \free(X)$, we put \begin{displaymath}
	w^{\otimes 1} \, \defeq \, w, \qquad \, W^{\otimes 1} \, \defeq \, W
\end{displaymath} and \begin{displaymath}
	w^{\otimes (n+1)} \, \defeq \, w^{\otimes n} \otimes w , \qquad \, W^{\otimes (n+1)} \, \defeq \, W^{\otimes n} \otimes W
\end{displaymath} for any $n \in \N_{>0}$. Let us record the following elementary observations in preparation of the proof of Theorem~\ref{theorem:verbally.elliptic}.

\begin{remark}\label{remark:word.images} Let $X$ be a set, let $W \subseteq \free(X)$, and let $G$ be a group. \begin{enumerate}
	\item\label{remark:word.images.1} $W^{\otimes n}(G) = W(G)^{n}$ for every $n \in \N_{>0}$.
	\item\label{remark:word.images.2} Let $H$ be a group and let $\pi \colon H \to G$ be a homomorphism. Then \begin{displaymath}
		\qquad \pi(W(H)) \, = \, W(\pi(H)) \, \subseteq \, W(G)
	\end{displaymath} and thus \begin{displaymath}
		\qquad \Cl_{G}(\pi(W(H))) \, \subseteq \, \Cl_{G}(W(G)) \, = \, W(G) .
	\end{displaymath} Hence, if $H = W(H)$ and $G = \Cl_{G}(\pi(H))^{n}$ for some $n \in \N$, then~$G = W(G)^{n}$.
\end{enumerate} \end{remark}

We now turn to the study of universal sequences. A \emph{universal sequence}~\cite{MitchellPeresse1} for a group $G$ is an element $w \in \free(m)^{\N}$ such that $m \in \N$ and \begin{displaymath}
	\forall g \in G^{\N} \ \exists h \in G^{m} \ \forall n \in \N \colon \qquad g_{n} \, = \, w_{n}(h) .
\end{displaymath} Universal sequences for groups have been studied extensively in the literature: the interested reader may consult~\cite{MitchellPeresse1,MitchellPeresse2,HydeJonusasMitchellPeresse} and the references therein for a discussion of examples. As noted in~\cite{HydeJonusasMitchellPeresse}, the work of Galvin~\cite{galvin} provides a concrete universal sequence for the symmetric group $\Sym(X)$ of any infinite set $X$. The subsequent Lemma~\ref{lemma:galvin}, a central ingredient in the proof of our Theorem~\ref{theorem:universal.sequence} and the main result of this section, is deeply inspired by the argument in~\cite[Proof of Theorem~3.3]{galvin}.

Let us agree on some convenient notation. Here and throughout the entire article, the set of \emph{involutions} of a group $G$ will be denoted by \begin{displaymath}
	\I(G) \, \defeq \, \left\{ g \in G \left\vert \, g^{2} = 1 \right\} . \right.
\end{displaymath} Furthermore, let us make the following definition, which is relevant only in Lemma~\ref{lemma:galvin}, Lemma~\ref{lemma:farrell.wagoner}, Lemma~\ref{lemma:essential.embedding}, and the proof of Lemma~\ref{lemma:generation.by.moietious.involutions}.

\begin{definition}\label{definition:permutation.groups} Let $G$ be a group and let $X$ be a set. Then we consider the group \begin{displaymath}
	G(X) \, \defeq \, G^{X}\! \rtimes \Sym(X) \, = \, \left\{ (g_{\bullet},g_{\circ}) \left\vert \, g_{\bullet} \in G^{X}\!, \, g_{\circ} \in \Sym(X) \right\} \right.
\end{displaymath} and, for every $Y \subseteq X$, its subgroup\footnote{For any $h \in \Sym(X)$, we define $\Fix(h) \defeq \{ x \in X \mid h(x) = x \}$.} \begin{displaymath}
	G(X,Y) \, \defeq \, \{ g \in G(X) \mid  g_{\bullet}\vert_{X\setminus Y} = 1, \, X \setminus Y \subseteq \Fix(g_{\circ}) \} \, \cong \, G(Y) .
\end{displaymath} Moreover, for any $u \in \I(\Sym(X))$, let \begin{displaymath}
	G(u) \, \defeq \, \left\{ g \in G(X) \left\vert \, g_{\circ} = 1, \, {g_{\bullet}\vert_{\Fix(u)}} = 1, \, {{g_{\bullet}} \circ u} = {g_{\bullet}^{-1}} \right\} . \right.
\end{displaymath} \end{definition}

\begin{lem}[cf.~{\cite[Proof of Theorem~3.3]{galvin}}]\label{lemma:galvin} Let $G$ be a group, let $X$ be an infinite set, let\footnote{Note that $S$ is a subgroup of $G(X)$ isomorphic to $\Sym(X)$.} $S \defeq \{ g \in G(X) \mid g_{\bullet} = 1 \}$. For every $i \in \N$, consider the two-letter words \begin{align*}
	u_{i} \, &\defeq \, a^{-1}ba^{-(2i+3)}b^{-1}a^{2i+3}b^{-1}a^{-(2i+3)}ba^{2i+4}\! \, \in \, \free(\{ a,b \}), \\
	v_{i} \, &\defeq \, u_{i}\!\left(a,b^{2}\right) \, = \, a^{-1}b^{2}a^{-(2i+3)}b^{-2}a^{2i+3}b^{-2}a^{-(2i+3)}b^{2}a^{2i+4}\! \, \in \, \free(\{ a,b \}) .
\end{align*} Let $x_{0} \in X$. Then the following hold. \begin{enumerate}
	\item\label{lemma:galvin.1} There exists $s \in S$ such that, for every $g \in G(X,\{ x_{0}\})^{\N}$, there exists some $t \in G(X,X\setminus \{x_{0}\})$ such that \begin{displaymath}
			\qquad \forall t' \in G(X,\{ x_{0}\}) \ \forall i \in \N \colon \quad u_{i}(s,tt') = g_{i} .
		\end{displaymath}
	\item\label{lemma:galvin.2} Let $H$ be a group, let $\pi \colon G(X) \to H$ be a homomorphism, and let $p \in \I(H)$ and $n \in \N$ be such that \begin{displaymath}
			\qquad H \, = \, (\pi(G(X,\{ x_{0}\}))p\pi(G(X,\{ x_{0}\}))p)^{n}
		\end{displaymath} and $p\pi(G(X,X\setminus \{x_{0}\}))p \subseteq \pi(G(X,\{ x_{0}\}))$. Then\footnote{where $\prod_{j=0}^{-1} \gamma_{j} \defeq 1$ and $\prod_{j=0}^{m} \gamma_{j} \defeq (\prod_{j=0}^{m-1}\gamma_{j})\gamma_{m}$ for $m \in \N$ and $\gamma \in \Gamma^{\N}$ over any group $\Gamma$} \begin{displaymath}
			\qquad w_{i} \, \defeq \, \prod\nolimits_{j=0}^{n-1} v_{2ni+2j}b^{-1}v_{2ni+2j+1}b \, \in \, \free(\{ a,b \}) \, \cong \, \free(2) \qquad (i \in \N)
		\end{displaymath} is a universal sequence for $H$.
\end{enumerate}\end{lem}

\begin{proof} \ref{lemma:galvin.1} Since $X$ is infinite, we may and will assume without loss of generality that $x_{0} = (-1,0) \in \Z \times \Z \subseteq X$. Then $s \defeq (1,s_{\circ}) \in S$ for \begin{displaymath}
	s_{\circ} \colon \, X \, \longrightarrow \, X, \quad x \, \longmapsto \, \begin{cases}
			\, (y+1,z) & \text{if } x = (y,z) \in \Z \times \Z, \\
			\, x & \text{otherwise}.
		\end{cases}
\end{displaymath} Consider any sequence $g \in G(X,\{ x_{0}\})^{\N}$. Then we obtain a well-defined element $t \defeq (t_{\bullet},t_{\circ}) \in G(X,X\setminus \{x_{0}\})$ by setting \begin{displaymath}
	t_{\bullet} \colon \, X \, \longrightarrow \, G, \quad x \, \longmapsto \, \begin{cases}
			\, (g_{n})_{\bullet}(-1,0) & \text{if } x = (2n+3,z) \text{ for } n \in \N \text{ and } z \in \N, \\
			\, 1 & \text{otherwise}
		\end{cases}
\end{displaymath} and \begin{displaymath}
	t_{\circ} \colon \, X \, \longrightarrow \, X, \quad x \, \longmapsto \, \begin{cases}
			\, (0,z+1) & \text{if } x = (0,z) \in \{ 0 \} \times \Z, \\
			\, x & \text{otherwise}.
		\end{cases}
\end{displaymath} Consider any $t' = (t'_{\bullet},1) \in G(X,\{ x_{0}\})$. Observe that \begin{equation}\label{eq:fixator}
	\Fix((tt')_{\circ}) \, = \, \Fix(t_{\circ}) \, = \, X \setminus (\{ 0 \} \times \Z) 
\end{equation} and \begin{equation}\label{eq:support}
	\{ x \in X \mid (tt')_{\bullet}(x) \ne 1 \} \, \subseteq \, (2\Z + 1) \times \Z .
\end{equation} Now, let $i \in \N$ and consider \begin{align*}
	h \, &\defeq \, u_{i}(s,tt') \, = \, s^{-1}tt's^{-(2i+3)}(tt')^{-1}s^{2i+3}(tt')^{-1}s^{-(2i+3)}tt's^{2i+4} \\
		& = \, s^{-1}tt's^{-(2i+3)}(t')^{-1}t^{-1}s^{2i+3}(t')^{-1}t^{-1}s^{-(2i+3)}tt's^{2i+4} .
\end{align*} We note that $h_{\bullet}(x) = 1 = (g_{i})_{\bullet}(x)$ and $h_{\circ}(x) = x = (g_{i})_{\circ}(x)$ for all $x \in X \setminus (\Z \times \Z)$. Therefore, suppose that $x = (y,z) \in \Z \times \Z$. If $y = -1$, then \begin{align*}
	h_{\bullet}(x) \, &= \, \left(s^{-1}tt's^{-(2i+3)}(t')^{-1}t^{-1}s^{2i+3}(t')^{-1}t^{-1}s^{-(2i+3)}tt's^{2i+4}\right)\!_{\bullet}(-1,z) \\
		&= \, \left(tt's^{-(2i+3)}(t')^{-1}t^{-1}s^{2i+3}(t')^{-1}t^{-1}s^{-(2i+3)}tt's^{2i+4}\right)\!_{\bullet}(0,z) \\
		&= \, \left(t's^{-(2i+3)}(t')^{-1}t^{-1}s^{2i+3}(t')^{-1}t^{-1}s^{-(2i+3)}tt's^{2i+4}\right)\!_{\bullet}(0,z-1) \\
		&= \, \left(s^{-(2i+3)}(t')^{-1}t^{-1}s^{2i+3}(t')^{-1}t^{-1}s^{-(2i+3)}tt's^{2i+4}\right)\!_{\bullet}(0,z-1) \\
		&= \, \left((t')^{-1}t^{-1}s^{2i+3}(t')^{-1}t^{-1}s^{-(2i+3)}tt's^{2i+4}\right)\!_{\bullet}(2i+3,z-1) \\
		&= \, \left(t^{-1}s^{2i+3}(t')^{-1}t^{-1}s^{-(2i+3)}tt's^{2i+4}\right)\!_{\bullet}(2i+3,z-1) \\
		&= \, t_{\bullet}(2i+3,z-1)^{-1} \left(s^{2i+3}(t')^{-1}t^{-1}s^{-(2i+3)}tt's^{2i+4}\right)\!_{\bullet}(2i+3,z-1) \\
		&= \, t_{\bullet}(2i+3,z-1)^{-1} \left((t')^{-1}t^{-1}s^{-(2i+3)}tt's^{2i+4}\right)\!_{\bullet}(0,z-1) \\
		&= \, t_{\bullet}(2i+3,z-1)^{-1} \left(t^{-1}s^{-(2i+3)}tt's^{2i+4}\right)\!_{\bullet}(0,z-1) \\
		&= \, t_{\bullet}(2i+3,z-1)^{-1} \left(s^{-(2i+3)}tt's^{2i+4}\right)\!_{\bullet}(0,z) \\
		&= \, t_{\bullet}(2i+3,z-1)^{-1} \left(tt's^{2i+4}\right)\!_{\bullet}(2i+3,z) \\
		&= \, t_{\bullet}(2i+3,z-1)^{-1} t_{\bullet}(2i+3,z) \left(t's^{2i+4}\right)\!_{\bullet}(2i+3,z) \\
		&= \, t_{\bullet}(2i+3,z-1)^{-1} t_{\bullet}(2i+3,z) \left(s^{2i+4}\right)\!_{\bullet}(2i+3,z) \\
		&= \, t_{\bullet}(2i+3,z-1)^{-1} t_{\bullet}(2i+3,z) \, = \, \begin{cases}
			\, (g_{i})_{\bullet}(-1,0) & \text{if } z = 0, \\
			\, 1 & \text{otherwise}
		\end{cases} \\
		&= \, (g_{i})_{\bullet}(x) ;
\end{align*} if $y = -(2i+4)$, then \begin{align*}
	h_{\bullet}(x) \, &= \, \left(s^{-1}tt's^{-(2i+3)}(t')^{-1}t^{-1}s^{2i+3}(t')^{-1}t^{-1}s^{-(2i+3)}tt's^{2i+4}\right)\!_{\bullet}(-(2i+4),z) \\
		&= \, \left(tt's^{-(2i+3)}(t')^{-1}t^{-1}s^{2i+3}(t')^{-1}t^{-1}s^{-(2i+3)}tt's^{2i+4}\right)\!_{\bullet}(-(2i+3),z) \\
		&= \, \left(t's^{-(2i+3)}(t')^{-1}t^{-1}s^{2i+3}(t')^{-1}t^{-1}s^{-(2i+3)}tt's^{2i+4}\right)\!_{\bullet}(-(2i+3),z) \\
		&\stackrel{i \geq 0}{=} \, \left(s^{-(2i+3)}(t')^{-1}t^{-1}s^{2i+3}(t')^{-1}t^{-1}s^{-(2i+3)}tt's^{2i+4}\right)\!_{\bullet}(-(2i+3),z) \\
		&= \, \left((t')^{-1}t^{-1}s^{2i+3}(t')^{-1}t^{-1}s^{-(2i+3)}tt's^{2i+4}\right)\!_{\bullet}(0,z) \\
		&= \, \left(t^{-1}s^{2i+3}(t')^{-1}t^{-1}s^{-(2i+3)}tt's^{2i+4}\right)\!_{\bullet}(0,z) \\
		&= \, \left(s^{2i+3}(t')^{-1}t^{-1}s^{-(2i+3)}tt's^{2i+4}\right)\!_{\bullet}(0,z+1) \\
		&= \, \left((t')^{-1}t^{-1}s^{-(2i+3)}tt's^{2i+4}\right)\!_{\bullet}(-(2i+3),z+1) \\
		&\stackrel{i \geq 0}{=} \, \left(t^{-1}s^{-(2i+3)}tt's^{2i+4}\right)\!_{\bullet}(-(2i+3),z+1) \\
		&= \, \left(s^{-(2i+3)}tt's^{2i+4}\right)\!_{\bullet}(-(2i+3),z+1) \, = \, \left(tt's^{2i+3}\right)\!_{\bullet}(0,z+1) \\
		&= \, \left(t's^{2i+3}\right)\!_{\bullet}(0,z) \, = \, \left(s^{2i+3}\right)\!_{\bullet}(0,z) \, = \, 1 \, = \, (g_{i})_{\bullet}(x) ;
\end{align*} and if $y \notin \{ -1, -(2i+4) \}$, then \begin{align*}
	h_{\bullet}(x) \, &= \, \left(s^{-1}tt's^{-(2i+3)}(tt')^{-1}s^{2i+3}(tt')^{-1}s^{-(2i+3)}tt's^{2i+4}\right)\!_{\bullet}(y,z) \\
		&= \, \left(tt's^{-(2i+3)}(tt')^{-1}s^{2i+3}(tt')^{-1}s^{-(2i+3)}tt's^{2i+4}\right)\!_{\bullet}(y+1,z) \\
		&\stackrel{\eqref{eq:fixator}}{=} \, (tt')_{\bullet}(y+1,z) \left(s^{-(2i+3)}(tt')^{-1}s^{2i+3}(tt')^{-1}s^{-(2i+3)}tt's^{2i+4}\right)\!_{\bullet}(y+1,z) \\
		&= \, (tt')_{\bullet}(y+1,z) \left((tt')^{-1}s^{2i+3}(tt')^{-1}s^{-(2i+3)}tt's^{2i+4}\right)\!_{\bullet}(y+2i+4,z) \\
		&\stackrel{\eqref{eq:fixator}}{=} \, (tt')_{\bullet}(y+1,z) (tt')_{\bullet}(y+2i+4,z)^{-1} \\
		& \hspace{20mm} \cdot \left(s^{2i+3}(tt')^{-1}s^{-(2i+3)}tt's^{2i+4}\right)\!_{\bullet}(y+2i+4,z) \\
		&= \, (tt')_{\bullet}(y+1,z) (tt')_{\bullet}(y+2i+4,z)^{-1} \left((tt')^{-1}s^{-(2i+3)}tt's^{2i+4}\right)\!_{\bullet}(y+1,z) \\
		&\stackrel{\eqref{eq:fixator}}{=} \, (tt')_{\bullet}(y+1,z) (tt')_{\bullet}(y+2i+4,z)^{-1} (tt')_{\bullet}(y+1,z)^{-1} \\
		& \hspace{20mm} \cdot  \left(s^{-(2i+3)}tt's^{2i+4}\right)\!_{\bullet}(y+1,z) \\
		&= \, (tt')_{\bullet}(y+1,z) (tt')_{\bullet}(y+2i+4,z)^{-1} (tt')_{\bullet}(y+1,z)^{-1} \\
		& \hspace{20mm} \cdot \left(tt's^{2i+4}\right)\!_{\bullet}(y+2i+4,z) \\
		&\stackrel{\eqref{eq:fixator}}{=} \, (tt')_{\bullet}(y+1,z) (tt')_{\bullet}(y+2i+4,z)^{-1} (tt')_{\bullet}(y+1,z)^{-1} (tt')_{\bullet}(y+2i+4,z) \\
		&\stackrel{\eqref{eq:support}}{=} \, \begin{cases}
			\, (tt')_{\bullet}(y+1,z)(tt')_{\bullet}(y+1,z)^{-1} & \text{if } y \in 2\Z, \\
			\, (tt')_{\bullet}(y+2i+4,z)^{-1}(tt')_{\bullet}(y+2i+4,z) & \text{otherwise}
		\end{cases} \\
		&= \, 1 \, = \, (g_{i})_{\bullet}(x) .
\end{align*} That is, $h_{\bullet}(x) = (g_{i})_{\bullet}(x)$. Furthermore, if $y = -1$, then \begin{align*}
	h_{\circ}(x) \, &= \, \left( s_{\circ}^{-1}t_{\circ}s_{\circ}^{-(2i+3)}t_{\circ}^{-1}s_{\circ}^{2i+3}t_{\circ}^{-1}s_{\circ}^{-(2i+3)}t_{\circ}s_{\circ}^{2i+4} \right)\!(-1,z) \\
		& = \, \left( s_{\circ}^{-1}t_{\circ}s_{\circ}^{-(2i+3)}t_{\circ}^{-1}s_{\circ}^{2i+3}t_{\circ}^{-1}s_{\circ}^{-(2i+3)}t_{\circ} \right)\!(2i+3,z) \\
		&= \, \left( s_{\circ}^{-1}t_{\circ}s_{\circ}^{-(2i+3)}t_{\circ}^{-1}s_{\circ}^{2i+3}t_{\circ}^{-1}s_{\circ}^{-(2i+3)} \right)\!(2i+3,z) \\
		&= \, \left( s_{\circ}^{-1}t_{\circ}s_{\circ}^{-(2i+3)}t_{\circ}^{-1}s_{\circ}^{2i+3}t_{\circ}^{-1} \right)\!(0,z) \, = \, \left( s_{\circ}^{-1}t_{\circ}s_{\circ}^{-(2i+3)}t_{\circ}^{-1}s_{\circ}^{2i+3} \right)\!(0,z-1) \\
		& = \, \left( s_{\circ}^{-1}t_{\circ}s_{\circ}^{-(2i+3)}t_{\circ}^{-1} \right)\!(2i+3,z-1) \, = \, \left( s_{\circ}^{-1}t_{\circ}s_{\circ}^{-(2i+3)} \right)\!(2i+3,z-1) \\
		& = \, \left( s_{\circ}^{-1}t_{\circ} \right)\!(0,z-1) \, = \, s_{\circ}^{-1}(0,z) \, = \, (-1,z) \, = \, x \, = \, (g_{i})_{\circ}(x);
\end{align*} if $y = -(2i+4)$, then \begin{align*}
	h_{\circ}(x) \, &= \, \left( s_{\circ}^{-1}t_{\circ}s_{\circ}^{-(2i+3)}t_{\circ}^{-1}s_{\circ}^{2i+3}t_{\circ}^{-1}s_{\circ}^{-(2i+3)}t_{\circ}s_{\circ}^{2i+4} \right)\!(y,z) \\
		&= \, \left( s_{\circ}^{-1}t_{\circ}s_{\circ}^{-(2i+3)}t_{\circ}^{-1}s_{\circ}^{2i+3}t_{\circ}^{-1}s_{\circ}^{-(2i+3)}t_{\circ} \right)\!(0,z) \\
		&= \, \left( s_{\circ}^{-1}t_{\circ}s_{\circ}^{-(2i+3)}t_{\circ}^{-1}s_{\circ}^{2i+3}t_{\circ}^{-1}s_{\circ}^{-(2i+3)} \right)\!(0,z+1) \\
		&= \, \left( s_{\circ}^{-1}t_{\circ}s_{\circ}^{-(2i+3)}t_{\circ}^{-1}s_{\circ}^{2i+3}t_{\circ}^{-1} \right)\!(-(2i+3),z+1) \\
		&= \, \left( s_{\circ}^{-1}t_{\circ}s_{\circ}^{-(2i+3)}t_{\circ}^{-1}s_{\circ}^{2i+3} \right)\!(-(2i+3),z+1) \\
		&= \, \left( s_{\circ}^{-1}t_{\circ}s_{\circ}^{-(2i+3)}t_{\circ}^{-1} \right)\!(0,z+1) \, = \, \left( s_{\circ}^{-1}t_{\circ}s_{\circ}^{-(2i+3)} \right)\!(0,z) \\
		&= \, \left( s_{\circ}^{-1}t_{\circ} \right)\!(-(2i+3),z) \, = \, s_{\circ}^{-1}(-(2i+3),z) \, = \, (y,z) \, = \, x \, = \, (g_{i})_{\circ}(x) ;
\end{align*} and if $y \notin \{ -1, -(2i+4) \}$, then \begin{align*}
	h_{\circ}(x) \, &= \, \left( s_{\circ}^{-1}t_{\circ}s_{\circ}^{-(2i+3)}t_{\circ}^{-1}s_{\circ}^{2i+3}t_{\circ}^{-1}s_{\circ}^{-(2i+3)}t_{\circ}s_{\circ}^{2i+4} \right)\!(y,z) \\
		&= \, \left( s_{\circ}^{-1}t_{\circ}s_{\circ}^{-(2i+3)}t_{\circ}^{-1}s_{\circ}^{2i+3}t_{\circ}^{-1}s_{\circ}^{-(2i+3)}t_{\circ} \right)\!(y+2i+4,z) \\
		&= \, \left( s_{\circ}^{-1}t_{\circ}s_{\circ}^{-(2i+3)}t_{\circ}^{-1}s_{\circ}^{2i+3}t_{\circ}^{-1}s_{\circ}^{-(2i+3)} \right)\!(y+2i+4,z) \\
		&= \, \left( s_{\circ}^{-1}t_{\circ}s_{\circ}^{-(2i+3)}t_{\circ}^{-1}s_{\circ}^{2i+3}t_{\circ}^{-1} \right)\!(y+1,z) \\
		&= \, \left( s_{\circ}^{-1}t_{\circ}s_{\circ}^{-(2i+3)}t_{\circ}^{-1}s_{\circ}^{2i+3} \right)\!(y+1,z) \, = \, \left( s_{\circ}^{-1}t_{\circ}s_{\circ}^{-(2i+3)}t_{\circ}^{-1} \right)\!(y+2i+4,z) \\
		& = \, \left( s_{\circ}^{-1}t_{\circ}s_{\circ}^{-(2i+3)} \right)\!(y+2i+4,z) \, = \, \left( s_{\circ}^{-1}t_{\circ} \right)\!(y+1,z) \, = \, s_{\circ}^{-1}(y+1,z) \\
		& = \, (y,z) \, = \, x \, = \, (g_{i})_{\circ}(x) .
\end{align*} So, $h_{\circ}(x) = (g_{i})_{\circ}(x)$. This shows that $h = (h_{\bullet},h_{\circ}) = ((g_{i})_{\bullet},(g_{i})_{\circ}) = g_{i}$, as intended.
	
\ref{lemma:galvin.2} Let $h \in H^{\N}$. Since $H = (\pi(G(X,\{ x_{0}\}))p\pi(G(X,\{ x_{0}\}))p)^{n}$ by our hypothesis, we find $g \in G(X,\{ x_{0}\})^{\N}$ such that $h_{i} = \prod_{j=0}^{n-1} \pi(g_{2ni+2j})p\pi(g_{2ni+2j+1})p$ for all $i \in \N$. According to~\ref{lemma:galvin.1}, there exist $s \in S$ and $t \in G(X,X\setminus \{x_{0}\})$ such that \begin{displaymath}
	\forall t' \in G(X,\{ x_{0}\}) \ \forall i \in \N \colon \qquad u_{i}(s,tt') \, = \, g_{i} .
\end{displaymath} As $p\pi(G(X,X\setminus \{x_{0}\}))p \subseteq \pi(G(X,\{ x_{0}\}))$ by assumption, there exists $t' \in G(X,\{ x_{0}\})$ such that $p\pi(t)p = \pi(t')$, and we conclude that \begin{align*}
	\pi(g_{i}) \, &= \, \pi(u_{i}(s,tt')) \, = \, u_{i}(\pi(s),\pi(tt')) \, = \, u_{i}(\pi(s),\pi(t)\pi(t')) \\
		& = \, u_{i}(\pi(s),\pi(t)p\pi(t)p) \, = \, u_{i}\!\left(\pi(s),(\pi(t)p)^{2}\right) \, = \, v_{i}(\pi(s),\pi(t)p) 
\end{align*} for all $i \in \N$. Moreover, as every element of $G(X,X\setminus \{x_{0}\})$ commutes with every element of $G(X,\{ x_{0}\})$, we see that \begin{displaymath}
	(\pi(t)p)^{-1}\pi(g_{i})\pi(t)p \, = \, p\pi(t)^{-1}\pi(g_{i})\pi(t)p \, = \, p\pi\!\left( t^{-1}g_{i}t\right)p \, = \, p\pi(g_{i})p
\end{displaymath} whenever $i \in \N$. Consequently, for every $i \in \N$, \begin{align*}
	h_{i} \, &= \, \prod\nolimits_{j=0}^{n-1} \pi(g_{2ni+2j})p\pi(g_{2ni+2j+1})p \\
		&= \, \prod\nolimits_{j=0}^{n-1} \pi(g_{2ni+2j})(\pi(t)p)^{-1}\pi(g_{2ni+2j+1})\pi(t)p \\
		&= \, \prod\nolimits_{j=0}^{n-1} v_{2ni+2j}(\pi(s),\pi(t)p)(\pi(t)p)^{-1}v_{2ni+2j+1}(\pi(s),\pi(t)p)\pi(t)p \\
		&= \, w_{i}\!\left( \pi(s), \pi(t)p \right) .
\end{align*} This shows that $(w_{i})_{i \in \N}$ is universal for $H$. \end{proof}

For later use, in the proof of Lemma~\ref{lemma:generation.by.moietious.involutions}, we record a basic abstract variant of a well-known decomposition trick (see, e.g.,~\cite[Proof of Lemma~1.13]{FarrellWagoner},~\cite[Proof of Theorem~2.7]{MenalMoncasi}, and~\cite[Proof of Theorem~1.1]{tolstykh}).

\begin{lem}\label{lemma:farrell.wagoner} Let $G$ be a group and let $X$ be an infinite set. For every $x \in X$, there exist $u,v \in \I(\Sym(X))$ such that $G(X,\{x\}) \subseteq G(u)G(v)$. \end{lem}

\begin{proof} Let $x \in X$. As $X$ is infinite, we find an injection $f \colon \N \to X$ with $f(0) = x$. Define \begin{displaymath}
	u \colon \, X \, \longrightarrow \, X, \quad y \, \longmapsto \, \begin{cases}
																		\, f\!\left(f^{-1}(y)+1\right) & \text{if } y \in f(2\N), \\
																		\, f\!\left(f^{-1}(y)-1\right) & \text{if } y \in f(2\N+1), \\
																		\, y & \text{otherwise}
																	\end{cases}
\end{displaymath} and \begin{displaymath}
	v \colon \, X \, \longrightarrow \, X, \quad y \, \longmapsto \, \begin{cases}
																		\, f\!\left(f^{-1}(y)+1\right) & \text{if } y \in f(2\N+1), \\
																		\, f\!\left(f^{-1}(y)-1\right) & \text{if } y \in f(2\N+2), \\
																		\, y & \text{otherwise}
																	\end{cases}
\end{displaymath} and observe that $u,v \in \I(\Sym(X))$. Moreover, if $g \in G(X,\{x\})$, then considering $s \defeq (s_{\bullet},1) \in G(u)$ and $t \defeq (t_{\bullet},1) \in G(v)$ given by \begin{displaymath}
	s_{\bullet} \colon \, X \, \longrightarrow \, G, \quad y \, \longmapsto \, \begin{cases}
		\, g_{\bullet}(x) & \text{if } y \in f(2\N), \\
		\, g_{\bullet}(x)^{-1} & \text{if } y \in f(2\N+1), \\
		\, 1 & \text{otherwise}
	\end{cases}
\end{displaymath} and \begin{displaymath}
	t_{\bullet} \colon \, X \, \longrightarrow \, G, \quad y \, \longmapsto \, \begin{cases}
		\, g_{\bullet}(x) & \text{if } y \in f(2\N+1), \\
		\, g_{\bullet}(x)^{-1} & \text{if } y \in f(2\N+2), \\
		\, 1 & \text{otherwise},
	\end{cases}
\end{displaymath} one readily checks that $g = st$. Hence, $G(X,\{x\}) \subseteq G(u)G(v)$. \end{proof}

\section{Strong distortion and cofinality}\label{section:strong.distortion.and.cofinality}

This section revolves around several rigidity and boundedness phenomena in infinite groups. Let $G$ be a group. Then \begin{itemize}
	\item[---\,] $G$ is said to have the \emph{Schreier property}~\cite{LeRouxMann} if every countable subset of $G$ is contained in a finitely generated subgroup of $G$,
	\item[---\,] the \emph{Sierpi\'nski rank}~\cite{MitchellPeresse1,MitchellPeresse2} of $G$ is the least $m \in \N \cup \{ \infty \}$ such that every countable subset of $G$ is contained in an $m$-generated subgroup of $G$,
	\item[---\,] $G$ is called \emph{strongly distorted}~\cite[Remark~A.3]{CalegariFreedman}\footnote{This is called \emph{property $P$} in~\cite{khelif}.} if there exist $m \in \N$ and $w \in \N^{\N}$ such that, for every $g \in G^{\N}$, there exists a subset $S \subseteq G$ such that $\vert S \vert \leq m$ and \begin{displaymath}
		\qquad \forall n \in \N \colon \qquad g_{n} \, \in \, S^{w_{n}} .
	\end{displaymath}
\end{itemize}

\begin{remark}[{\cite[Section~1]{MitchellPeresse1}}, {\cite[Section~1]{LeRouxMann}}]\label{remark:strong.distortion} \begin{enumerate}
	\item\label{remark:strong.distortion.1} Every group admitting a universal sequence is strongly distorted.
	\item\label{remark:strong.distortion.2} Every strongly distorted group has finite Sierpi\'nski rank.
	\item\label{remark:strong.distortion.3} Every group of finite Sierpi\'nski rank has the Schreier property.
\end{enumerate}\end{remark}

The following definitions have been taken from the work of Droste and Goebel~\cite{DrosteGoebel}, but slightly extended to formally cover the trivial cases of finite (resp., finitely generated) groups.

\begin{definition}[cf.~{\cite[Definition~2.1]{DrosteGoebel}}] Let $G$ be a group. \begin{enumerate}
	\item The \emph{cofinality} of $G$, denoted by $\cf(G)$, is defined as follows: if $G$ is finitely generated, then $\cf(G) \defeq 0$; and if $G$ is not finitely generated, then $\cf(G)$ is the least cardinality of a chain $\mathcal{C}$ of proper subgroups of $G$ such that $G = \bigcup \mathcal{C}$.
	\item The \emph{strong cofinality} of $G$, denoted by $\scf(G)$, is defined as follows: if $G$ is finite, then $\scf(G) \defeq 0$; and if $G$ is infinite, then $\scf(G)$ is the smallest cardinality of a chain $\mathcal{C}$ of proper subsets of $G$ such that \begin{itemize}
		\item[---\,] $G = \bigcup \mathcal{C}$,
		\item[---\,] $\forall C \in \mathcal{C} \colon \, C^{-1} = C$,
		\item[---\,] $\forall C \in \mathcal{C} \, \exists C' \in \mathcal{C} \colon \, C^{2} \subseteq C'$.
	\end{itemize}
	\item We say that $G$ has the \emph{Bergman property} if, for every generating set $S$ of $G$, there exists $n \in \N$ such that $G = \left(S \cup S^{-1} \cup \{ 1 \}\right)^{n}$.
\end{enumerate} \end{definition}

\begin{remark}\label{remark:strong.distortion.implies.uscf} Let $G$ be a group. \begin{enumerate}
	\item\label{remark:strong.distortion.implies.uscf.-1} $\scf(G) \ne \aleph_{0}$ if and only if $\cf(G) \ne \aleph_{0}$ and $G$ has the Bergman property. This is established in~\cite[Theorem~2.2]{DrosteHolland} and~\cite[Proofs of Corollaries~2.1, 2.2]{DrosteHolland} (see also~\cite[Proposition~2.2(i)]{DrosteGoebel}).
	\item\label{remark:strong.distortion.implies.uscf.0} $\scf(G) = \aleph_{0}$ or $\scf(G) = \cf(G)$, due to~\cite[Proposition~2.2(ii)]{DrosteGoebel}.
	\item\label{remark:strong.distortion.implies.uscf.1} If $G$ has strong distortion, then $\scf(G) \ne \aleph_{0}$. This is independently due to Khelif~\cite{khelif} and Cornulier~\cite[Remark~A.3]{CalegariFreedman} (see also~\cite[Proposition~4.E.13]{CornulierDeLaHarpeBook} or~\cite[Section~1]{LeRouxMann}).
	\item\label{remark:strong.distortion.implies.uscf.2} If $G$ has the Schreier property, then $\cf(G) \ne \aleph_{0}$. This has been established in~\cite[Section~1]{LeRouxMann} (see also~\cite[Proposition~4.E.9]{CornulierDeLaHarpeBook}).
\end{enumerate} \end{remark}

The remainder of this section is dedicated to a specific rigidity phenomenon resulting from uncountable strong cofinality and concerning homomorphisms taking values in Banach--Lie groups (Proposition~\ref{propostion:automatic.SIN.reps}). Following the notation of~\cite{MorrisPestov}, if $G$ is a group, then we let \begin{displaymath}
	\tfrac{1}{n}U \, \defeq \, \{ g \in G \mid g^{1},\ldots,g^{n} \in U \}
\end{displaymath} whenever $U \subseteq G$ and $n \in \N$. A topological group $G$ is said to be \emph{uniformly free from small subgroups}~\cite[\S2.1, p.~239]{enflo} if \begin{displaymath}
	\exists U \in \Neigh(G) \ \forall V \in \Neigh(G) \ \exists n \in \N \colon \qquad \tfrac{1}{n}U \, \subseteq \, V .
\end{displaymath}

\begin{remark}\label{remark:banach.lie} Every Banach--Lie group is uniformly free from small subgroups, by a result of Morris and Pestov~\cite[Theorem~2.7]{MorrisPestov}, which had been also noted without proof by Enflo~\cite[p.~241, remark after Theorem~2.1.1]{enflo}. \end{remark}

A set $\Phi$ of maps from a uniform space $X$ to another uniform space $Y$ is called \emph{uniformly equicontinuous} if, for every entourage $E$ of $Y$, there exists an entourage $F$ of $X$ such that \begin{displaymath}
	\forall \phi \in \Phi \ \forall (x,y) \in F \colon \quad (\phi(x),\phi(y)) \in E .
\end{displaymath} The reader is referred to~\cite[X, \S2.1]{bourbakitop2} for further background on this concept, while we confine ourselves to recording the following well-known basic observation for later use in the proof of Proposition~\ref{propostion:automatic.SIN.reps}.

\begin{remark}\label{remark:composition.UEC.sets} Let $X$, $Y$, $Z$ be uniform spaces. If both $\Phi \subseteq Y^{X}$ and $\Psi \subseteq Z^{Y}$ are uniformly equicontinuous, then so is $\{ \psi \circ \phi \mid \psi \in \Psi, \, \phi \in \Phi \} \subseteq Z^{X}$. \end{remark}

The set of all homomorphisms from a group $G$ to another group $H$ will be denoted by $\Hom(G,H)$. One readily checks that, for any two topological groups $G$ and $H$, a subset $\Phi \subseteq \Hom(G,H)$ is uniformly equicontinuous if and only if \begin{displaymath}
	\forall V \in \Neigh(H) \ \exists U \in \Neigh(G) \ \forall \phi \in \Phi \colon \qquad \phi(U) \, \subseteq \, V .
\end{displaymath}

\begin{lem}\label{lemma:UEC.sets.in.UNSS.groups} Let $G$ and $H$ be topological groups, let $V \in \Neigh(H)$ be such that \begin{displaymath}
	\forall W \in \Neigh(H) \ \exists n \in \N \colon \qquad \tfrac{1}{n}V \, \subseteq \, W ,
\end{displaymath} and let $U \in \Neigh(G)$. Then \begin{displaymath}
	\{ \pi \in \Hom(G,H) \mid \pi(U) \subseteq V \}
\end{displaymath} is uniformly equicontinuous. \end{lem}

\begin{proof} Let $W \in \Neigh(H)$. By our hypothesis on $V$, there exists $n \in \N$ with $\tfrac{1}{n}V \subseteq W$. Since $G$ is a topological group, we find $N \in \Neigh(G)$ such that $N^{n} \subseteq U$, which readily entails that $N \subseteq \tfrac{1}{n}U$. Now, if $\pi \in \Hom(G,H)$ and $\pi(U) \subseteq V$, then \begin{displaymath}
	\pi(N) \, \subseteq \, \pi\!\left(\tfrac{1}{n}U\right)\! \, \subseteq \, \tfrac{1}{n}\pi(U) \, \subseteq \, \tfrac{1}{n}V \, \subseteq \, W. \qedhere
\end{displaymath} \end{proof}

An action of a group $G$ on a uniform space $X$ will be called \emph{uniformly equicontinuous} if the set consisting of the corresponding maps \begin{displaymath}
	X \, \longrightarrow \, X, \quad x \, \longmapsto \, gx \qquad (g \in G)
\end{displaymath} is uniformly equicontinuous.

\begin{remark}\label{remark:UEC.action.by.conjugation} Let $G$ be a topological group, let $H \leq G$. The following are equivalent. \begin{enumerate}
	\item\label{remark:UEC.action.by.conjugation.1} The action of $H$ on $G$ by conjugation is uniformly equicontinuous.
	\item\label{remark:UEC.action.by.conjugation.2} $\forall U \in \Neigh(G) \colon \ \bigcap_{h \in H} hUh^{-1} \in \Neigh(G)$.
	\item\label{remark:UEC.action.by.conjugation.3} $\forall U \in \Neigh(G) \ \exists V \in \Neigh(G), \, V \subseteq U \ \forall h \in H \colon \ hVh^{-1} = V$.
\end{enumerate} In particular, if~\ref{remark:UEC.action.by.conjugation.1} holds, then $H$ has small invariant neighborhoods with respect to the relative topology inherited from $G$. \end{remark}

\begin{prop}\label{propostion:automatic.SIN.reps} Let $G$ be a group with $\scf(G) \ne \aleph_{0}$, let $H$ be a topological group uniformly free from small subgroups, let $\pi \colon G \to H$ be a homomorphism. Then the action of $\pi(G)$ on $H$ by conjugation is uniformly equicontinuous. In particular, $\pi(G)$ has small invariant neighborhoods with respect to the topology inherited from $H$. \end{prop}

\begin{proof} As $H$ is uniformly free from small subgroups, there exists $U \in \Neigh(H)$ such that \begin{displaymath}
	\forall V \in \Neigh(H) \ \exists n \in \N \colon \qquad \tfrac{1}{n}U \, \subseteq \, V .
\end{displaymath} Also, being uniformly free from small subgroups, $H$ must be first-countable (cf.~\cite[Theorem~2.1.1]{enflo}). Hence, we find a countable neighborhood basis $\mathcal{V}$ at the neutral element in $H$ such that, moreover, $(\mathcal{V},{\subseteq})$ is a chain. For each $V \in \mathcal{V}$, let \begin{displaymath}
	H(V,U) \, \defeq \, \{ h \in H \mid hVh^{-1} \subseteq U \} .
\end{displaymath} As $H$ is a topological group and $\mathcal{V}$ is a neighborhood basis at the neutral element of $H$, it follows that $H = \bigcup_{V \in \mathcal{V}} H(V,U)$. Since $(\mathcal{V},{\subseteq})$ is a chain, $(\{ H(V,U) \mid V \in \mathcal{V} \},{\subseteq})$ is a chain, too. Now, due to $\scf(G) \ne \aleph_{0}$, there exist $V \in \mathcal{V}$ and $m \in \N$ such that $\pi(G) \subseteq H(V,U)^{m}$. Consequently, the action of $\pi(G)$ on $H$ by conjugation is uniformly equicontinuous by Lemma~\ref{lemma:UEC.sets.in.UNSS.groups} and Remark~\ref{remark:composition.UEC.sets}. The final assertion now follows by Remark~\ref{remark:UEC.action.by.conjugation}. \end{proof}

\begin{prop}\label{proposition:rigidity.for.BL.reps} Let $G$ be a group with $\scf(G) \ne \aleph_{0}$. Let $T \subseteq G$ be such that $G$ has uniformly bounded normal generation over $TT^{-1}$. For every non-trivial homomorphism $\pi \colon G \to H$ to a Hausdorff topological group $H$ uniformly free from small subgroups, the restriction $\pi\vert_{T}$ is injective and $\pi(T)$ is discrete in $H$. \end{prop}

\begin{proof} This follows from Proposition~\ref{propostion:automatic.SIN.reps} and Lemma~\ref{lemma:rigidity.for.SIN.reps}. \end{proof}

\section{Uniform strong cofinality}\label{section:uniform.strong.cofinality}

The following cardinal invariant defined for general groups captures and measures a uniform version of uncountable strong cofinality. The latter property has been studied explicitly in~\cite{MillerThesis,KechrisRosendal,RicardRosendal,rosendal,dowerk}. Given a set $X$, we let $\Pow(X)$ denote the power set of $X$ and we let $\Pfin(X)$ denote the set of all finite subsets of $X$.

\begin{definition}\label{definition:uniform.strong.cofinality} Let $k \in \N_{>0}$. The \emph{$k$-strong cofinality} of a group $G$, denoted by $\scf_{k}(G)$, is defined as follows: if $G$ is finite, then $\scf_{k}(G) \defeq 0$; and if $G$ is infinite, then $\scf_{k}(G)$ is the smallest cardinality of a directed subset $\mathcal{D}$ of $(\Pow(G),{\subseteq})$ such that \begin{displaymath}
	\bigcup \mathcal{D} \, = \, G \, \notin \, \{ D^{k} \mid D \in \mathcal{D} \} .
\end{displaymath} \end{definition}

\begin{remark}\label{remark:uniform.strong.cofinality} Let $G$ be a group and let $k \in \N_{>0}$. \begin{enumerate}
	\item\label{remark:uniform.strong.cofinality.1} If $G$ is infinite, then $\Pfin(G)$ constitutes a directed subset of $(\Pow(G),{\subseteq})$ satisfying $\bigcup \Pfin(G) = G \notin \{ F^{k} \mid F \in \Pfin(G) \}$. Hence, $\scf_{k}(G) \leq \vert G \vert$.
	\item\label{remark:uniform.strong.cofinality.2} If $\mathcal{D}$ is any directed subset of $(\Pow(G),{\subseteq})$ with $G = \bigcup \mathcal{D}$, then \begin{displaymath}
			\qquad \mathcal{D}' \, \defeq \, \! \left. \left\{ D \cap D^{-1} \, \right\vert 1 \in D \in \mathcal{D} \right\}
		\end{displaymath} constitutes a directed subset of $(\Pow(G),{\subseteq})$ with $\vert \mathcal{D}' \vert \leq \vert \mathcal{D} \vert$ and $G = \bigcup \mathcal{D}'$. In particular, if $G$ is infinite, then $\scf_{k}(G)$ is the least cardinality of a directed subset $\mathcal{D}$ of $(\Pow(G),{\subseteq})$ such that \begin{itemize}
			\item[---\,] $\bigcup \mathcal{D} = G \notin \{ D^{k} \mid D \in \mathcal{D} \}$, and
			\item[---\,] $1 \in D = D^{-1}$ for every $D \in \mathcal{D}$.
		\end{itemize}
\end{enumerate} \end{remark}

\begin{lem}\label{lemma:uniform.strong.cofinality} Let $G$ be a group and let $k \in \N_{>0}$. The following hold. \begin{enumerate}
	\item\label{lemma:uniform.strong.cofinality.1} If $G$ is infinite, then $\scf_{k}(G) \geq \aleph_{0}$.
	\item\label{lemma:uniform.strong.cofinality.2} $\scf_{k}(G) \leq \scf_{k+1}(G)$.
	\item\label{lemma:uniform.strong.cofinality.3} $\scf_{k}(G) \leq \scf(G)$.
	\item\label{lemma:uniform.strong.cofinality.4} If $G$ is an infinite homomorphic image of a group $H$, then $\scf_{k}(H) \leq \scf_{k}(G)$.
\end{enumerate} \end{lem}

\begin{proof} \ref{lemma:uniform.strong.cofinality.1} If $\mathcal{D}$ is a finite directed subset of $(\Pow(G),{\subseteq})$ such that $\bigcup \mathcal{D} = G$, then $G = \bigcup \mathcal{D} \in \mathcal{D}$ and hence $G = G^{k} \in \{ D^{k} \mid D \in \mathcal{D} \}$.

\ref{lemma:uniform.strong.cofinality.2} If $G$ is finite, then $\scf_{k}(G) = 0 = \scf_{k+1}(G)$ by definition. Now, suppose that $G$ is infinite. If $\mathcal{D}$ is a directed subset of $(\Pow(G),{\subseteq})$ such that \begin{itemize}
	\item[---\,] $\bigcup \mathcal{D} = G \notin \{ D^{k+1} \mid D \in \mathcal{D} \}$, and
	\item[---\,] $1 \in D = D^{-1}$ for every $D \in \mathcal{D}$,
\end{itemize} then $D^{k} \subseteq D^{k+1}$ for every $D \in \mathcal{D}$ and hence $G \notin \{ D^{k} \mid D \in \mathcal{D} \}$. In view of Remark~\ref{remark:uniform.strong.cofinality}\ref{remark:uniform.strong.cofinality.2}, this shows that $\scf_{k}(G) \leq \scf_{k+1}(G)$.

\ref{lemma:uniform.strong.cofinality.3} As above, if $G$ is finite, then $\scf_{k}(G) = 0 = \scf(G)$ according to the definitions. Assume that $G$ is infinite. Let $\mathcal{C}$ be a chain of proper subsets of $G$ such that \begin{itemize}
	\item[---\,] $G = \bigcup \mathcal{C}$,
	\item[---\,] $\forall C \in \mathcal{C} \colon \, C^{-1} = C$,
	\item[---\,] $\forall C \in \mathcal{C} \, \exists C' \in \mathcal{C} \colon \, C^{2} \subseteq C'$.
\end{itemize} Moreover, upon replacing $\mathcal{C}$ by $\{ C \in \mathcal{C} \mid 1 \in C\}$, we may and will assume without loss of generality that $1 \in C$ for every $C \in \mathcal{C}$. Then $\mathcal{C}$ is a directed subset of~$(\Pow(G),{\subseteq})$. Furthermore, $G \notin \{ C^{k} \mid C \in \mathcal{C}\}$: otherwise, $G = C^{k}$ for some $C \in \mathcal{C}$, whence by assumption there exists $C' \in \mathcal{C}$ with $C^{k} \subseteq C^{2k} \subseteq C'$, and thus $G = C'$, contradicting another assumption. Hence, $\scf_{k}(G) \leq \vert \mathcal{C} \vert$. This shows that $\scf_{k}(G) \leq \scf(G)$.

\ref{lemma:uniform.strong.cofinality.4} Suppose that $G$ is infinite and let $\pi \colon H \to G$ be a surjective homomorphism from a (necessarily infinite) group $H$. If $\mathcal{D}$ is a directed subset of $(\Pow(G),{\subseteq})$ such that $\bigcup \mathcal{D} = G \notin \{ D^{k} \mid D \in \mathcal{D} \}$, then $\mathcal{D}' \defeq \{ \pi^{-1}(D) \mid D \in \mathcal{D} \}$ constitutes a directed subset of $(\Pow(H),{\subseteq})$ such that $\vert \mathcal{D}' \vert = \vert \mathcal{D} \vert$, \begin{displaymath}
	H \, = \, \pi^{-1}(G) \, = \, \pi^{-1}\!\left( \bigcup \mathcal{D}\right)\! \, = \, \bigcup \left.\! \left\{ \pi^{-1}(D) \right\vert \, D \in \mathcal{D} \right\} \, = \, \bigcup \mathcal{D}',
\end{displaymath} and \begin{displaymath}
	H \, = \, \pi^{-1}(G) \, \stackrel{G\ne D^{k}}{\supsetneq} \, \pi^{-1}\bigl( D^{k}\bigr) \, \supseteq \, \pi^{-1}(D)^{k}
\end{displaymath} for every $D \in \mathcal{D}$. Hence, $\scf_{k}(H) \leq \scf_{k}(G)$. \end{proof}

Our Section~\ref{section:general.criterion} establishes a lower bound on uniform strong cofinalities based on an algebraic interpolation method (Theorem~\ref{theorem:abstract.2}). We conclude the current section by noting an extension of an observation made by Rosendal~\cite[Proposition~1.33]{rosendal}.

\begin{prop}\label{proposition:large.uscf.rigidity} Let $G$ be a group and let $\alpha$ be an infinite cardinal such that $\scf_{k}(G) > \alpha$ for some $k \in \N_{>0}$. Let $\pi \colon G \to H$ be a homomorphism to an $\alpha$-narrow topological group $H$ \begin{itemize}
	\item[---\,] having small invariant neighborhoods, or
	\item[---\,] being uniformly free from small subgroups\footnote{e.g., a Banach--Lie group (cf.~Remark~\ref{remark:banach.lie})}.
\end{itemize} Then $\pi(G)$ is precompact. In particular, if $G$ is minimally almost periodic and $H$ is Hausdorff, then $\pi$ is trivial. \end{prop}

\begin{proof} Without loss of generality, we may and will assume that $\pi(G)$ is infinite.
	
First consider the case where $H$ has small invariant neighborhoods. Let $U \in \Neigh(H)$. Since $H$ is a topological group with small invariant neighborhoods, there exists a conjugation-invariant $V \in \Neigh(H)$ such that $V^{k} \subseteq U$. Due to $H$ being $\alpha$-narrow, there exists $T \subseteq H$ such that $H = VT$ and $\vert T \vert \leq \alpha$. Since $\{ \pi(G) \cap VF \mid F \in \Pfin(T) \}$ is a directed subset of $(\Pow(\pi(G)),{\subseteq})$ and $\scf_{k}(\pi(G)) \geq \scf_{k}(G) > \alpha$ by Lemma~\ref{lemma:uniform.strong.cofinality}\ref{lemma:uniform.strong.cofinality.4}, there exists a finite subset $F \subseteq T$ with $\pi(G) \subseteq (\pi(G) \cap VF)^{k}$. As $V$ is conjugation-invariant, we conclude that \begin{displaymath}
	\pi(G) \, \subseteq \, (\pi(G) \cap VF)^{k} \, \subseteq \, (VF)^{k} \, = \, V^{k}F^{k} \, \subseteq \, UF^{k} .
\end{displaymath} Since $F^{k}$ is finite, this shows that $\pi(G)$ is precompact by~\cite[Lemma~3.7.3]{AT} (see also~\cite[Proposition~3.4.1]{AT}).

Suppose now that $H$ is uniformly free from small subgroups. Then $\pi(G)$ has small invariant neighborhoods by Lemma~\ref{lemma:uniform.strong.cofinality}\ref{lemma:uniform.strong.cofinality.3} and Proposition~\ref{propostion:automatic.SIN.reps}. Since $\alpha$-narrowness of $H$ implies $\alpha$-narrowness of $\pi(G)$ according to~\cite[Proposition~5.1.1(a)]{AT}, the statement established in the preceding paragraph entails precompactness of $\pi(G)$.

Finally, if $H$ is Hausdorff, then $\pi(G)$ embeds into a compact Hausdorff topological group by~\cite[Corollary~3.7.17]{AT}, whence minimal almost periodicity of $G$ necessitates triviality of $\pi$. \end{proof}

\section{A general criterion}\label{section:general.criterion}

The purpose of this section is to prove Theorem~\ref{theorem:abstract.2} (resp., Theorem~\ref{theorem:abstract.1}), thus providing a feasible criterion for establishing lower bounds on uniform strong cofinality (resp., strong cofinality). The result is a natural extension of an idea going back to~\cite[Theorem~2.1]{DrosteHollandUlbrich} (see also~\cite[Theorem~2.8]{DrosteGoebel}) and its generalizations in~\cite[Theorem~3.1]{dowerk} and~\cite[Theorem~4.8]{Schneider25}. The following notion is essential.

\begin{definition}\label{definition:full} Let $G$ be a group. We define $\Sub(G)$ to be the set of all subgroups of $G$. A subset $D \subseteq G$ is said to be \emph{full} for a subgroup $H \in \Sub(G)$ if \begin{displaymath}
	\forall h,x \in H \ \exists g \in D \colon \qquad hxh^{-1} = gxg^{-1} .
\end{displaymath} Let $I$ be a set. A subset $D \subseteq G$ is said to be \emph{full} for a family $(H_{i})_{i \in I} \in \Sub(G)^{I}$ if \begin{displaymath}
	\forall h,x \in \prod\nolimits_{i \in I} H_{i} \ \exists g \in D \ \forall i \in I \colon \qquad h_{i}x_{i}h_{i}^{-1} = gx_{i}g^{-1} .
\end{displaymath} \end{definition}

Before moving on to establishing this section's main results, Theorem~\ref{theorem:abstract.2} and Theorem~\ref{theorem:abstract.1}, we note some general observations regarding persistence of fullness.

\begin{remark}\label{remark:full.persistence} Let $G$ be a group, let $D_{0} \subseteq D \subseteq G$, and let $I$ be a set. If $D_{0}$ is full for some element of $\Sub(G)^{I}$, then $D$ is full for the same. \end{remark}

\begin{lem}\label{lemma:full.homomorphic.images} Let $G$ and $H$ be groups, let $\pi \colon G \to H$ be a homomorphism, and let $I$ be a set. If a subset $D \subseteq G$ is full for a family $(G_{i})_{i \in I} \in \Sub(G)^{I}$, then $\pi(D)$ is full for $(\pi(G_{i}))_{i \in I}$. \end{lem}

\begin{proof} Suppose that $D \subseteq G$ is full for some $(G_{i})_{i \in I} \in \Sub(G)^{I}$. If $h,x \in \prod\nolimits_{i \in I} G_{i}$, then there exists $g \in D$ such that \begin{displaymath}
	\forall i \in I \colon \qquad h_{i}x_{i}h_{i}^{-1}\! \, = \, gx_{i}g^{-1} ,
\end{displaymath} whence \begin{displaymath}
	\pi(h_{i})\pi(x_{i})\pi(h_{i})^{-1}\! \, = \, \pi\!\left( h_{i}x_{i}h_{i}^{-1} \right) \, = \, \pi\!\left( gx_{i}g^{-1} \right) \, = \, \pi(g)\pi(x_{i})\pi(g)^{-1} ,
\end{displaymath} for every $i \in I$. This shows that $\pi(D)$ is full for $(\pi(G_{i}))_{i \in I}$. \end{proof}

\begin{lem}\label{lemma:direct.products.are.full} Let $(G_{i})_{i \in I}$ be a family of groups. Consider the group $G \defeq \prod_{i \in I} G_{i}$ along with its subgroups \begin{displaymath}
	H_{i} \, \defeq \, \{ g \in G \mid \forall j \in I \setminus \{ i \} \colon \, g_{j} = 1 \} \, \cong \, G_{i} \qquad (i \in I) .
\end{displaymath} Then $G$ is full for $(H_{i})_{i \in I}$. \end{lem}

\begin{proof} Let $h \in \prod\nolimits_{i \in I} H_{i}$. Consider $g \defeq (h_{i,i})_{i \in I} \in G$. Now, if $x \in \prod\nolimits_{i \in I} H_{i}$ and $i \in I$, then \begin{displaymath}
	h_{i,i}x_{i,i}h_{i,i}^{-1} \, = \, g_{i}x_{i,i}g_{i}^{-1}
\end{displaymath} and \begin{displaymath}
	h_{i,j}x_{i,j}h_{i,j}^{-1} \, = \, 1 \, = \, g_{j}g_{j}^{-1} \, = \, g_{j}x_{i,j}g_{j}^{-1}
\end{displaymath} for every $j \in I\setminus \{ i \}$, that is, $h_{i}x_{i}h_{i}^{-1} = gx_{i}g^{-1}$. \end{proof}

We now prepare for stating and proving the advertised criteria.

\begin{lem}\label{lemma:full.induction} Let $G$ be a group, let $\emptyset \ne \mathcal{D} \subseteq \Pow(G)$, and let $I$ be a set with $\vert I \vert \geq \vert \mathcal{D} \vert$. If $\bigcup \mathcal{D}$ is full for a family $(H_{i})_{i \in I} \in \Sub(G)^{I}$, then there exist $D \in \mathcal{D}$ and $i \in I$ such that $D$ is full for $H_{i}$. \end{lem}

\begin{proof} First of all, since $\vert I \vert \geq \vert \mathcal{D} \vert > 0$, we find a surjection $\Delta \colon I \to \mathcal{D}$. Now, for contradiction, suppose that no element of $\mathcal{D}$ is full for any member of $(H_{i})_{i \in I}$. Then, in particular, for every $i \in I$ the set $\Delta(i)$ is not full for $H_{i}$. That is, for each $i \in I$, there exist $h_{i},x_{i} \in H_{i}$ such that \begin{displaymath}
	\forall g \in \Delta(i) \colon \quad h_{i}x_{i}h_{i}^{-1} \ne gx_{i}g^{-1} .
\end{displaymath} On the other hand, due to $\bigcup \mathcal{D}$ being full for $(H_{i})_{i \in I}$, we find $g \in \bigcup \mathcal{D}$ such that \begin{displaymath}
	\forall i \in I \colon \quad h_{i}x_{i}h_{i}^{-1} = gx_{i}g^{-1} .
\end{displaymath} Hence, $g \notin \bigcup\nolimits_{i \in I} \Delta(i) = \bigcup \mathcal{D}$, which is the desired contradiction.\end{proof}

\enlargethispage{6mm}

\begin{thm}\label{theorem:abstract.2} Let $G$ be an infinite group, let $m,\ell \in \N_{>0}$, and let $\alpha$ be a cardinal. Suppose that there exists a set $\mathcal{H} \subseteq \Sub(G)$ satisfying the following conditions: \begin{enumerate}[label=\textnormal{(\Alph{enumi})}]
	\item\label{theorem:abstract.2.a} $G$ is full for some element of $\mathcal{H}^{\alpha}$. 
	\item\label{theorem:abstract.2.b} For every $H \in \mathcal{H}$, there exist a finite subset $F \subseteq H$ and $K \in \mathcal{H}$ such that \begin{displaymath}
					\qquad K \, \subseteq \, \Cl_{H}(F)^{m} .
				\end{displaymath}
	\item\label{theorem:abstract.2.c} For every $H \in \mathcal{H}$, there exists a finite subset $F \subseteq G$ such that \begin{displaymath}
					\qquad G \, = \, (H \cup F)^{\ell} .
				\end{displaymath}
\end{enumerate} Then $\scf_{3m\ell}(G) > \alpha$. \end{thm}

\begin{proof} In view of Remark~\ref{remark:uniform.strong.cofinality}\ref{remark:uniform.strong.cofinality.2}, let $\mathcal{D}$ be a directed subset of $(\Pow(G),{\subseteq})$ such that \begin{itemize}
	\item[---\,] $\vert \mathcal{D} \vert \leq \alpha$,
	\item[---\,] $G = \bigcup \mathcal{D}$,
	\item[---\,] $1 \in D = D^{-1}$ for every $D \in \mathcal{D}$.
\end{itemize} It now suffices to show that $G = D^{3m\ell}$ for some $D \in \mathcal{D}$. We establish this conclusion in three consecutive claims.

\textit{Claim 0}. There exist $D_{0} \in \mathcal{D}$ and $H_{0} \in \mathcal{H}$ such that $D_{0}$ is full for $H_{0}$.

\textit{Proof of Claim 0}. From~\ref{theorem:abstract.2.a}, we know that $G$ is full for some element of $\mathcal{H}^{\alpha}$. Since $G = \bigcup \mathcal{D}$ and $\alpha \geq \vert \mathcal{D} \vert$, our Lemma~\ref{lemma:full.induction} implies the existence of $D_{0} \in \mathcal{D}$ and $H_{0} \in \mathcal{H}$ such that $D_{0}$ is full for $H_{0}$. $\qed_{\text{Claim\! 0}}$

\textit{Claim 1}. There exist $D_{1} \in \mathcal{D}$ and $H_{1} \in \mathcal{H}$ such that $H_{1} \subseteq D_{1}^{3m}$.

\textit{Proof of Claim 1}. According to~\ref{theorem:abstract.2.b}, we find a finite subset $F_{1} \subseteq H_{0}$ and $H_{1} \in \mathcal{H}$ such that $H_{1} \subseteq \Cl_{H_{0}}(F_{1})^{m}$. As $F_{1}$ is finite and $\mathcal{D}$ is a directed subset of $(\Pow(G),{\subseteq})$ with $G = \bigcup \mathcal{D}$, we find $D_{1} \in \mathcal{D}$ with $D_{0} \cup F_{1} \subseteq D_{1}$. We will show that $H_{1} \subseteq D_{1}^{3m}$. To this end, let $h \in H_{1}$. Then there exist $g_{1},\ldots,g_{m} \in F_{1}$ and $t_{1},\ldots,t_{m} \in H_{0}$ such that $h = t_{1}g_{1}t_{1}^{-1}\cdots t_{m}g_{m}t_{m}^{-1}$. Since $D_{0}$ is full for $H_{0}$, for each $i \in \{ 1,\ldots,m \}$ there exists $s_{i} \in D_{0}$ with $t_{i}g_{i}t_{i}^{-1} = s_{i}g_{i}s_{i}^{-1}$. Now, \begin{align*}
	h \, &= \, t_{1}g_{1}t_{1}^{-1}\cdots t_{m}g_{m}t_{m}^{-1} \\
		& = \, s_{1}g_{1}s_{1}^{-1}\cdots s_{m}g_{m}s_{m}^{-1} \, \in \, \! \left(D_{0}D_{1}D_{0}^{-1}\right)^{m} \, \subseteq \, D_{1}^{3m} . \hfill \qed_{\text{Claim\! 1}}
\end{align*} 

\textit{Claim 2}. There exists $D_{2} \in \mathcal{D}$ such that $G = D_{2}^{3m\ell}$.

\textit{Proof of Claim 2}. By~\ref{theorem:abstract.2.c}, there exists a finite subset $F_{2} \subseteq G$ with $G = (H_{1} \cup F_{2})^{\ell}$. Since $F_{2}$ is finite and $\mathcal{D}$ is a directed subset of $(\Pow(G),{\subseteq})$ with $G = \bigcup \mathcal{D}$, we find some $D_{2} \in \mathcal{D}$ such that $D_{1} \cup F_{2} \subseteq D_{2}$. Then \begin{displaymath}
	G \, = \, (H_{1} \cup F_{2})^{\ell} \, \subseteq \, \left( D_{1}^{3m} \cup F_{2} \right)^{\ell} \, \subseteq \, D_{2}^{3m\ell} . \hfill \qed_{\text{Claim\! 2}}
\end{displaymath}

This shows that $\scf_{3m\ell}(G) > \alpha$. \end{proof}

\begin{thm}\label{theorem:abstract.1} Let $G$ be an infinite group and let $\alpha$ be a cardinal. Suppose that there exists a set $\mathcal{H} \subseteq \Sub(G)$ satisfying the following conditions: \begin{enumerate}[label=\textnormal{(\Alph{enumi})}]
	\item\label{theorem:abstract.1.a} $G$ is full for some element of $\mathcal{H}^{\alpha}$. 
	\item\label{theorem:abstract.1.b} For every $H \in \mathcal{H}$, there exist $m \in \N$, a finite subset $F \subseteq H$ and $K \in \mathcal{H}$ with \begin{displaymath}
				\qquad K \, \subseteq \, \Cl_{H}(F)^{m} .
			\end{displaymath}
	\item\label{theorem:abstract.1.c} For every $H \in \mathcal{H}$, there exist $\ell \in \N$ and a finite subset $F \subseteq G$ such that \begin{displaymath}
				\qquad G \, = \, (H \cup F)^{\ell} .
			\end{displaymath}
\end{enumerate} Then $\scf(G) > \alpha$. \end{thm}

\begin{proof} This follows by revisiting the proof of Theorem~\ref{theorem:abstract.2}. \end{proof}

Let us conclude this section by discussing an example. The following result is essentially contained in Bergman's work~\cite{bergman}, but with a different proof.

\begin{thm}[cf.~{\cite{bergman}}]\label{theorem:bergman} Let $X$ be an infinite set. Then $\vert X \vert < \scf_{42}(\Sym(X))$. \end{thm}

\begin{proof} Let $\mathcal{M} \defeq \{ Y \subseteq X \mid \vert Y \vert = \vert X \vert = \vert X\setminus Y \vert \}$. Consider the group $G \defeq \Sym(X)$ along with its subgroups \begin{displaymath}
	S(Y) \, \defeq \, \{ g \in G \mid X \setminus Y \subseteq \Fix(g) \} \qquad (Y \subseteq X) .
\end{displaymath} We verify the conditions specified in Theorem~\ref{theorem:abstract.2} for $\mathcal{H} \defeq \{ S(Y) \mid Y \in \mathcal{M} \}$ with respect to $m \defeq 2$ and $\ell \defeq 7$.

\ref{theorem:abstract.2.a} Since $X$ is infinite, we may invoke the axiom of choice to find a family $(Y_{x})_{x \in X}$ of pairwise disjoint subsets of $X$ such that $X = \bigcupdot_{x \in X} Y_{x}$ and $\vert Y_{x} \vert = \vert X \vert$ for all~$x \in X$. Evidently, $\{ Y_{x} \mid x \in X \} \subseteq \mathcal{M}$ by the Schröder--Bernstein theorem. Consider the group $H \defeq \prod_{x \in X} \Sym(Y_{x})$ along with its subgroups \begin{displaymath}
	H_{x} \, \defeq \, \{ h \in H \mid \forall y \in X \setminus \{ x \} \colon \, h_{y} = 1 \} \qquad (x \in X)
\end{displaymath} and the embedding $\iota \colon H \to G$ defined by \begin{displaymath}
	\iota(h)\vert_{Y_{x}} \, \defeq \, h_{x} \qquad (h \in H, \, x \in X) .
\end{displaymath} Since $H$ is full for $(H_{x})_{x \in X}$ according to Lemma~\ref{lemma:direct.products.are.full} and $\iota$ is a homomorphism, Lemma~\ref{lemma:full.homomorphic.images} asserts that $\iota (H)$ is full for $(\iota(H_{x}))_{x \in X} = (S(Y_{x}))_{x \in X}$, whence $G$ is full for $(S(Y_{x}))_{x \in X} \in \mathcal{H}^{X}$ thanks to Remark~\ref{remark:full.persistence}.	
	
\ref{theorem:abstract.2.b} Let $Y \in \mathcal{M}$. Since $Y$ is infinite, \cite[Lemma~14]{bergman} asserts the existence of some $g \in \Sym(Y)$ such that $\Sym(Y) = \Cl_{\Sym(Y)}(g)^{2}$. As $\Sym(Y) \cong S(Y)$, this entails the desired conclusion.
	
\ref{theorem:abstract.2.c} Let $Y \in \mathcal{M}$. Using the axiom of choice, we find $Y_{0},Y_{1} \subseteq Y$ such that $Y = Y_{0} \cupdot Y_{1}$ and $\vert Y_{0} \vert = \vert Y \vert = \vert Y_{1} \vert$. Moreover, let $Y_{2} \defeq X\setminus Y$. Since $X = Y_{0} \cupdot Y_{1} \cupdot Y_{2}$ and $\vert Y_{1} \vert = \vert Y_{2} \vert$, there exists $g \in \I(\Sym(X))$ such that $\Fix(g) = Y_{0}$ and $g(Y_{1}) = Y_{2}$, which implies that \begin{displaymath}
	S(Y_{0}\cup Y_{2}) \, = \, S(g(Y)) \, = \, gS(Y)g .
\end{displaymath} Using~\cite[Lemma~2.1]{galvin}, we conclude that \begin{align*}
	\Sym(X) \, &= \, S(Y_{0}\cup Y_{1})S(Y_{0}\cup Y_{2})S(Y_{0}\cup Y_{1}) \cup S(Y_{0}\cup Y_{2})S(Y_{0}\cup Y_{1})S(Y_{0}\cup Y_{2}) \\
	&= \, S(Y)S(Y_{0}\cup Y_{2})S(Y) \cup S(Y_{0}\cup Y_{2})S(Y)S(Y_{0}\cup Y_{2}) \\
	&= \, S(Y)gS(Y)gS(Y) \cup gS(Y)gS(Y)gS(Y)g \, \subseteq \, (S(Y) \cup \{ g \})^{7} .
\end{align*}
		
By Theorem~\ref{theorem:abstract.2}, this gives the desired conclusion. \end{proof}

\section{A variation on Kazhdan's property}\label{section:kazhdan}

We briefly turn to the study of unitary representations and exhibit, in Proposition~\ref{proposition:restricted.kazhdan.property} below, a connection between uniform strong cofinality and a weak version of Kazhdan's property~$(T)$ suited for non-finitely generated groups (cf.~Definition~\ref{definition:restricted.property.t}). For convenience, let us agree on the following basic terminology.

\begin{definition} Let $\alpha$ be an infinite cardinal. An \emph{$\alpha$-net} in a set $X$ is a net $(x_{i})_{i \in I}$ in $X$ such that $\vert I \vert \leq \alpha$. \end{definition}

As usual, if $\pi$ is a unitary representation of a group $G$ on a Hilbert space $H$, then the set of \emph{invariant} vectors for $\pi$ will be denoted by \begin{displaymath}
	\Fix(\pi) \, \defeq \, \{ x \in H \mid \forall g \in G \colon \, \pi(g)x = x \} ,
\end{displaymath} and a net $(x_{i})_{i \in I}$ of vectors in $H$ will be called \emph{almost invariant} with respect to $\pi$ if \begin{displaymath}
	\forall g \in G \colon \qquad \lim\nolimits_{i \to I} \Vert x_{i} - \pi(g)x_{i} \Vert \, = \, 0 .
\end{displaymath}

\begin{definition}\label{definition:restricted.property.t} Let $\alpha$ be an infinite cardinal. We say that a group $G$ has \emph{property~$(T_{\alpha})$} if every unitary representation of $G$ admitting an almost invariant $\alpha$-net of unit vectors has an invariant unit vector. \end{definition}

\begin{remark}\label{remark:restricted.property.t} A group $G$ has property $(T)$ if and only if $G$ has property $(T_{\vert G \vert \vee \aleph_{0}})$ if and only if $G$ has property $(T_{\alpha})$ for every infinite cardinal $\alpha$. \end{remark}

A group $G$ is said to have \emph{Serre's property $(FH)$} if every action of $G$ by isometries on a real Hilbert space has a fixed point. Equivalently, a group $G$ possesses Serre's property~$(FH)$ if and only if the first cohomology of every orthogonal representation of $G$ vanishes (see, e.g.,~\cite[Theorem~2.10.4]{KazhdanBook}). We record an obvious extension of a result due to Delorme~\cite[Théorème~V.1]{delorme} (see also~\cite[Theorem~2.12.4(i)]{KazhdanBook}).

\begin{prop} Property $(T_{\aleph_{0}})$ implies Serre's property $(FH)$. \end{prop}

\begin{proof} The argument given in~\cite[Theorem~2.12.4(i)]{KazhdanBook} (using~\cite[Proposition~2.11.2]{KazhdanBook}) applies verbatim. \end{proof}

To point out a characterization in terms of positive-definite functions, let us recall that a \emph{positive-definite} function on a group $G$ is a map $f \colon G \to \C$ such that \begin{displaymath}
	\sum\nolimits_{i,j=1}^{n} c_{i}\overline{c_{j}} f\!\left(g_{j}^{-1}g_{i}\right)\! \, \geq \, 0 
\end{displaymath} for all $n \in \N$, $c_{1},\ldots,c_{n} \in \C$ and $g_{1},\ldots,g_{n} \in G$. We refer to~\cite[Appendix~C.4]{KazhdanBook} for further background on this class of functions.

\begin{lem}[{\cite[Proposition~C.4.2]{KazhdanBook}}]\label{lemma:positive.definite.functions} Let $G$ be a group and let $f \colon G \to \C$ be positive-definite. Then, for all $g,h \in G$, \begin{enumerate}
	\item\label{lemma:positive.definite.functions.0} $\overline{f(g)} = f\!\left( g^{-1}\right)$,
	\item\label{lemma:positive.definite.functions.1} $\vert f(g) \vert \leq f(1)$, 
	\item\label{lemma:positive.definite.functions.2} $\vert f(g) - f(h) \vert^{2} \leq 2f(1)\!\left(f(1)-\Re f\!\left(g^{-1}h\right)\right)$.
\end{enumerate} \end{lem}

The following result is just a straightforward adaptation of a well-known characterization of Kazhdan's property $(T)$ (cf.~\cite[5.a.11 Théorème, p.~61]{DeLaHarpeValette}).

\begin{prop}\label{proposition:restricted.kazhdan.property.characterization} Let $\alpha$ be an infinite cardinal and let $G$ be a group. Then the following are equivalent. \begin{enumerate}
	\item\label{proposition:restricted.kazhdan.property.characterization.1} $G$ has property $(T_{\alpha})$
	\item\label{proposition:restricted.kazhdan.property.characterization.2} For every unitary representation $\pi$ of $G$ and every almost invariant $\alpha$-net $(x_{i})_{i \in I}$ of unit vectors for $\pi$, \begin{displaymath}
			\qquad \lim\nolimits_{i \to I} \dist(x_{i},\Fix(\pi)) \, = \, 0 .
		\end{displaymath}
	\item\label{proposition:restricted.kazhdan.property.characterization.3} Every $\alpha$-net of positive-definite functions on $G$ converging pointwise to $1$ converges uniformly to $1$.
\end{enumerate} \end{prop}

\begin{proof} We include the standard argument only for the sake of completeness.
	
\ref{proposition:restricted.kazhdan.property.characterization.1}$\Longrightarrow$\ref{proposition:restricted.kazhdan.property.characterization.2}. Let $(x_{i})_{i \in I}$ be an almost invariant $\alpha$-net of unit vectors for a unitary representation $\pi$ of $G$ on a Hilbert space $H$. Consider the $\pi$-invariant subspaces $H_{0} \defeq \Fix(\pi)$ and $H_{1} \defeq H_{0}^{\perp}$ along with the orthogonal projection $p \colon H \to H_{0}$ and the quantity \begin{displaymath}
	\epsilon \, \defeq \, \limsup\nolimits_{i \to I} \dist(x_{i},H_{0}) \, = \, \limsup\nolimits_{i \to I} \Vert x_{i} - px_{i} \Vert \, \in \, \R_{\geq 0} .
\end{displaymath} For contradiction, suppose that $\epsilon > 0$. Then there exists a cofinal subset $J \subseteq I$ such that $\Vert x_{i} - px_{i} \Vert \geq \tfrac{\epsilon}{2}$ for each $i \in J$. The $\alpha$-net of unit vectors \begin{displaymath}
	y_{i} \, \defeq \, \tfrac{1}{\Vert x_{i}-px_{i} \Vert}(x_{i}-px_{i}) \, \in \, H_{1} \qquad (i \in J)
\end{displaymath} is almost invariant for the unitary representation $\pi_{1} \colon G \to \U(H_{1}), \, g \mapsto \pi(g)\vert_{H_{1}}$: indeed, if $g \in G$, then \begin{align*}
	\Vert y_{i} - \pi_{1}(g)y_{i} \Vert^{2} \, = \, \tfrac{\Vert x_{i}-\pi(g)x_{i} \Vert^{2}}{\Vert x_{i}-px_{i} \Vert^{2}} \, \leq \, \tfrac{4}{\epsilon^{2}}\Vert x_{i}-\pi(g)x_{i} \Vert^{2} \, \stackrel{i \to J}{\longrightarrow} \, 0 .
\end{align*} Since $G$ has property $(T_{\alpha})$, this necessitates that $\pi_{1}$ admits an invariant unit vector, contradicting the fact that $H_{0} \perp H_{1}$. Hence, $\epsilon = 0$ as desired.
	
\ref{proposition:restricted.kazhdan.property.characterization.2}$\Longrightarrow$\ref{proposition:restricted.kazhdan.property.characterization.3}. Let $(f_{i})_{i \in I}$ be an $\alpha$-net of positive-definite functions on $G$ converging pointwise to $1$. In particular, $f_{i}(1) \to 1$ as $i \to I$. Hence, upon replacing $(f_{i})_{i \in I}$ by $(f_{i}(1)^{-1}f_{i})_{i \in J}$ for the cofinal subset $J \defeq \{ i \in I \mid \vert 1-f_{i}(1) \vert \leq 1/2 \} \subseteq I$, we may and will assume without loss of generality that $f_{i}(1) = 1$ for all $i \in I$. According to~\cite[Theorem~C.4.10]{KazhdanBook}, for each $i \in I$, we find a unitary representation $\pi_{i}$ of $G$ on a Hilbert space $H_{i}$ and a unit vector $x_{i} \in H_{i}$ such that $f_{i}(g) = \langle x_{i},\pi_{i}(g)x_{i} \rangle$ for all $g \in G$. Considering the Hilbert space direct sum $H \defeq \bigoplus_{i \in I} H_{i}$ and the unitary representation $\pi \colon G \to \U(H), \, g \mapsto \bigoplus_{i \in I} \pi_{i}(g)$, we may view $(H_{i})_{i \in I}$ as subspaces of $H$ and $(\pi_{i})_{i \in I}$ as subrepresentations of $\pi$. With this identification, $(x_{i})_{i \in I}$ constitutes an almost invariant $\alpha$-net of unit vectors for $\pi$: for every $g \in G$, \begin{displaymath}
	\Vert x_{i}-\pi(g)x_{i} \Vert^{2} \, = \, \Vert x_{i}-\pi_{i}(g)x_{i} \Vert^{2} \, = \, 2-2\Re \langle x_{i},\pi_{i}(g)x_{i} \rangle \, = \, 2(1-\Re f_{i}(g)) \, \stackrel{i \to I}{\longrightarrow} \, 0 .
\end{displaymath} Thus, by~\ref{proposition:restricted.kazhdan.property.characterization.2}, there exists a net $(y_{i})_{i \in I} \in \Fix(\pi)^{I}$ such that $\lim_{i \to I} \Vert x_{i} - y_{i} \Vert = 0$. Since \begin{align*}
	&\vert 1 - f_{i}(g) \vert \, = \vert \langle x_{i},x_{i} \rangle - \langle x_{i},\pi(g)x_{i} \rangle \vert \, = \, \vert \langle x_{i},x_{i}-\pi(g)x_{i} \rangle \vert \, \leq \, \Vert x_{i}-\pi(g)x_{i} \Vert \\
	&\qquad \leq \, \Vert x_{i}-y_{i} \Vert + \Vert y_{i}-\pi(g)x_{i} \Vert \, = \, \Vert x_{i}-y_{i} \Vert + \Vert \pi(g)(y_{i}-x_{i}) \Vert \, = \, 2 \Vert x_{i}-y_{i} \Vert
\end{align*} for all $g \in G$ and $i \in I$, it follows that \begin{displaymath}
	\sup\nolimits_{g \in G} \vert 1-f_{i}(g) \vert \, \leq \, 2\Vert x_{i} - y_{i} \Vert \, \stackrel{i \to I}{\longrightarrow} \, 0 ,
\end{displaymath} i.e., the net $(f_{i})_{i \in I}$ converges uniformly to $1$.
	
\ref{proposition:restricted.kazhdan.property.characterization.3}$\Longrightarrow$\ref{proposition:restricted.kazhdan.property.characterization.1}. Let $(x_{i})_{i \in I}$ be an almost invariant $\alpha$-net of unit vectors for a unitary representation $\pi$ of $G$ on a Hilbert space $H$. For each $i \in I$, \begin{displaymath}
	f_{i} \colon \, G \, \longrightarrow \, \C, \quad g \, \longmapsto \, \langle x_{i},\pi(g)x_{i} \rangle	
\end{displaymath} constitutes a positive-definite function on $G$ (see, e.g.,~\cite[Proposition~C.4.3]{KazhdanBook}). We observe that the net $(f_{i})_{i \in I}$ converges pointwise to $1$: indeed, if $g \in G$, then \begin{displaymath}
	\vert 1 - f_{i}(g) \vert \, = \vert \langle x_{i},x_{i} \rangle - \langle x_{i},\pi(g)x_{i} \rangle \vert \, = \, \vert \langle x_{i},x_{i}-\pi(g)x_{i} \rangle \vert \, \leq \, \Vert x_{i}-\pi(g)x_{i} \Vert \, \stackrel{i \to I}{\longrightarrow} \, 0 .
\end{displaymath} Hence, due to~\ref{proposition:restricted.kazhdan.property.characterization.3}, the net $(f_{i})_{i \in I}$ converges uniformly to $1$. In particular, there exists $i \in I$ such that $\sup_{g \in G} \vert 1 - f_{i}(g) \vert < \tfrac{1}{\sqrt{2}}$. Since \begin{displaymath}
	\Vert x_{i}-\pi(g)x_{i} \Vert^{2} \, = \, 2-2\Re \langle x_{i},\pi(g)x_{i} \rangle \, = \, 2(1-\Re f_{i}(g)) \, \leq \, 2\vert 1-f_{i}(g) \vert
\end{displaymath} for every $g \in G$, we conclude that \begin{displaymath}
	\sup\nolimits_{g \in G} \Vert x_{i}-\pi(g)x_{i} \Vert \, \leq \, \sqrt{2\sup\nolimits_{g \in G} \vert 1-f_{i}(g) \vert} \, < \, 2^{1/4} \, < \, \sqrt{2} .
\end{displaymath} By~\cite[Proposition~1.1.5]{KazhdanBook}, this implies that $\pi$ admits an invariant unit vector. \end{proof}

We turn to connections with (uniform strong) cofinality. Our first observation is inspired by the fact that property~$(T)$ implies finite generation~\cite[Theorem~1.3.1]{KazhdanBook}.

\begin{prop} Let $\alpha$ be an infinite cardinal and let $G$ be a group. If $G$ has property $(T_{\alpha})$, then either $G$ has property $(T)$, or $\cf(G) > \alpha$. \end{prop}

\begin{proof} First of all, let us observe that the two alternative conclusions are indeed mutually exclusive: if $G$ has property~$(T)$, then $G$ must be finitely generated (see, e.g.,~\cite[Theorem~1.3.1]{KazhdanBook}), whence $\cf(G) = 0$ by definition.
	
Suppose now that $G$ has property $(T_{\alpha})$. If $G$ is finitely generated, then $\vert G \vert \leq \aleph_{0} \leq \alpha$ and so $G$ has property $(T_{\vert G \vert \vee \aleph_{0}})$, thus $G$ has property $(T)$ by Remark~\ref{remark:restricted.property.t}. On the other hand, if $G$ is not finitely generated and $\cf(G) \leq \alpha$, then there exists a chain $\mathcal{C}$ of proper subgroups of $G$ such that $G = \bigcup \mathcal{C}$ and $\vert \mathcal{C} \vert \leq \alpha$, wherefore the corresponding family of characteristic functions $(\chi_{C})_{C \in \mathcal{C}}$ constitutes an $\alpha$-net of positive-definite functions on $G$ converging to $1$ pointwise, but not converging to $1$ uniformly, which contradicts $G$ having property~$(T_{\alpha})$ according to Proposition~\ref{proposition:restricted.kazhdan.property.characterization}. \end{proof}

We arrive at this section's main insight.

\begin{prop}\label{proposition:restricted.kazhdan.property} Let $\alpha$ be an infinite cardinal and let $G$ be a group. If $\scf_{k}(G) > \alpha$ for some $k \in \N_{>0}$, then $G$ has property $(T_{\alpha})$, but not property $(T)$. \end{prop}

\begin{proof} Suppose that $\scf_{k}(G) > \alpha$ for some $k \in \N_{>0}$. In order to verify condition~\ref{proposition:restricted.kazhdan.property.characterization.3} of Proposition~\ref{proposition:restricted.kazhdan.property.characterization}, let $(f_{i})_{i \in I}$ be an $\alpha$-net of positive-definite functions on $G$ converging pointwise to $1$. To deduce uniform convergence, let $\epsilon \in (0,1]$. Define $\delta \defeq \tfrac{\epsilon^{2}}{8(k+1)^{2}}$. Then \begin{align*}
	D \colon \, I \, &\longrightarrow \, \Pow(G) \\
	i \, &\longmapsto \, \{ g \in G \mid \forall j \in I \colon \, i \leq j \, \Longrightarrow \, \max\{\vert 1-f_{j}(1) \vert, \vert 1 - f_{j}(g) \vert\} \leq \delta \}
\end{align*} is a monotone map from $(I,{\leq})$ to $(\Pow(G),{\subseteq})$. Hence, $\mathcal{D} \defeq \{ D(i) \mid i \in I \}$ is a directed subset of $(\Pow(G),{\subseteq})$. Since $G = \bigcup \mathcal{D}$ by pointwise convergence of $(f_{i})_{i \in I}$ to $1$ and $\vert \mathcal{D} \vert \leq \vert I \vert \leq \alpha < \scf_{k}(G)$, there exists $i_{0} \in I$ such that $G = D(i_{0})^{k}$. Consequently, if $i \in I$ and $i_{0} \leq i$, then for every $g \in G$ there exist $h_{1},\ldots,h_{k} \in D(i_{0})$ such that $g = h_{1}\cdots h_{k}$, thus \begin{align*}
	\vert f_{i}(1) - f_{i}(g) \vert \, &\leq \, \sum\nolimits_{\ell = 1}^{k} \vert f_{i}(h_{1}\cdots h_{\ell-1}) - f_{i}(h_{1}\cdots h_{\ell}) \vert \\
	&\stackrel{\ref{lemma:positive.definite.functions}\ref{lemma:positive.definite.functions.0}+\ref{lemma:positive.definite.functions}\ref{lemma:positive.definite.functions.2}}{\leq} \, \sum\nolimits_{\ell = 1}^{k} \sqrt{2f_{i}(1)(f_{i}(1) - \Re f_{i}(h_{\ell}))} \, \leq \, k\sqrt{8\delta}
\end{align*} and therefore \begin{displaymath}
	\vert 1-f_{i}(g) \vert \, \leq \, \vert 1-f_{i}(1) \vert + \vert f_{i}(1) -f_{i}(g) \vert \, \leq \, (k+1)\sqrt{8\delta} \, = \, \epsilon ,
\end{displaymath} which means that $\Vert 1-f_{i} \Vert_{\infty} \leq \epsilon$. This shows that $(f_{i})_{i \in I}$ converges uniformly to~$1$. Using Proposition~\ref{proposition:restricted.kazhdan.property.characterization}, we conclude that $G$ has property $(T_{\alpha})$. On the other hand, our hypothesis entails that $G$ is infinite and has the Bergman property due to Lemma~\ref{lemma:uniform.strong.cofinality}\ref{lemma:uniform.strong.cofinality.3} and Remark~\ref{remark:strong.distortion.implies.uscf}\ref{remark:strong.distortion.implies.uscf.-1}. Since property~$(T)$ implies finite generation (see, e.g.,~\cite[Theorem~1.3.1]{KazhdanBook}) and finitely generated Bergman groups are necessarily finite, $G$ thus fails to have property $(T)$. \end{proof}

\begin{remark} We record an alternative argument for the first part of the proof of Proposition~\ref{proposition:restricted.kazhdan.property}, avoiding the application of Proposition~\ref{proposition:restricted.kazhdan.property.characterization} and appealing to~\cite[Proposition~1.1.5]{KazhdanBook} instead. Suppose that $\scf_{k}(G) > \alpha$ for some $k \in \N_{>0}$. Let $(x_{i})_{i \in I}$ be an almost invariant $\alpha$-net of unit vectors for a unitary representation $\pi$ of $G$ on a Hilbert space $H$. Then \begin{displaymath}
	D \colon \, I \, \longrightarrow \, \Pow(G), \quad i \, \longmapsto \, \left\{ g \in G \left\vert \, \forall j \in I \colon \, i \leq j \, \Longrightarrow \, \Vert x_{j} - \pi(g)x_{j} \Vert \leq \tfrac{1}{k\sqrt{2}} \right\} \right.
\end{displaymath} is a monotone map from $(I,{\leq})$ to $(\Pow(G),{\subseteq})$. Hence, $\mathcal{D} \defeq \{ D(i) \mid i \in I \}$ is a directed subset of $(\Pow(G),{\subseteq})$. Since $G = \bigcup \mathcal{D}$ by assumption and $\vert \mathcal{D} \vert \leq \vert I \vert \leq \alpha < \scf_{k}(G)$, there exists $i \in I$ such that $G = D(i)^{k}$. Consequently, for every $g \in G$, there exist $h_{1},\ldots,h_{k} \in D(i)$ such that $g = h_{1}\cdots h_{k}$, and it follows that \begin{align*}
	\Vert x_{i} - \pi(g)x_{i} \Vert \, &= \, \Vert x_{i} - \pi(h_{1}\cdots h_{k})x_{i} \Vert \, \leq \, \sum\nolimits_{\ell = 1}^{k} \Vert \pi(h_{1}\cdots h_{\ell-1})x_{i} - \pi(h_{1}\cdots h_{\ell})x_{i} \Vert \\
		&= \, \sum\nolimits_{\ell = 1}^{k} \Vert x_{i} - \pi(h_{\ell})x_{i} \Vert \, \leq \, \tfrac{1}{\sqrt{2}} \, = \, \tfrac{\sqrt{2}}{2} \, < \, \sqrt{2} .
\end{align*} By~\cite[Proposition~1.1.5]{KazhdanBook}, this necessitates that $\pi$ admits an invariant unit vector. This shows that $G$ has property $(T_{\alpha})$. \end{remark}

\begin{cor} If $X$ is an infinite set, then $\Sym(X)$ has property $(T_{\vert X \vert})$, but not property $(T)$. \end{cor}

\begin{proof} This follows from Theorem~\ref{theorem:bergman} and Proposition~\ref{proposition:restricted.kazhdan.property}. \end{proof}

\section{Idempotents and their conjugacy classes}\label{section:idempotents}

This preliminary section recollects some basic well-known material from ring theory, more precisely, concerning orbits of idempotents under the natural conjugation action by the corresponding unit group.

For a start, let us briefly clarify some bits of notation. Let $R$ be a unital ring. Then we consider its \emph{center} \begin{displaymath}
	\cent(R) \, \defeq \, \{ a \in R \mid \forall b \in R \colon \, ab = ba \} ,
\end{displaymath} its \emph{unit group} \begin{displaymath}
	\GL(R) \, \defeq \, \{ a \in R \mid \exists b \in R \colon \, ab = ba = 1\} ,
\end{displaymath} as well as the corresponding set \begin{displaymath}
	\I(R) \, \defeq \, \left\{ a \in R \left\vert \, a^{2} = 1 \right\} \! \right. \, = \, \I(\GL(R))
\end{displaymath} of \emph{involutions}. Furthermore, let us recall that the set \begin{displaymath}
	\E(R) \, \defeq \, \{ e \in R \mid ee = e \}
\end{displaymath} of \emph{idempotents} of $R$ is partially ordered by the relation \begin{displaymath}
	e \leq f \ \, :\Longleftrightarrow \ \, ef = fe = e \qquad (e,f \in \E(R)) .
\end{displaymath} Two elements $e,f \in \E(R)$ are called \emph{orthogonal} and we write $e \perp f$ if $ef=fe=0$. Given a set $I$ and a family $(e_{i})_{i \in I} \in \E(R)^{I}$, we define \begin{displaymath}
	(e_{i})_{i \in I}\perp \quad :\Longleftrightarrow \quad \forall i,j \in I \colon \ i \ne j \, \Longrightarrow \, {e_{i}} \perp {e_{j}} .
\end{displaymath}

The endomorphism ring of a module $M$ will be denoted by $\End(M)$ and the corresponding automorphism group by $\GL(M) \defeq \GL(\End(M))$.

\begin{remark}\label{remark:cayley} Let $R$ be a unital ring and let $e \in \E(R)$. \begin{enumerate}
	\item\label{remark:cayley.1} The map \begin{displaymath}
					\qquad eRe \, \longrightarrow \, \End(eR_{R}), \quad a \, \longmapsto \, [b \mapsto ab]
			\end{displaymath} constitutes a ring isomorphism and, in turn, restricts to a group isomorphism from $\GL(eRe)$ to $\GL(eR_{R})$.
	\item\label{remark:cayley.2} $S \defeq \{ a \in R \mid ae = ea \}$ is a unital subring of $R$ and the map $S \to eRe, \, a \mapsto ae$ constitutes a unital ring homomorphism, thus restricts to a group homomorphism from $\GL(S)$ to $\GL(eRe)$.
\end{enumerate} \end{remark}

\begin{lem}[{\cite[Lemma~5.2]{BernardSchneider}}]\label{lemma:corner.group} Let $R$ be a unital ring and let $e \in \E(R)$. Then \begin{displaymath}
	\Gamma_{R}(e) \, \defeq \, \GL(eRe) + 1-e \, = \, \GL(R) \cap (eRe + 1-e)
\end{displaymath} is a subgroup of $\GL(R)$ and \begin{displaymath}
	\GL(eRe) \, \longrightarrow \, \Gamma_{R}(e) , \quad a \, \longmapsto \, a+1-e
\end{displaymath} is an isomorphism. Moreover, the following hold. \begin{enumerate}
	\item\label{lemma:corner.group.1} If $f \in \E(R)$ and $e \leq f$, then $\Gamma_{R}(e) \leq \Gamma_{R}(f)$.
	\item\label{lemma:corner.group.2} If $f \in \E(R)$ and $e \perp f$, then $ab = ba$ for all $a \in \Gamma_{R}(e)$ and $b \in \Gamma_{R}(f)$.
	\item\label{lemma:corner.group.3} If $g \in \GL(R)$, then $\Gamma_{R}\!\left( geg^{-1}\right) = g\Gamma_{R}(e)g^{-1}$.
\end{enumerate} \end{lem}

We recollect some well-known basic facts concerning the conjugacy of idempotents.

\begin{remark}[see, e.g.,~{\cite[Theorem~14, p.~21]{KaplanskyBook}} or~{\cite[Part~1, Chapter~1, \S1, Exercise~5A, pp.~8--9]{BerberianBook}}]\label{remark:algebraic.equivalence.of.idempotents} Let $R$ be a unital ring and let $e,f \in \E(R)$. The following are equivalent. \begin{enumerate}
	\item\label{remark:algebraic.equivalence.of.idempotents.1} ${eR_{R}} \cong {fR_{R}}$.
	\item\label{remark:algebraic.equivalence.of.idempotents.2} There exist $a,b \in R$ such that $ab = e$ and $ba = f$.
	\item\label{remark:algebraic.equivalence.of.idempotents.3} There exist $a \in eRf$ and $b \in fRe$ such that $ab = e$ and $ba = f$.
\end{enumerate} \end{remark}

\begin{lem}\label{lemma:ideal.isomorphism} Let $R$ be a unital ring, let $e_{0},e_{1},f \in \E(R)$ with $e_{0},e_{1} \leq f$, and let $S \defeq fRf$. Then ${e_{0}R_{R}} \cong {e_{1}R_{R}}$ if and only if ${e_{0}S_{S}} \cong {e_{1}S_{S}}$. \end{lem}

\begin{proof} Since $e_{0}Re_{1} = e_{0}Se_{1}$ and $e_{1}Re_{0} = e_{1}Se_{0}$, this follows from Remark~\ref{remark:algebraic.equivalence.of.idempotents}. \end{proof}

\begin{lem}\label{lemma:conjugacy.of.idempotents} Let $R$ be a unital ring and let $e,f \in \E(R)$. The following are equivalent. \begin{enumerate}
		\item\label{lemma:conjugacy.of.idempotents.1} ${eR_{R}} \cong {fR_{R}}$ and ${(1-e)R_{R}} \cong {(1-f)R_{R}}$.
		\item\label{lemma:conjugacy.of.idempotents.2} There exists $g \in \GL(R)$ such that $geg^{-1} = f$.
\end{enumerate} \end{lem}

\begin{proof} \ref{lemma:conjugacy.of.idempotents.2}$\Longrightarrow$\ref{lemma:conjugacy.of.idempotents.1}. If $g \in \GL(R)$ is such that $geg^{-1} = f$, then \begin{align*}
	eg^{-1}ge \, &= \, ee \, = \, e, \\
	geeg^{-1} \, &= \, geg^{-1} \, = \, f , \\
	(1-e)g^{-1}g(1-e) \, &= \, (1-e)(1-e) \, = \, 1-e, \\
	g(1-e)(1-e)g^{-1} \, &= \, g(1-e)g^{-1} \, = \, 1-geg^{-1} \, = \, 1-f ,
\end{align*} wherefore ${eR_{R}} \cong {fR_{R}}$ and ${(1-e)R_{R}} \cong {(1-f)R_{R}}$ by Remark~\ref{remark:algebraic.equivalence.of.idempotents}.
	
\ref{lemma:conjugacy.of.idempotents.1}$\Longrightarrow$\ref{lemma:conjugacy.of.idempotents.2}. Suppose that ${eR_{R}} \cong {fR_{R}}$ and ${(1-e)R_{R}} \cong {(1-f)R_{R}}$. By Remark~\ref{remark:algebraic.equivalence.of.idempotents}, we find $u,v,u',v' \in R$ such that $uv = e$, $vu = f$, $u'v' = 1-e$, and $v'u' = 1-f$. Considering $g \defeq ve + v'(1-e)$ and $h \defeq uf + u'(1-f)$, we note that \begin{align*}
	hg \, &= \, (uf + u'(1-f))(ve + v'(1-e)) \\
		&= \, ufve + ufv'(1-e) + u'(1-f)ve + u'(1-f)v'(1-e) \\
		&= \, uvuve + ufv'u'v' + u'v'u'vuv + u'v'u'v'(1-e) \\
		&= \, e + uf(1-f)v' + u'(1-f)fv + 1-e \, = \, e+1-e \, = \, 1 , \\
		gh \, &= \, (ve + v'(1-e))(uf + u'(1-f)) \\
		&= \, veuf + veu'(1-f) + v'(1-e)uf + v'(1-e)u'(1-f) \\
		&= \, vuvuf + vuvu'v'u' + v'(1-e)uvu + v'u'v'u'(1-f) \\
		&= \, f + ve(1-e)u' + v'(1-e)eu + 1-f \, = \, f+1-f \, = \, 1 ,
\end{align*} whence $g,h \in \GL(R)$ and $g^{-1} = h$. Furthermore, \begin{align*}
	geg^{-1} \, &= \, geh \, = \, (ve + v'(1-e))e(uf + u'(1-f)) \, = \, ve(uf + u'(1-f)) \\
		& = \, veuf + veu'(1-f) \, = \, vuvuf + vuvu'v'u' \, = \, f + ve(1-e)u' \, = \, f . \qedhere
\end{align*} \end{proof}

\begin{lem}\label{lemma:conjugacy.of.commuting.idempotents} Let $R$ be a unital ring, let $e \in \E(R)$ and $f \in \{ geg^{-1} \mid g \in \GL(R) \}$. If $e \perp f$, then there exists $g \in \I(\Gamma_{R}(e+f))$ such that $geg = f$. \end{lem}

\begin{proof} Suppose that $e \perp f$. Choose any $g \in \GL(R)$ such that $geg^{-1} = f$. Considering \begin{displaymath}
		u \, \defeq \, ge + g^{-1}f + 1-(e+f) ,
	\end{displaymath} we observe that \begin{align*}
		u^{2} \, &= \, gege + geg^{-1}f + ge(1-(e+f)) \\
		& \qquad + g^{-1}fge + g^{-1}fg^{-1}f + g^{-1}f(1-(e+f)) \\
		& \qquad + (1-(e+f))ge + (1-(e+f))g^{-1}f + 1-(e+f) \\
		&= \, gefg + ff + ge(1-e)(1-(e+f)) \\
		& \qquad + ee + g^{-1}feg^{-1} + g^{-1}f(1-f)(1-(e+f)) \\
		& \qquad + (1-(e+f))(1-f)fg + (1-(e+f))(1-e)eg^{-1} + 1-(e+f) \\
		&= \, f + e + 1-(e+f) \, = \, 1
	\end{align*} and \begin{align*}
		u(1-(e+f)) \, &= \, 1-(e+f) \\
		&= \, (1-(e+f))\!\left( fg + eg^{-1} + 1-(e+f) \right) \, = \, (1-(e+f))u ,
	\end{align*} wherefore $u \in \I(\Gamma_{R}(e+f))$. Furthermore, \begin{align*}
		ueu \, &= \, ge\!\left(ge + g^{-1}f + 1-(e+f)\right) \, = \, gege + geg^{-1}f + ge(1-(e+f)) \\
		& = \, gefg + ff + ge(1-e)(1-(e+f)) \, = \, f . \qedhere
\end{align*} \end{proof}

\section{Injectivity and regularity}\label{section:strcuture.theory}

A right module $E$ over a ring $R$ is called \emph{injective} if, for every pair of right $R$-modules $A \leq B$ and every homomorphism $\phi \colon A \to E$, there exists a homomorphism $\psi \colon B \to E$ such that ${\psi\vert_{A}} = \phi$. A ring $R$ is called \emph{right self-injective} if the right $R$-module $R_{R}$ is injective.

\begin{remark}\label{remark:summands.of.injective.modules} Any direct summand of an injective module is injective (see, e.g.,~\cite[Proposition~18.2]{AndersonFuller}). In particular, if $R$ is a right self-injective unital ring, then $eR_{R}$ is injective for every $e \in \E(R)$. \end{remark}

Given two right modules $A$ and $B$ over the same ring, we write $A \lesssim B$ if $A$ is isomorphic to a submodule of $B$.

\begin{thm}[{\cite[Theorem~10.14]{GoodearlBook}}]\label{theorem:schroeder.bernstein} Let $R$ be a unital ring and let $A$ and $B$ be two injective right $R$-modules. If $A \lesssim B$ and $B \lesssim A$, then $A \cong B$. \end{thm}

A unital ring $R$ is called \emph{(von Neumann) regular}~\cite[II.II, Definition~2.2, p.~70]{VonNeumannBook} if \begin{displaymath}
	\forall a \in R \ \exists b \in R \colon \quad aba = a .
\end{displaymath}

\begin{remark}[{\cite[II.II, Theorem~2.2, p.~70]{VonNeumannBook}}]\label{remark:regular} Let $R$ be a unital ring. The following are equivalent. \begin{itemize}
	\item[---\,] $R$ is regular.
	\item[---\,] For every $a \in R$, there exists $e \in \E(R)$ such that $aR=eR$.
	\item[---\,] For every $a \in R$, there exists $e \in \E(R)$ such that $Ra=Re$.
\end{itemize} \end{remark}

We recollect some terminology concerning the structure theory of regular, right self-injective rings from~\cite[Chapter~10]{GoodearlBook}. A unital ring $R$ is called \emph{directly finite} if \begin{displaymath}
	\forall a,b \in R \colon \quad ab=1 \ \Longrightarrow \ ba=1 .
\end{displaymath} Given a regular ring $R$, an idempotent $e \in \E(R)$ is called \begin{itemize}
	\item[---\,] \emph{directly finite} if the ring $eRe$ is directly finite, 
	\item[---\,] \emph{abelian} if $\E(eRe) \subseteq \cent(eRe)$,
	\item[---\,] \emph{faithful} if $0$ is the only central idempotent in $R$ orthogonal to $e$.
\end{itemize} A regular, right self-injective ring $R$ is said to be of \begin{itemize}
	\item[---\,] \emph{Type I} if $R$ contains a faithful abelian idempotent, 
	\item[---\,] \emph{Type II} if $R$ contains a faithful directly finite idempotent, but $R$ contains no non-zero abelian idempotents,
	\item[---\,] \emph{Type III} if $R$ contains no non-zero directly finite idempotents.
\end{itemize} A regular, right self-injective ring $R$ is called \emph{purely infinite} if $R$ contains no non-zero directly finite central idempotents.

\begin{exmpl}[{\cite[Chapter~10, p.~116]{GoodearlBook}}]\label{example:rings} \begin{enumerate}
	\item\label{example:rings.1} If $V$ is any infinite-dimensional vector space over a division ring, then $\End(V)$ constitutes a purely infinite, regular, right self-injective ring of Type~I.
	\item\label{example:rings.2} If $S$ is any simple, right self-injective, unit-regular ring of Type II (such as given in~\cite[Example~10.7]{GoodearlBook}), then the endomorphism ring of the \emph{injective envelope}\footnote{cf.~\cite[Theorem~18.10]{AndersonFuller}} of the right $S$-module $S_{S}^{\aleph_{0}}$ is a purely infinite, regular, right self-injective ring of Type~II. %Alternative references: \cite[1, \S3D, (3.31)~Definition, p.~75]{lam}, \cite[\S1.B, Definition on p.~22]{GoodearlRingTheory}, \cite[\S4.2, Proposition~10, p.~92]{lambek} and~\cite[\S4.2, Proposition~9, p.~91]{lambek}.
	\item\label{example:rings.3} Any regular, right self-injective ring of Type~III (see~\cite[Example~10.11]{GoodearlBook} for a concrete instance) is purely infinite.
\end{enumerate} \end{exmpl}

A submodule $N$ of a module $M$ is called \emph{essential} and we write $N \less M$ if $N$ has non-zero intersection with every non-zero submodule of $M$.

\begin{thm}[{\cite[Theorem~10.16]{GoodearlBook}}]\label{theorem:purely.infinite} Let $R$ be a regular, right self-injective ring. The following are equivalent. \begin{itemize}
	\item[---\,] $R$ is purely infinite.
	\item[---\,] ${R_{R}^{\oplus n}} \lesssim {R_{R}}$ for some integer $n \geq 2$.
	\item[---\,] ${R_{R}^{\oplus n}} \cong {R_{R}}$ for every positive integer $n$.
	\item[---\,] $R_{R}^{\oplus \aleph_{0}}$ is isomorphic to an essential submodule of $R_{R}$.
\end{itemize} \end{thm}

\begin{proof} This is an equivalent reformulation of~\cite[Theorem~10.16]{GoodearlBook} using~\cite[Theorem~18.10]{AndersonFuller}. \end{proof} %Alternative references: \cite[1, \S3D, (3.30)~Theorem, p.~75]{lam}, \cite[\S1.B, Proposition~1.11, p.~22]{GoodearlRingTheory}, \cite[\S4.2, Proposition~10, p.~92]{lambek} and~\cite[\S4.2, Proposition~9, p.~91]{lambek}.

The reader is referred to~\cite{AraGoodearlPardo,ArandaPinoGoodearlPereraSilesMolina} for notions of pure infiniteness beyond the class of regular, right self-injective rings.

\section{Matrix decomposition}\label{section:matrix.decomposition}

This section revisits and refines a result of Menal and Moncasi~\cite[Theorem~2.2]{MenalMoncasi} concerning the decomposition of unit groups of matrix rings. Our central objective is the proof of Lemma~\ref{lemma:matrix.groups} below.

In order to clarify some general notation and terminology, let $R$ be a unital ring and let $n \in \N$. Then we let $\M_{n}(R)$ denote the unital ring of $n \times n$ matrices with entries in $R$, and we define $\GL_{n}(R) \defeq \GL(\M_{n}(R))$. Moreover, $s \in \M_{n}(R)$ is called a \emph{family of matrix units} if $s_{11} + \ldots + s_{nn} = 1$ and \begin{displaymath}
	s_{ij}s_{k \ell} \, = \, \begin{cases}
		\, s_{i\ell} & \text{if } j = k , \\
		\, 0 & \text{otherwise}
	\end{cases}
\end{displaymath} for all $i,j,k, \ell \in \{ 1,\ldots,n \}$.

\begin{remark}\label{remark:matrix.unit.isomorphism} Let $R$ be a unital ring, let $n \in \N_{>0}$, let $s \in \M_{n}(R)$ be a family of matrix units for $R$, and let $S \defeq s_{11}Rs_{11}$. Then \begin{displaymath}
	\M_{n}(S) \, \longrightarrow \, R , \quad a \, \longmapsto \, \sum\nolimits_{i,j=1}^{n} s_{i1}a_{ij}s_{1j}
\end{displaymath} constitutes a (necessarily unital) ring isomorphism (see, e.g.,~\cite[II.III, Proof of Theorem~3.3, pp.~99--100]{VonNeumannBook}) and, in particular, restricts to a group isomorphism from $\GL_{n}(S) = \GL(\M_{n}(S))$ to $\GL(R)$.\end{remark}

\begin{lem}\label{lemma:matrix.units} Let $R$ be a unital ring, let $n \in \N_{>0}$ and let $e_{1},\ldots,e_{n} \in \E(R)$ be such that $(e_{1},\ldots,e_{n})\perp$, $e_{1} + \ldots + e_{n} = 1$, and $e_{i} \in \{ ue_{1}u^{-1} \mid u \in \GL(R) \}$ for each $i \in \{2,\ldots,n \}$. Then there exists a family of matrix units $s \in \M_{n}(R)$ for $R$ such that $s_{ii} = e_{i}$ for each $i \in \{ 1,\ldots,n \}$. \end{lem}

\begin{proof} By assumption, we find $u_{2},\ldots,u_{n} \in \GL(R)$ such that $u_{i}e_{1}u_{i}^{-1} = e_{i}$ for each $i \in \{ 2,\ldots,n \}$. Let $u_{1} \defeq 1$. Then $s \in \M_{n}(R)$ defined by \begin{displaymath}
	s_{ij} \, \defeq \, u_{i}u_{j}^{-1}e_{j} \qquad (i,j \in \{ 1,\ldots,n \})
\end{displaymath} is a family of matrix units for $R$: indeed, if $i,j,k,\ell \in \{ 1,\ldots,n \}$ and $k \ne j$, then \begin{displaymath}
	s_{ij}s_{j\ell} \, = \, u_{i}u_{j}^{-1}e_{j}u_{j}u_{\ell}^{-1}e_{\ell} \, = \, u_{i}e_{1}u_{\ell}^{-1}e_{\ell} \, = \, u_{i}u_{\ell}^{-1}e_{\ell}e_{\ell} \, = \, u_{i}u_{\ell}^{-1}e_{\ell} \, = \, s_{i\ell}
\end{displaymath} and \begin{displaymath}
	s_{ij}s_{k\ell} \, = \, u_{i}u_{j}^{-1}e_{j}u_{k}u_{\ell}^{-1}e_{\ell} \, = \, u_{i}u_{j}^{-1}e_{j}u_{k}e_{1}u_{\ell}^{-1} \, = \, u_{i}u_{j}^{-1}e_{j}e_{k}u_{k}u_{\ell}^{-1} \, = \, 0 .
\end{displaymath} Evidently, $s_{ii} = e_{i}$ for each $i \in \{ 1,\ldots,n \}$. \end{proof}

We now turn to the work of Menal and Moncasi~\cite{MenalMoncasi}.

\begin{lem}[{\cite[Proof of Theorem~2.2, pp.~305--306]{MenalMoncasi}}]\label{lemma:menal.moncasi} Let $R$ be a purely infinite, regular, right self-injective ring. For every $n \in \N$ and for all $a_{1},\ldots,a_{n} \in R$ such that $a_{1}R + \ldots + a_{n}R = R$, there exist $e_{1},\ldots,e_{n} \in \E(\cent(R))$ such that \begin{itemize}
	\item[---\,] $(e_{1},\ldots,e_{n}) \perp$,
	\item[---\,] $e_{1}R + \ldots + e_{n}R = R$,
	\item[---\,] and $e_{i} \in \{ xa_{i}y \mid x,y \in R \}$ for each $i \in \{ 1,\ldots,n \}$.
\end{itemize} \end{lem}

\begin{remark}\label{remark:menal.moncasi} Let $R$ be a unital ring. If $a,b \in R$ and $ab = 1$, then \begin{displaymath}
	\begin{pmatrix} a & 0 \\ 1-ba & b \end{pmatrix} \begin{pmatrix} b & 1-ba \\ 0 & a \end{pmatrix} = \begin{pmatrix} 1 & 0 \\ 0 & 1 \end{pmatrix} \, = \, \begin{pmatrix} b & 1-ba \\ 0 & a \end{pmatrix}\begin{pmatrix} a & 0 \\ 1-ba & b \end{pmatrix}
\end{displaymath} and therefore $\begin{pmatrix} b & 1-ba \\ 0 & a \end{pmatrix} \in \GL_{2}(R)$. \end{remark}

\begin{lem}[cf.~{\cite[Proof of Theorem~2.2]{MenalMoncasi}}]\label{lemma:matrix.groups} Let $R$ be a unital ring.  Consider the subgroups \begin{displaymath}
	G_{1} \defeq \left. \! \left\{ \begin{pmatrix} a & 0 \\ 0 & 1 \end{pmatrix} \, \right\vert a \in \GL_{2}(R) \right\}, \quad G_{2} \defeq \left. \! \left\{ \begin{pmatrix} 1 & 0 \\ 0 & a \end{pmatrix} \, \right\vert a \in \GL_{2}(R) \right\} 
\end{displaymath} of $\GL_{3}(R)$. Then the following hold. \begin{enumerate}
	\item\label{lemma:matrix.groups.1} If $a \in \GL_{3}(R)$ and $1 \in \{ xa_{2i}y \mid x,y \in R \}$ for some $i \in \{ 1,2,3 \}$, then \begin{displaymath}
			\qquad a \, \in \, (G_{1}G_{2})^{3}G_{1} \cup (G_{2}G_{1})^{3}G_{2} .
		\end{displaymath}
	\item\label{lemma:matrix.groups.2} If $R$ is a purely infinite, regular, right self-injective ring, then \begin{displaymath}
			\qquad \GL_{3}(R) \, = \, (G_{1}G_{2})^{8} .
		\end{displaymath}
\end{enumerate} \end{lem}

\begin{proof} \ref{lemma:matrix.groups.1} The argument makes use of the involutions \begin{align*}
	u &\defeq \begin{pmatrix} 0 & 1 & 0 \\ 1 & 0 & 0 \\ 0 & 0 & 1 \end{pmatrix} \in \I(G_{1}) , \qquad v \defeq \begin{pmatrix} 1 & 0 & 0 \\ 0 & 0 & 1 \\ 0 & 1 & 0 \end{pmatrix} \in \I(G_{2}) , \\
	w &\defeq \begin{pmatrix} 0 & 0 & 1 \\ 0 & 1 & 0 \\ 1 & 0 & 0 \end{pmatrix} \in \I(\GL_{3}(R))
\end{align*} and proceeds by case analysis.
	
\textit{Case 1}: $1 \in \{ xa_{21}y \mid x,y \in R \}$. Let $x,y \in R$ be such that $xa_{21}y = 1$. Considering \begin{displaymath}
	g_{1} \defeq \begin{pmatrix} a_{21}y & 1-a_{21}yx & 0 \\ 0 & x & 0 \\ 0 & 0 & 1 \end{pmatrix} \stackrel{\ref{remark:menal.moncasi}}{\in} G_{1}, \qquad g_{2} \defeq \begin{pmatrix} y & 1-yxa_{21} & 0 \\ 0 & xa_{21} & 0 \\ 0 & 0 & 1 \end{pmatrix} \stackrel{\ref{remark:menal.moncasi}}{\in} G_{1} ,
\end{displaymath} we see that $(g_{1}ag_{2})_{21} = 1$. It follows that \begin{displaymath}
	ug_{1}ag_{2} \, = \, \begin{pmatrix} 1 & \begin{matrix} s_{12} & s_{13} \end{matrix} \\ \begin{matrix} s_{21} \\ s_{31} \end{matrix} & s' \end{pmatrix}
\end{displaymath} for suitable $s_{12},s_{13},s_{21},s_{31} \in R$ and $s' \in R^{2 \times 2}$. Using \begin{displaymath}
	g_{3} \defeq \begin{pmatrix} 1 & 0 & 0 \\ -s_{21} & 1 & 0 \\ 0 & 0 & 1 \end{pmatrix} \in G_{1}, \qquad g_{4} \defeq \begin{pmatrix} 1 & -s_{12} & 0 \\ 0 & 1 & 0 \\ 0 & 0 & 1 \end{pmatrix} \in G_{1} ,
\end{displaymath} we obtain that \begin{displaymath}
	g_{3}ug_{1}ag_{2}g_{4} \, = \, \begin{pmatrix} 1 & \begin{matrix} 0 & t_{13} \end{matrix} \\ \begin{matrix} 0 \\ t_{31} \end{matrix} & t' \end{pmatrix}
\end{displaymath} for suitable $t_{13},t_{31} \in R$ and $t' \in R^{2 \times 2}$. Thus, letting \begin{displaymath}
	g_{5} \defeq \begin{pmatrix} 1 & 0 & 0 \\ -t_{31} & 1 & 0 \\ 0 & 0 & 1 \end{pmatrix} \in G_{1}, \qquad g_{6} \defeq \begin{pmatrix} 1 & -t_{13} & 0 \\ 0 & 1 & 0 \\ 0 & 0 & 1 \end{pmatrix} \in G_{1} ,
\end{displaymath} one readily computes that \begin{displaymath}
	g_{5}vg_{3}ug_{1}ag_{2}g_{4}vg_{6} \, = \, \begin{pmatrix} 1 & \begin{matrix} 0 & 0 \end{matrix} \\ \begin{matrix} 0 \\ 0 \end{matrix} & h \end{pmatrix}
\end{displaymath} for some $h \in R^{2 \times 2}$. Since $g_{5}vg_{3}ug_{1}ag_{2}g_{4}vg_{6} \in \GL_{3}(R)$, we conclude that $h \in \GL_{2}(R)$, whence $g_{5}vg_{3}ug_{1}ag_{2}g_{4}vg_{6} \in G_{2}$. Consequently, \begin{displaymath}
	a \, \in \, (g_{3}ug_{1})^{-1}vg_{5}^{-1}G_{2}g_{6}^{-1}v(g_{2}g_{4})^{-1} \, \subseteq \, G_{1}G_{2}G_{1}G_{2}G_{1}G_{2}G_{1} \, = \, (G_{1}G_{2})^{3}G_{1} .
\end{displaymath}
	
\textit{Case 2}: $1 \in \{ xa_{22}y \mid x,y \in R \}$. Since $(au)_{21} = a_{22}$ and thus $au \in (G_{1}G_{2})^{3}G_{1}$ as established in Case~1, we see that \begin{displaymath}
	a \, \in \, (G_{1}G_{2})^{3}G_{1}u \, = \, (G_{1}G_{2})^{3}G_{1} .
\end{displaymath}
	
\textit{Case 3}: $1 \in \{ xa_{23}y \mid x,y \in R \}$. Note that $wG_{1} = G_{2}w$. Since $(waw)_{21} = a_{23}$ and therefore $waw \in (G_{1}G_{2})^{3}G_{1}$ due to Case~1, we conclude that \begin{displaymath}
	a \, \in \, w(G_{1}G_{2})^{3}G_{1}w \, = \, (G_{2}G_{1})^{3}G_{2}ww \, = \, (G_{2}G_{1})^{3}G_{2} \, = \, G_{2}(G_{1}G_{2})^{3} .
\end{displaymath}
	
\ref{lemma:matrix.groups.2} Suppose now that $R$ is a purely infinite, regular, right self-injective ring. Consider any $a \in \GL_{3}(R)$. Then $a_{21}R + a_{22}R + a_{23}R = R$, whence Lemma~\ref{lemma:menal.moncasi} asserts the existence of pairwise orthogonal elements $e_{1},e_{2},e_{3} \in \E(\cent(R))$ such that $e_{1}R + e_{2}R + e_{3}R = R$ and $e_{i} \in \{ xa_{2i}y \mid x,y \in R \}$ for each $i \in \{ 1,2,3 \}$. For each $i \in \{ 1,2,3 \}$, let us consider the unital ring $S_{i} \defeq e_{i}Re_{i}$ and the unital ring homomorphism $\pi_{i} \colon R \to S_{i}, \, x \mapsto e_{i}x$. Then \begin{displaymath}
	\phi \colon \, \M_{3}(R) \, \longrightarrow \, \prod\nolimits_{i=1}^{3} \M_{3}(S_{i}), \quad a \, \longmapsto \, (\pi_{1} \circ a, \pi_{2} \circ a, \pi_{3} \circ a)
\end{displaymath} is a (necessarily unital) ring isomorphism, thus restricts to a well-defined group isomorphism from $\GL(\M_{3}(R)) = \GL_{3}(R)$ to $\GL\bigl(\prod\nolimits_{i=1}^{3} \M_{3}(S_{i})\bigr) = \prod\nolimits_{i=1}^{3} \GL_{3}(S_{i})$. Moreover, letting \begin{displaymath}
	G_{i,1} \defeq \left. \! \left\{ \begin{pmatrix} a & 0 \\ 0 & 1 \end{pmatrix} \, \right\vert a \in \GL_{2}(S_{i}) \right\}, \quad G_{i,2} \defeq \left. \! \left\{ \begin{pmatrix} 1 & 0 \\ 0 & a \end{pmatrix} \, \right\vert a \in \GL_{2}(S_{i}) \right\} 
\end{displaymath} for each $i \in \{ 1,2,3 \}$, we observe that \begin{equation}\label{eq:range}
	\forall j \in \{ 1,2 \} \colon \qquad \phi(G_{j}) \, = \, \prod\nolimits_{i=1}^{3} G_{i,j} .
\end{equation} For each $i \in \{ 1,2,3 \}$, since $e_{i} \in \{ xa_{2i}y \mid x,y \in R \}$ and hence \begin{displaymath}
	e_{i} \, = \, \pi_{i}(e_{i}) \, \in \, \pi_{i}(\{ xa_{2i}y \mid x,y \in R \}) \, = \, \{ x\pi_{i}(a_{2i})y \mid x,y \in S_{i} \} \, = \, S_{i}(\pi_{i} \circ a)_{2i}S_{i} ,
\end{displaymath} item~\ref{lemma:matrix.groups.1} implies that ${\pi_{i} \circ a} \in (G_{i,1}G_{i,2})^{8}$. It follows that \begin{align*}
	\phi(a) \, \in \, \prod\nolimits_{i=1}^{3} (G_{i,1}G_{i,2})^{8} \, &= \, \left( \! \left(\prod\nolimits_{i=1}^{3} G_{i,1} \right)\! \left( \prod\nolimits_{i=1}^{3} G_{i,2}\right) \! \right)^{8} \\
	& \stackrel{\eqref{eq:range}}{=} \, (\phi(G_{1})\phi(G_{2}))^{8} \, = \, \phi\!\left( (G_{1}G_{2})^{8} \right) 
\end{align*} and so $a \in (G_{1}G_{2})^{8}$, as desired. \end{proof}

\section{Moieties}\label{section:moieties}

The following definition, which is fundamental to all our ring-theoretic results, is inspired by the corresponding set-theoretic notion defined and used for the study of infinite permutation groups in~\cite{neumann,MacPhersonNeumann,thomas,bergman,DrosteGoebel}, as well as its linear analogue explored in~\cite{macpherson,tolstykh}.

\begin{definition} Let $R$ be a unital ring. A \emph{moiety} in $R$ is an element $e \in \E(R)$ such that ${eR_{R}} \cong {R_{R}} \cong {(1-e)R_{R}}$. The set of all moieties in $R$ is denoted by $\Emo(R)$. \end{definition}

Any two moieties in a unital ring are conjugate:

\begin{lem}\label{lemma:conjugacy.of.moieties} Let $R$ be a unital ring and let $e,f \in \Emo(R)$. \begin{enumerate}
	\item\label{lemma:conjugacy.of.moieties.1} There exists $g \in \GL(R)$ such that $geg^{-1} = f$.
	\item\label{lemma:conjugacy.of.moieties.2} If $e \perp f$, then there exists $g \in \I(\Gamma_{R}(e+f))$ such that $geg = f$.
\end{enumerate} \end{lem}

\begin{proof} While~\ref{lemma:conjugacy.of.moieties.1} is an immediate consequence of Lemma~\ref{lemma:conjugacy.of.idempotents}, item~\ref{lemma:conjugacy.of.moieties.2} follows from~\ref{lemma:conjugacy.of.moieties.1} and Lemma~\ref{lemma:conjugacy.of.commuting.idempotents}. \end{proof}

The existence of a moiety characterizes the purely infinite ones among all regular, right self-injective rings, as the following consequence of Theorem~\ref{theorem:purely.infinite} shows.

\begin{prop}\label{proposition:pure.infinity.via.moieties} Let $R$ be a regular, right self-injective ring. The following are equivalent. \begin{enumerate}
	\item\label{proposition:pure.infinity.via.moieties.1} $R$ is purely infinite.
	\item\label{proposition:pure.infinity.via.moieties.2} $\Emo(R) \ne \emptyset$.
	\item\label{proposition:pure.infinity.via.moieties.3} For every $n \in \N_{>1}$, there exist $e_{1},\ldots,e_{n} \in \Emo(R)$ pairwise orthogonal such that $\sum_{i=1}^{n} e_{i} = 1$.
\end{enumerate} \end{prop}

\begin{proof} \ref{proposition:pure.infinity.via.moieties.1}$\Longrightarrow$\ref{proposition:pure.infinity.via.moieties.3}. Consider the ring isomorphism $\iota \colon R \to \End(R_{R}), \, a \mapsto [b \mapsto ab]$ from Remark~\ref{remark:cayley}\ref{remark:cayley.1}. Let $n \in \N_{>1}$. Since $R$ is purely infinite, Theorem~\ref{theorem:purely.infinite} asserts the existence of an isomorphism $\pi \colon {R_{R}^{\oplus n}} \to {R_{R}}$. For each $i \in \{ 1,\ldots,n \}$, considering the idempotent endomorphism \begin{displaymath}
	f_{i} \colon \,  {R_{R}^{\oplus n}} \, \longrightarrow \, {R_{R}^{\oplus n}}, \quad (x_{1},\ldots,x_{n}) \, \longmapsto \, (0,\ldots,0,x_{i},0,\ldots,0)
\end{displaymath} along with the corresponding element \begin{displaymath}
	e_{i} \, \defeq \, \iota^{-1}\!\left(\pi \circ {f_{i}} \circ {\pi^{-1}}\right) \, \in \, \E(R),
\end{displaymath} we see that \begin{displaymath}
	{e_{i}R_{R}} \, \cong \, {\left( \pi \circ {f_{i}} \circ {\pi^{-1}}\right)\!(R)_{R}} \, \cong \, {\pi\!\left(f_{i}\!\left(R^{\oplus n}\right)\right)_{R}} \, \cong \, {f_{i}\!\left(R^{\oplus n}\right)_{R}} \, \cong \, {R_{R}} .
\end{displaymath} Moreover, straightforward calculations reveal that $(f_{1},\ldots,f_{n}) \perp$ and $\sum_{i=1}^{n} f_{i} = 1$, which entails that $(e_{1},\ldots,e_{n}) \perp$ and $\sum_{i=1}^{n} e_{i} = 1$. Finally, for each $i \in \{ 1,\ldots,n \}$, since $n > 1$ and \begin{displaymath}
	{R_{R}} \, \cong \, {e_{j}R}_{R} \, \stackrel{{e_{i}}\perp {e_{j}}}{\leq} \, {(1-e_{i})R_{R}} \, \leq \, {R_{R}}
\end{displaymath} for any $j \in \{ 1,\ldots,n \} \setminus \{ i \}$, an application of Remark~\ref{remark:summands.of.injective.modules} and Theorem~\ref{theorem:schroeder.bernstein} yields that ${R_{R}} \cong {(1-e_{i})R_{R}}$, wherefore $e_{i} \in \Emo(R)$.
	
\ref{proposition:pure.infinity.via.moieties.3}$\Longrightarrow$\ref{proposition:pure.infinity.via.moieties.2}. This is trivial.

\ref{proposition:pure.infinity.via.moieties.2}$\Longrightarrow$\ref{proposition:pure.infinity.via.moieties.1}. If there exists some $e \in \Emo(R)$, then ${R_{R}} \oplus {R_{R}} \cong {eR_{R}} \oplus {(1-e)R_{R}} = {R_{R}}$ and hence $R$ is purely infinite by Theorem~\ref{theorem:purely.infinite}. \end{proof}

For later use, we record an obvious comparison criterion.

\begin{lem}\label{lemma:schroeder.bernstein.moieties} Let $R$ be a right self-injective unital ring. If $e_{0},e_{1} \in \Emo(R)$ and $f \in \E(R)$ are such that $e_{0} \leq f \leq e_{1}$, then $f \in \Emo(R)$. \end{lem}

\begin{proof} Let $e_{0},e_{1} \in \Emo(R)$ and $f \in \E(R)$ be such that $e_{0} \leq f \leq e_{1}$. Then \begin{align*}
	&{R_{R}} \, \cong \, {e_{0}R_{R}} \, \leq \, {fR_{R}} \, \leq \, {e_{1}R_{R}} \, \cong \, {R_{R}} , \\
	&{R_{R}} \, \cong \, {(1-e_{1})R_{R}} \, \leq \, {(1-f)R_{R}} \, \leq \, {(1-e_{0})R_{R}} \, \cong \, {R_{R}},
\end{align*} so ${fR_{R}} \cong {R_{R}} \cong {(1-f)R_{R}}$ by Remark~\ref{remark:summands.of.injective.modules} and Theorem~\ref{theorem:schroeder.bernstein}. Thus, $f \in \Emo(R)$. \end{proof}

The following decomposition lemma will be used to facilitate the application of Lemma~\ref{lemma:matrix.groups} in the proof of Lemma~\ref{lemma:decomposition.into.moieties}.

\begin{lem}\label{lemma:splitting.moieties} Let $R$ be a right self-injective unital ring. For every $e \in \Emo(R)$, there exist $e_{1},e_{2} \in \Emo(R)$ such that $e_{1} \perp e_{2}$ and $e_{1} + e_{2} = e$. \end{lem} 

\begin{proof} Let $e \in \Emo(R)$. In particular, ${R_{R}} \cong {eR_{R}}$ and so Remark~\ref{remark:algebraic.equivalence.of.idempotents} asserts the existence of $a,b \in R$ with $ab=1$ and $ba = e$. Considering $e_{1} \defeq bbaa$, we see that \begin{align*}
	e_{1}e_{1} \, &= \, bbaabbaa \, = \, bbabaa \, = \, bbaa \, = \, e_{1} , \\
	ee_{1} \, &= \, babbaa \, = \, bbaa \, = \, e_{1} , \\
	e_{1}e \, &= \, bbaaba \, = \, bbaa \, = \, e_{1} ,
\end{align*} i.e., $e_{1} \in \E(R)$ and $e_{1} \leq e$. Note that $aabb = ab = 1$ and $bbaa = e_{1}$, whence ${R_{R}} \cong {e_{1}R_{R}}$ by Remark~\ref{remark:algebraic.equivalence.of.idempotents}. From ${R_{R}} \cong {(1-e)R_{R}} \leq {(1-e_{1})R_{R}} \leq {R_{R}}$, we infer using Remark~\ref{remark:summands.of.injective.modules} and Theorem~\ref{theorem:schroeder.bernstein} that ${R_{R}} \cong {(1-e_{1})R_{R}}$. Therefore, $e_{1} \in \Emo(R)$. Furthermore, concerning $e_{2} \defeq e-e_{1} \in \E(R)$, we observe that $e_{1} \perp e_{2} \leq e$ and $e = e_{1}+e_{2}$. Since ${R_{R}} \cong {(1-e)R_{R}}$, according to Remark~\ref{remark:algebraic.equivalence.of.idempotents} we find $a',b' \in R$ such that $a'b'=1$ and $b'a' = 1-e$. In turn, $a'abb' = a'b' = 1$ and \begin{displaymath}
	bb'a'a \, = \, b(1-e)a \, = \, ba - bea \, = \, e - bbaa \, = \, e - e_{1} \, = \, e_{2} ,
\end{displaymath} thus ${R_{R}} \cong {e_{2}R_{R}}$ due to Remark~\ref{remark:algebraic.equivalence.of.idempotents}. Appealing to Remark~\ref{remark:summands.of.injective.modules} and Theorem~\ref{theorem:schroeder.bernstein} again, we deduce from ${R_{R}} \cong {(1-e)R_{R}} \leq {(1-e_{2})R_{R}} \leq {R_{R}}$ that ${R_{R}} \cong {(1-e_{2})R_{R}}$. Consequently, $e_{2} \in \Emo(R)$ as well. \end{proof}

We arrive at the two main insights of the current section.

\begin{lem}\label{lemma:decomposition.into.moieties} Let $R$ be a regular, right self-injective ring, and let $e_{1},e_{2},e_{3} \in \Emo(R)$ be pairwise orthogonal with $e_{1} + e_{2} + e_{3} = 1$. Then \begin{displaymath}
	\GL(R) \, = \, (\Gamma_{R}(e_{1}+e_{2})\Gamma_{R}(e_{2}+e_{3}))^{8} .
\end{displaymath} \end{lem}

\begin{proof} Note that $R$ is purely infinite according to Proposition~\ref{proposition:pure.infinity.via.moieties}. By Lemma~\ref{lemma:conjugacy.of.moieties}\ref{lemma:conjugacy.of.moieties.1} and Lemma~\ref{lemma:matrix.units}, we find a family of matrix units $s \in \M_{3}(R)$ for $R$ such that $s_{ii} = e_{i}$ for each $i \in \{ 1,2,3 \}$. Since $R_{R} \cong e_{1}R_{R}$, we see that \begin{displaymath}
	R \, \stackrel{\ref{remark:cayley}\ref{remark:cayley.1}}{\cong} \, \End(R_{R}) \, \cong \, \End(e_{1}R_{R}) \, \stackrel{\ref{remark:cayley}\ref{remark:cayley.1}}{\cong} \, e_{1}Re_{1} \, \eqdef \, S .
\end{displaymath} In particular, $S$ is a purely infinite, regular, right self-injective ring. By Remark~\ref{remark:matrix.unit.isomorphism}, \begin{displaymath}
	\phi \colon \, \M_{3}(S) \, \longrightarrow \, R, \quad a \, \longmapsto \, \sum\nolimits_{i,j=1}^{3} s_{i1}a_{ij}s_{1j}
\end{displaymath} is a (necessarily unital) ring isomorphism, thus restricts to a  group isomorphism from $\GL(\M_{3}(S)) = \GL_{3}(S)$ to $\GL(R)$. Considering \begin{displaymath}
	G_{1} \defeq \left. \! \left\{ \begin{pmatrix} a & 0 \\ 0 & 1 \end{pmatrix} \, \right\vert a \in \GL_{2}(S) \right\}, \quad G_{2} \defeq \left. \! \left\{ \begin{pmatrix} 1 & 0 \\ 0 & a \end{pmatrix} \, \right\vert a \in \GL_{2}(S) \right\} ,
\end{displaymath} we observe that $\phi(G_{1}) = \Gamma_{R}(e_{1}+e_{2})$ and $\phi(G_{2}) = \Gamma_{R}(e_{2}+e_{3})$. Hence, \begin{align*}
	\GL(R) \, &= \, \phi(\GL_{3}(S)) \, \stackrel{\ref{lemma:matrix.groups}\ref{lemma:matrix.groups.2}}{=} \, \phi\!\left((G_{1}G_{2})^{8}\right) \\
		& = \, (\phi(G_{1})\phi(G_{2}))^{8} \, = \, (\Gamma_{R}(e_{1}+e_{2})\Gamma_{R}(e_{2}+e_{3}))^{8} . \qedhere
\end{align*} \end{proof}

\begin{lem}\label{lemma:bounded.generation} Let $R$ be a regular, right self-injective ring. For every $e \in \Emo(R)$, there exists $g \in \I(R)$ such that $1-e \leq geg$ and \begin{displaymath}
	\GL(R) \, = \, (\Gamma_{R}(e)g\Gamma_{R}(e)g)^{8} .
\end{displaymath} \end{lem}

\begin{proof} Let $e \in \Emo(R)$. By Lemma~\ref{lemma:splitting.moieties}, there exist $e_{1},e_{2} \in \Emo(R)$ with $e_{1} \perp e_{2}$ and $e_{1} + e_{2} = e$. Note that $e_{3} \defeq 1-e \in \Emo(R)$ and $e_{1} \perp e_{3}$. Hence, by Lemma~\ref{lemma:conjugacy.of.moieties}\ref{lemma:conjugacy.of.moieties.2}, we find $g \in \I(\Gamma_{R}(e_{1}+e_{3}))$ such that $e_{3} = ge_{1}g \leq geg$. It follows that \begin{displaymath}
	geg \, = \, ge_{1}g + ge_{2}g \, = \, e_{3} + g(1-(e_{1}+e_{3}))g \, = \, e_{3} + 1-(e_{1}+e_{3}) \, = \, 1-e_{1} \, = \, e_{2}+e_{3}
\end{displaymath} and, in turn, \begin{displaymath}
	g\Gamma_{R}(e)g \, \stackrel{\ref{lemma:corner.group}\ref{lemma:corner.group.3}}{=} \, \Gamma_{R}(geg) \, = \, \Gamma_{R}(e_{2}+e_{3}) .
\end{displaymath} Since $e_{1} + e_{2} + e_{3} = 1$, we conclude that \begin{displaymath}
	\GL(R) \, \stackrel{\ref{lemma:decomposition.into.moieties}}{=} \, (\Gamma_{R}(e_{1}+e_{2})\Gamma_{R}(e_{2}+e_{3}))^{8} \, = \, (\Gamma_{R}(e)g\Gamma_{R}(e)g)^{8} . \qedhere
\end{displaymath} \end{proof}

We introduce three special classes of units defined by their interaction with moieties.

\begin{definition}\label{definition:moietious.involutions} If $R$ is a unital ring, then we let \begin{align*}
	\IM(R) \, &\defeq \, \{ g \in \I(R) \mid \exists e \in \E(R)\colon \, {R_{R}} \cong {eR_{R}}, \, geg = 1-e \} , \\
	\IME(R) \, &\defeq \, \{ g \in \I(R) \mid \exists e \in \E(R) \colon \, {R_{R}} \cong {eR_{R}} , \, e \perp geg, \, g \in \Gamma_{R}(e+geg) \} , \\
	\GLM(R) \, &\defeq \, \left\{ g \in \GL(R) \left\vert \, \exists e \in \E(R) \colon \, {R_{R}} \cong {eR_{R}}, \, e \perp geg^{-1} \right\} . \right.
\end{align*} \end{definition}

\begin{remark}\label{remark:moietious.units} $\IM(R) \subseteq \IME(R) \subseteq \GLM(R)$ for any unital ring $R$. \end{remark}

\begin{lem}\label{lemma:characterization.of.ime} Let $R$ be a right self-injective unital ring. Then \begin{displaymath}
	\IME(R) \, = \, \bigcup \{ \IM(eRe)+1-e \mid e \in \E(R), \, {R_{R}} \cong {eR_{R}} \} .
\end{displaymath} \end{lem}

\begin{proof} ($\supseteq$): Let $g \in \IM(eRe)+1-e$ for some $e \in \E(R)$ with ${R_{R}} \cong {eR_{R}}$. In particular, $g \in \I(R)$ by Lemma~\ref{lemma:corner.group}. Let $h \in \IM(eRe)$ be such that $g = h+1-e$. Considering the unital ring $S \defeq eRe$, we find some $f \in \E(S)$ such that ${S_{S}} \cong {fS_{S}}$ and $hfh = e-f$. Since ${eS_{S}} = {S_{S}}$, the former implies that \begin{displaymath}
	{R_{R}} \, \cong \, {eR_{R}} \, \stackrel{\ref{lemma:ideal.isomorphism}}{\cong} \, {fR_{R}}
\end{displaymath} and hence ${R_{R}} \cong {fR_{R}}$. Furthermore, \begin{align}
	gf \, &= \, (h+1-e)f \, \stackrel{f \leq e}{=} \, hf, \label{eq:detail.1} \\
	fg \, &= \, f(h+1-e) \, \stackrel{f \leq e}{=} \, fh \, = \, h(e-f) , \label{eq:detail.2} \\
	gfg \, &\stackrel{\eqref{eq:detail.1}}{=} \, hfg \, \stackrel{\eqref{eq:detail.2}}{=} \, hfh \, = \, e-f . \label{eq:detail.3}
\end{align} Using~\eqref{eq:detail.3}, we conclude that $gfg = e-f \perp f$. Moreover, \begin{align*}
	g(1-f-gfg) \, &= \, g-gf-fg \, \stackrel{\eqref{eq:detail.1}+\eqref{eq:detail.2}}{=} \, g - hf - h(e-f) \\
	& = \, g-he \, = \, g-h \, = \, 1-e \, \stackrel{\eqref{eq:detail.3}}{=} \, 1-f-gfg
\end{align*} and thus also \begin{displaymath}
	(1-f-gfg)g \, = \, g-fg-gf \, = \, g(1-f-gfg) \, = \, 1-f-gfg .
\end{displaymath} Hence, $g \in \Gamma_{R}(f+gfg)$ and therefore $g \in \IME(R)$.
	
($\subseteq$): Let $g \in \I(R)$ and $e \in \E(R)$ be such that ${R_{R}} \cong {eR_{R}}$, $e \perp geg$ and $g \in \Gamma_{R}(f)$ for $f \defeq e+geg$. Then \begin{displaymath}
	{R_{R}} \, \cong \, {eR_{R}} \, \leq \, {fR_{R}} \, \leq \, {R_{R}}	
\end{displaymath} and therefore ${R_{R}} \cong {fR_{R}}$ due to Remark~\ref{remark:summands.of.injective.modules} and Theorem~\ref{theorem:schroeder.bernstein}. Consider the unital ring $S \defeq fRf$. From Lemma~\ref{lemma:corner.group}, we know that $h \defeq gf \in \I(S)$. Also, $e,geg \in \E(S)$ and $heh = geg = f-e$. Moreover, using Lemma~\ref{lemma:ideal.isomorphism}, from ${fR_{R}} \cong {R_{R}} \cong {eR_{R}}$ we infer that ${S_{S}} = {fS_{S}} \cong {eS_{S}}$. Hence, $h \in \IM(S)$ and thus \begin{displaymath}
	g \, = \, gf + 1 - f \, = \, h+1-f \, \in \, \IM(S)+1-f \, \subseteq \, \IME(R) . \qedhere
\end{displaymath} \end{proof}

\begin{prop}\label{proposition:purely.infinite.involutions} Let $R$ be a regular, right self-injective ring. Then the following are equivalent. \begin{enumerate}
	\item\label{proposition:purely.infinite.involutions.1} $R$ is purely infinite.
	\item\label{proposition:purely.infinite.involutions.2} $\IM(R) \ne \emptyset$. 
	\item\label{proposition:purely.infinite.involutions.3} $\IME(R) \ne \emptyset$. 
	\item\label{proposition:purely.infinite.involutions.4} $\GLM(R) \ne \emptyset$.
\end{enumerate} \end{prop}

\begin{proof} \ref{proposition:purely.infinite.involutions.1}$\Longrightarrow$\ref{proposition:purely.infinite.involutions.2}. Suppose that $R$ is purely infinite. According to Proposition~\ref{proposition:pure.infinity.via.moieties}, we find some $e \in \Emo(R)$. Then also $1-e \in \Emo(R)$, and so Lemma~\ref{lemma:conjugacy.of.moieties}\ref{lemma:conjugacy.of.moieties.2} asserts the existence of $g \in \I(R)$ such that $geg = 1-e$. In turn, $g \in \IM(R)$.
	
\ref{proposition:purely.infinite.involutions.2}$\Longrightarrow$\ref{proposition:purely.infinite.involutions.3}$\Longrightarrow$\ref{proposition:purely.infinite.involutions.4}. In view of Remark~\ref{remark:moietious.units}, these implications are obvious.
	
\ref{proposition:purely.infinite.involutions.4}$\Longrightarrow$\ref{proposition:purely.infinite.involutions.1}. Assume that $\GLM(R) \ne \emptyset$. Then there exist $g \in \GL(R)$ and $e \in \E(R)$ such that ${R_{R}} \cong {eR_{R}}$ and $e \perp geg^{-1}$. Then \begin{displaymath}
	{R_{R}} \, \cong \, {eR_{R}} \, \cong \, {geR_{R}} \, = \, {geg^{-1}R_{R}} \, \leq \, {(1-e)R_{R}} \, \leq \, {R_{R}}
\end{displaymath} and hence ${R_{R}} \cong {(1-e)R_{R}}$ by Remark~\ref{remark:summands.of.injective.modules} and Theorem~\ref{theorem:schroeder.bernstein}, wherefore $e \in \Emo(R)$. Thus, $R$ is purely infinite by Proposition~\ref{proposition:pure.infinity.via.moieties}. \end{proof}

The remainder of this section collects some basic observations concerning conjugacy classes of the most restrictive type of involutions introduced in Definition~\ref{definition:moietious.involutions}. A more comprehensive study, in particular with regard to bounded normal generation, is the subject of Section~\ref{section:bng}.

\begin{lem}\label{lemma:conjugacy.of.moietious.involutions} Let $R$ be a unital ring. For any two elements $g,h \in \IM(R)$, there exists $w \in \GL(R)$ such that $wgw^{-1} = h$. \end{lem}

\begin{proof} Let $g,h \in \IM(R)$. Then there exist $e,f \in \Emo(R)$ such that $geg = 1-e$ and $hfh = 1-f$. By Lemma~\ref{lemma:conjugacy.of.moieties}\ref{lemma:conjugacy.of.moieties.1}, there exists $u \in \GL(R)$ such that $ueu^{-1} = f$. Note that $u^{-1}hu \in \I(R)$ and \begin{align}\label{eq:transformer}
	u^{-1}hue \, &=\, u^{-1}hfu \, = \, u^{-1}(1-f)hu \, = \, \left(u^{-1}-u^{-1}f\right)\!hu \nonumber \\
		&= \, \left(u^{-1}-eu^{-1}\right)\!hu \, = \, (1-e)u^{-1}hu .
\end{align} In turn, \begin{displaymath}
	u^{-1}huge \, = \, u^{-1}hu(1-e)g \, \stackrel{\eqref{eq:transformer}}{=} \, eu^{-1}hug
\end{displaymath} and therefore, by Remark~\ref{remark:cayley}\ref{remark:cayley.2} and Lemma~\ref{lemma:corner.group}, \begin{displaymath}
	v \, \defeq \, u^{-1}huge + 1-e \, \in \, \Gamma_{R}(e) \, \subseteq \, \GL(R)
\end{displaymath} with $v^{-1} = gu^{-1}hue + 1-e$. We conclude that \begin{align*}
	vgv^{-1} \, &= \, \left( u^{-1}huge + 1-e \right) \! g \! \left( gu^{-1}hue + 1-e\right) \\
		& = \, \left( u^{-1}huge + 1-e \right) \! \left( u^{-1}hue + g(1-e) \right) \\
		& \stackrel{\eqref{eq:transformer}}{=} \, \left( u^{-1}huge + 1-e \right) \! \left( (1-e)u^{-1}hu + eg \right) \\
		& = \, u^{-1}hugeg + (1-e)u^{-1}hu \, \stackrel{\eqref{eq:transformer}}{=} \, u^{-1}hugg(1-e) + u^{-1}hue \\
		& = \, u^{-1}hu(1-e) + u^{-1}hue \, = \, u^{-1}hu
\end{align*} and hence, for $w \defeq uv \in \GL(R)$, \begin{displaymath}
	wgw^{-1} \, = \, uvgv^{-1}u^{-1} \, = \, uu^{-1}huu^{-1} \, = \, h . \qedhere
\end{displaymath} \end{proof}

\begin{lem}\label{lemma:splits.over.moieties} Let $R$ be a unital ring. Moreover, let $e \in \E(R)$ and $g \in \I(R)$ be such that $geg = 1-e$. Then \begin{displaymath}
	\left. \left\{ ae + ga^{-1}g(1-e) \, \right\vert a \in \GL(R), \, ae = ea \right\} \, \subseteq \, \Cl_{\GL(R)}(g)^{2} .
\end{displaymath} \end{lem}

\begin{proof} If $a \in \GL(R)$ with $ae = ea$, then $b \defeq ae + 1-e \in \Gamma_{R}(e)$ and $b^{-1} = a^{-1}e+1-e$ by Remark~\ref{remark:cayley}\ref{remark:cayley.2} and Lemma~\ref{lemma:corner.group}, thus \begin{displaymath}
	bgb^{-1}ge \, = \, bgb^{-1}(1-e)g \, = \, bg(1-e)g \, = \, begg \, = \, ae
\end{displaymath} and \begin{align*}
	bgb^{-1}g(1-e) \, &= \, bgb^{-1}eg \, = \, bga^{-1}eg \, = \, bgea^{-1}g \, = \, b(1-e)ga^{-1}g \\
		& = \, (1-e)ga^{-1}g \, = \, gea^{-1}g \, = \, ga^{-1}eg \, = \, ga^{-1}g(1-e) ,
\end{align*} wherefore \begin{displaymath}
	ae + ga^{-1}g(1-e) \, = \, bgb^{-1}ge + bgb^{-1}g(1-e) \, = \, bgb^{-1}g \, \in \, \Cl_{\GL(R)}(g)^{2} . \qedhere
\end{displaymath}\end{proof}

\begin{lem}\label{lemma:splits.over.moieties.addon} Let $R$ be a right self-injective unital ring, let $e_{1},e_{2},e_{3} \in \Emo(R)$ be such that $(e_{1},e_{2},e_{3})\perp$ and $e_{1} + e_{2} + e_{3} = 1$, and let $g \in \I(R)$ be such that $ge_{1}g = e_{2}$. Then there exists $h \in \IM(R)$ such that \begin{displaymath}
	\left. \left\{ ae_{1} + ga^{-1}ge_{2} + e_{3} \, \right\vert a \in \GL(R), \, ae_{1} = e_{1}a \right\} \, \subseteq \, \Cl_{\GL(R)}(h)^{2} .
\end{displaymath} \end{lem}

\begin{proof} By Lemma~\ref{lemma:splitting.moieties}, we find $e_{1}',e_{2}' \in \Emo(R)$ such that $e_{1}' \perp e_{2}'$ and $e_{3} = e_{1}' + e_{2}'$. Then $(e_{1},e_{2},e_{1}',e_{2}')\perp$ and $e_{1} + e_{2} + e_{1}' + e_{2}' = 1$. Consider $f \defeq e_{1}+e_{1}' \in \E(R)$. Since $e_{1} \leq f \leq 1-e_{2}$, we see that $f \in \Emo(R)$ by Lemma~\ref{lemma:schroeder.bernstein.moieties}. Furthermore, according to Lemma~\ref{lemma:conjugacy.of.moieties}\ref{lemma:conjugacy.of.moieties.2}, there exists $g' \in \I(R)$ such that $g'e_{1}'g' = e_{2}'$. Define \begin{displaymath}
	h \, \defeq \, g(e_{1}+e_{2}) + g'(e_{1}'+e_{2}') \, = \, ge_{1} + ge_{2} + g'e_{1}' + g'e_{2}' \, = \, e_{2}g + e_{1}g + e_{2}'g' + e_{1}'g'	
\end{displaymath} and observe that \begin{align*}
	h^{2} \, &= \, (ge_{1} + ge_{2} + g'e_{1}' + g'e_{2}')(e_{2}g + e_{1}g + e_{2}'g' + e_{1}'g') \\
	& = \, ge_{1}g + ge_{2}g + g'e_{1}'g' + g'e_{2}'g' \, = \, e_{2}+e_{1}+e_{2}'+e_{1}' \, = \, 1,
\end{align*} thus $h \in \I(R)$. Moreover, \begin{align*}
	hfh \, &= \, (ge_{1} + ge_{2} + g'e_{1}' + g'e_{2}')(e_{1}+e_{1}')(e_{2}g + e_{1}g + e_{2}'g' + e_{1}'g') \\
	& = \, ge_{1}g + g'e_{1}'g' \, = \, e_{2}+e_{2}' \, = \, 1-f 
\end{align*} and hence $h \in \IM(R)$. Now, if $a \in \GL(R)$ and $ae_{1} = e_{1}a$, then \begin{displaymath}
	b \, \defeq \, ae_{1} + 1-e_{1} \, \in \, \Gamma_{R}(e_{1})
\end{displaymath} and $b^{-1} = a^{-1}e_{1}+1-e_{1}$ by Remark~\ref{remark:cayley}\ref{remark:cayley.2} and Lemma~\ref{lemma:corner.group}, moreover \begin{align*}
	&e_{1}' \, = \, (1-e_{1})e_{1}' \, = \, b(1-e_{1})e_{1}' \, = \, be_{1}', \\
	&ga^{-1}ge_{2} \, = \, ga^{-1}e_{1}g \, = \, ge_{1}b^{-1}g \, = \, he_{1}b^{-1}g \, = \, hb^{-1}e_{1}g \, = \, hb^{-1}ge_{2} \, = \, hb^{-1}he_{2}, \\
	&e_{2}' \, = \, g'e_{1}'g' \, = \, he_{1}'g' \, = \, hb^{-1}e_{1}'g' \, = \, hb^{-1}g'e_{2}' \, = \, hb^{-1}he_{2}',
\end{align*} and consequently \begin{align*}
	ae_{1} + ga^{-1}ge_{2} + e_{3} \, &= \, ae_{1} + ga^{-1}ge_{2} + e_{1}' + e_{2}' \, = \, be_{1} + be_{1}' + hb^{-1}he_{2} + hb^{-1}he_{2}' \\
	& = \, bf + hb^{-1}h(1-f) \, \stackrel{\ref{lemma:splits.over.moieties}}{\in} \, \Cl_{\GL(R)}(h)^{2} . \qedhere
\end{align*} \end{proof}

\section{Essential embeddings}\label{section:essential.embeddings}

The main objective of this section is the proof of (the slightly technical) Lemma~\ref{lemma:essential.embedding}, which in combination with Lemma~\ref{lemma:galvin} immediately leads to Theorem~\ref{theorem:universal.sequence} and which moreover serves as the basis for a number of structural results presented in subsequent sections. We start off with two basic preparatory lemmata. Given a unital ring $R$, let us consider the set \begin{displaymath}
	\lat(R) \, \defeq \, \{ aR \mid a \in R \}
\end{displaymath} of principal right ideals of $R$.

\begin{lem}\label{lemma:essential.inclusion} Let $R$ be a regular ring. \begin{enumerate}
	\item\label{lemma:essential.inclusion.1} If $I,J \in \lat(R)$ and $I \less J_{R}$, then $I = J$.
	\item\label{lemma:essential.inclusion.2} Let $a,b \in R$. If there exists $I \less R_{R}$ with $ax = bx$ for all $x \in I$, then $a = b$.
\end{enumerate} \end{lem}

\begin{proof} \ref{lemma:essential.inclusion.1} Let $I,J \in \lat(R)$ with $I \less J_{R}$. Thanks to~\cite[II.II, Theorem~2.4, p.~72]{VonNeumannBook} and~\cite[I.I, Theorem~1.3, p.~5]{VonNeumannBook} (see also~\cite[VIII.1, Theorem~1]{BirkhoffBook}), there exists $I' \in \lat(R)$ such that $I \oplus {I'} = J$. Since $I \less J_{R}$, it follows that $I' = \{ 0 \}$ and hence $I = J$.

\ref{lemma:essential.inclusion.2} By Remark~\ref{remark:regular}, we find $e \in \E(R)$ such that $R(a-b)= Re$. In turn, \begin{align*}
	J \, &\defeq \, \{ x \in R \mid ax = bx \} \, = \, \{ x \in R \mid (a-b)x = 0\} \, = \, \{ x \in R \mid ex = 0 \} \\
	& = \, \{ x \in R \mid (1-e)x = x \} \, = \, (1-e)R \, \in \, \lat(R).
\end{align*} Thus, if $J \less {R_{R}}$, then $J=R$ by~\ref{lemma:essential.inclusion.1}. This gives the desired conclusion. \end{proof}

\begin{lem}\label{lemma:essential.new} Let $R$ be a regular, right self-injective ring and let $I \less R_{R}$. Then, for each $a \in \End(I_{R})$, there exists a unique $\bar{a} \in R$ such that \begin{displaymath}
	\forall x \in I \colon \qquad \bar{a}x \, = \, a(x) .
\end{displaymath} The resulting map $\phi \colon \End(I_{R}) \to R, \, a \mapsto \bar{a}$ is a unital ring embedding. \end{lem}

\begin{proof} Let $a \in \End(I_{R})$. Since $R_{R}$ is self-injective, there exists $\phi_{a} \in \End(R_{R})$ such that ${\phi_{a}}(x) = a(x)$ for all $x \in I$. According to Remark~\ref{remark:cayley}\ref{remark:cayley.1}, we find $\bar{a} \in R$ such that \begin{displaymath}
	\forall x \in R \colon \qquad \phi_{a}(x) \, = \, \bar{a}x .
\end{displaymath} We conclude that $\bar{a}x = \phi_{a}(x) = a(x)$ for every $x \in I$. Moreover, if $b \in R$ is such that $bx = a(x)$ for each $x \in I$, then $\bar{a}x = a(x) = bx$ for every $x \in I$, whence $\bar{a} = b$ by Lemma~\ref{lemma:essential.inclusion}\ref{lemma:essential.inclusion.2}. This proves the first part.
	
For the second part, note that $1_{R}x = x = 1_{\End(I_{R})}(x)$ for all $x \in I$, so $\overline{1_{\End(I_{R})}} = 1_{R}$. Furthermore, if $a,b \in \End(I_{R})$, then \begin{align*}
	&(\bar{a}+\bar{b})x \, = \, \bar{a}x + \bar{b}x \, = \, a(x) + b(x) \, = \, (a+b)(x) , \\
	&\bar{a}\bar{b}x \, = \, \bar{a}b(x) \, = \, a(b(x)) \, = \, (ab)(x) 
\end{align*} for all $x \in I$, so that $\overline{a+b} = \bar{a}+\bar{b}$ and $\overline{ab} = \bar{a}\bar{b}$. Finally, if $a \in \End(I_{R})$ and $\bar{a} = 0$, then $a(x) = \bar{a}x = 0$ for all $x \in I$, hence $a=0$. \end{proof}

We turn to the central insight of this section.

\begin{lem}\label{lemma:essential.embedding} Let $R$ be a regular, right self-injective ring, and define $G \defeq \GL(R)$. Let $X$ be a set with $\vert X \vert > 1$, let $M_{R} \defeq R_{R}^{\oplus X}$, and let $\eta \colon M_{R} \to R_{R}$ be an embedding such that $\eta(M) \less R_{R}$. Consider the unital ring embedding\footnote{cf.~Lemma~\ref{lemma:essential.new}} \begin{displaymath}
	\phi \colon \, \End(M_{R}) \, \longrightarrow \, R, \quad a \, \longmapsto \, \overline{\eta \circ a \circ {\eta^{-1}}} .
\end{displaymath} For each $Y \subseteq X$, consider $p_{Y} \in \End(M_{R})$ defined by \begin{displaymath}
	p_{Y}(z)(x) \, \defeq \, \begin{cases}
			\, z(x) & \text{if } x \in Y , \\
			\, 0 & \text{otherwise}
		\end{cases} \qquad (z \in M, \, x \in X)
\end{displaymath} and let $e_{Y} \defeq \phi(p_{Y})$. Consider the homomorphism\footnote{Note that $\psi$ is an embedding if and only if $R$ is non-zero.} $\psi \colon G(X) \to \GL(M_{R})$ given by \begin{displaymath}
	\psi(g)z \, \defeq \, {g_{\bullet}} \cdot {\left(z \circ {g_{\circ}^{-1}}\right)} \qquad (g \in G(X), \, z \in M) 
\end{displaymath} and the homomorphism $\iota \defeq {\phi\vert_{\GL(M_{R})}} \circ \psi \colon G(X) \to G$. Then the following hold. \begin{enumerate}
	\item\label{lemma:essential.embedding.1} $e_{Y} \in \Emo(R)$ for every proper, non-empty subset $Y \subseteq X$.
	\item\label{lemma:essential.embedding.3} $\iota(G(X,Y)) \subseteq \Gamma_{R}(e_{Y})$ for every $Y \subseteq X$.
	\item\label{lemma:essential.embedding.2} $\iota(G(X,\{x\})) = \Gamma_{R}(e_{\{x\}})$ for every $x \in X$.
	\item\label{lemma:essential.embedding.4} If $g \in \Sym(X)\setminus \{ 1\}$, then $\iota(1,g) \in \GLM(R)$.
	\item\label{lemma:essential.embedding.6} If $u \in \I(\Sym(X))\setminus \{ 1 \}$, then $\iota(1,u) \in \IME(R)$ and $\iota(G(u)) \subseteq \Cl_{G}(h)^{2}$ for every $h \in \IM(R)$.
	\item\label{lemma:essential.embedding.5} For every $x \in X$, there exists $h \in \I(G)$ such that \begin{displaymath}
							\qquad G \, = \, (\iota(G(X,\{x\}))h\iota(G(X,\{x\}))h)^{8}
					\end{displaymath} and $h\iota(G(X,X\setminus \{x\}))h \subseteq \iota(G(X,\{x\}))$.
	\item\label{lemma:essential.embedding.7} If $X$ is infinite, then $\iota(G(X,\{x\})) \subseteq \Cl_{G}(h)^{4}$ for all $x \in X$ and $h \in \IM(R)$.
\end{enumerate} \end{lem}

\begin{proof} First of all, for every $Y \subseteq X$, one readily checks that $p_{Y} \in \E(\End(M_{R}))$ and thus $e_{Y} \in \E(R)$ due to $\phi$ being a ring homomorphism. Moreover, for every $Y \subseteq X$, \begin{equation}\label{eq:orthogonality}
	e_{X\setminus Y} \, = \, \phi(p_{X\setminus Y}) \, = \, \phi(1-p_{Y}) \, = \, 1-\phi(p_{Y}) \, = \, 1-e_{Y} .	
\end{equation} We also note that, if $Y \subseteq X$, then \begin{displaymath}
	\eta(p_{Y}(z)) \, = \, \left(\eta \circ {p_{Y}} \circ \eta^{-1}\right)\!(\eta(z)) \, \stackrel{\ref{lemma:essential.new}}{=} \, \phi(p_{Y})\eta(z) \, = \, e_{Y}\eta(z) \, \in \, e_{Y}R
\end{displaymath} for every $z \in M$. That is, \begin{equation}\label{eq:essential.equality.0}
	\forall Y \subseteq X \colon \qquad \eta(p_{Y}(M)) \, \subseteq \, e_{Y}R .
\end{equation}

\ref{lemma:essential.embedding.1} We claim that \begin{equation}\label{eq:essential.equality.1}
	\forall x \in X \colon \qquad \eta(p_{\{x\}}(M)) \, = \, e_{\{x\}}R .
\end{equation} To prove this, let $x \in X$. Considering the element $\hat{x} \in M$ defined by \begin{displaymath}
	\hat{x}(y) \, \defeq \, \begin{cases}
			\, 1 & \text{if } y = x, \\
			\, 0 & \text{if } y \in X \setminus \{x\} ,
		\end{cases}
\end{displaymath} we observe that $p_{\{x\}}(M) = \hat{x}R$ and hence \begin{displaymath}
	\eta(p_{\{x\}}(M)) \, = \, \eta(\hat{x}R) \, = \, \eta(\hat{x})R \, \in \, \lat(R) .
\end{displaymath} Furthermore, using~\eqref{eq:orthogonality}, we see that \begin{displaymath}
	\eta(p_{\{x\}}(M)) \oplus \eta(p_{X\setminus \{x\}}(M)) \, = \, \eta(M) \, \less \, {R_{R}} \, = \, {e_{\{x\}}R_{R}} \oplus {e_{X\setminus \{x\}}R_{R}} 
\end{displaymath} and so $\eta(p_{\{x\}}(M)) \less {e_{\{x\}}R_{R}}$ by~\eqref{eq:essential.equality.0} and~\cite[Proposition~5.20(2)]{AndersonFuller}. Therefore, $\eta(p_{\{x\}}(M)) = e_{\{x\}}R$ according to Lemma~\ref{lemma:essential.inclusion}\ref{lemma:essential.inclusion.1}. This proves~\eqref{eq:essential.equality.1}. It follows that \begin{displaymath}
	{e_{\{x\}}R_{R}} \, \stackrel{\eqref{eq:essential.equality.1}}{=} \, {\eta(p_{\{x\}}(M))_{R}} \, \cong \, {p_{\{x\}}(M)_{R}} \, \cong \, {R_{R}} 
\end{displaymath} for every $x \in X$. In turn, if $\emptyset \ne Y \subsetneq X$, then both \begin{displaymath}
	{R_{R}} \, \lesssim \, {e_{Y}R}_{R} \, \leq \, {R_{R}}
\end{displaymath} and \begin{displaymath}
	{R_{R}} \, \lesssim \, {e_{X\setminus Y}R}_{R} \, \stackrel{\eqref{eq:orthogonality}}{=} \, {(1-e_{Y})R}_{R} \, \leq \, {R_{R}} ,
\end{displaymath} thus Remark~\ref{remark:summands.of.injective.modules} and Theorem~\ref{theorem:schroeder.bernstein} imply that ${R_{R}} \cong {e_{Y}R}_{R}$ and ${R_{R}} \cong {(1-e_{Y})R}_{R}$, wherefore $e_{Y} \in \Emo(R)$.
	
\ref{lemma:essential.embedding.3} Let $Y \subseteq X$. If $g \in G(X,Y)$, then \begin{align*}
	&({\psi(g)} \circ (1-p_{Y}))(z) \, = \, ({\psi(g)} \circ {p_{X\setminus Y}})(z) \, = \, {g_{\bullet}} \cdot {\left({p_{X\setminus Y}(z)} \circ {g_{\circ}^{-1}}\right)} \, = \, p_{X\setminus Y}(z) \\
	&\qquad \qquad = \, p_{X\setminus Y}\!\left( {g_{\bullet}} \cdot {\left(z \circ {g_{\circ}^{-1}}\right)} \right) \, = \, ({p_{X\setminus Y}} \circ {\psi(g)})(z) \, = \, ((1-p_{Y}) \circ {\psi(g)})(z) 
\end{align*} for all $z \in M$, that is, $\psi(g) \circ (1-p_{Y}) = 1-p_{Y} = (1-p_{Y}) \circ \psi(g)$, which implies that \begin{displaymath}
	\psi(g) \, = \, p_{Y} \circ \psi(g) \circ p_{Y} + 1-p_{Y} \, \stackrel{\ref{lemma:corner.group}}{\in} \, \Gamma_{\End(M_{R})}(p_{Y})
\end{displaymath} and thus \begin{displaymath}
	\iota(g) \, = \, \phi(\psi(g)) \, \in \, \phi(\Gamma_{\End(M_{R})}(p_{Y})) \, \subseteq \, \Gamma_{R}(\phi(p_{Y})) \, = \, \Gamma_{R}(e_{Y}) .
\end{displaymath}
	
\ref{lemma:essential.embedding.2} Let $x \in X$. Then $\iota(G(X,\{x\})) \subseteq \Gamma_{R}(e_{\{x\}})$ thanks to~\ref{lemma:essential.embedding.3}. To prove the converse inclusion, let $a \in \Gamma_{R}(e_{\{x\}})$. We deduce that \begin{displaymath}
	a\eta(p_{\{x\}}(M)) \, \stackrel{\eqref{eq:essential.equality.1}}{=} \, ae_{\{x\}}R \, = \, e_{\{x\}}aR \, = \, e_{\{x\}}R \, \stackrel{\eqref{eq:essential.equality.1}}{=} \, \eta(p_{\{x\}}(M)) 
\end{displaymath} and \begin{align*}
	&a\eta(p_{X \setminus \{x\}}(M)) \, \stackrel{\eqref{eq:essential.equality.0}}{=} \, ae_{X\setminus \{x\}}\eta(p_{X \setminus \{x\}}(M)) \, \stackrel{\eqref{eq:orthogonality}}{=} \, a(1-e_{\{x\}})\eta(p_{X \setminus \{x\}}(M)) \\
	& \qquad = \, (1-e_{\{x\}})\eta(p_{X \setminus \{x\}}(M)) \, \stackrel{\eqref{eq:orthogonality}}{=} \, e_{X\setminus \{x\}}\eta(p_{X \setminus \{x\}}(M)) \, \stackrel{\eqref{eq:essential.equality.0}}{=} \, \eta(p_{X \setminus \{x\}}(M)) .
\end{align*} Hence, \begin{align*}
	a\eta(M) \, &= \, a\eta(p_{\{x\}}(M)) \oplus a\eta(p_{X \setminus \{x\}}(M)) \, = \, \eta(p_{\{x\}}(M)) \oplus \eta(p_{X \setminus \{x\}}(M)) \, = \, \eta(M) 
\end{align*} and so we find $b \in \GL(M_{R})$ such that $\phi(b) = a$. From \begin{displaymath}
	\phi(b) \, = \, a \, \in \, \Gamma_{R}(e_{\{x\}}) \cap \im(\phi) \, = \, \Gamma_{R}(\phi(p_{\{x\}})) \cap \im(\phi) \, = \, \phi(\Gamma_{\End(M_{R})}(p_{\{x\}})) ,
\end{displaymath} we infer that $b \in \Gamma_{\End(M_{R})}(p_{\{x\}})$. In particular, $b(p_{\{x\}}(M)) = p_{\{x\}}(b(M)) = p_{\{x\}}(M)$. As $p_{\{x\}}(M)_{R} \to R_{R}, \, z \mapsto z(x)$ is an isomorphism, Remark~\ref{remark:cayley}\ref{remark:cayley.1} asserts the existence of $c \in \GL(R)$ such that $b(z)(x) = cz(x)$ for all $z \in p_{\{x\}}(M)$. We conclude that \begin{equation}\label{eq:local.isomorphism.1}
	b(z)(x) \, = \, p_{\{x\}}(b(z))(x) \, = \, b(p_{\{x\}}(z))(x) \, = \, cp_{\{x\}}(z)(x) \, = \, cz(x) 
\end{equation} for all $z \in M$. On the other hand, if $y \in X \setminus \{ x\}$, then \begin{equation}\label{eq:local.isomorphism.2}
	b(z)(y) \, = \, (1-p_{\{x\}})(b(z))(y) \, = \, (1-p_{\{x\}})(z)(y) \, = \, z(y)
\end{equation} for every $z \in M$. Considering the element $g \defeq (g_{\bullet},1) \in G(X,\{x\})$ defined by \begin{displaymath}
	g_{\bullet}(y) \, \defeq \, \begin{cases}
			\, c & \text{if } y = x, \\
			\, 1 & \text{if } y \in X \setminus \{ x \} ,
		\end{cases}
\end{displaymath} from~\eqref{eq:local.isomorphism.1} and~\eqref{eq:local.isomorphism.2} we deduce that $b(z) = {g_{\bullet}} \cdot z = \psi(g)(z)$ for all $z \in M$, i.e., $b = \psi(g)$. Therefore, \begin{displaymath}
	a \, = \, \phi(b) \, = \, \phi(\psi(g)) \, = \, \iota(g) \, \in \, \iota(G(X,\{x\})) .
\end{displaymath}
	
\ref{lemma:essential.embedding.4} If $g \in \Sym(X)\setminus \{ 1\}$, then there exists $x \in X$ such that $g(x) \ne x$, and a straightforward calculation shows that ${\psi(1,g)} \circ {p_{\{x\}}} \circ {\psi(1,g)}^{-1} = p_{\{g(x)\}}$, whence \begin{displaymath}
	\iota(1,g) e_{\{x\}} \iota(1,g)^{-1} \, = \, \phi\!\left({\psi(1,g)} \circ {p_{\{x\}}} \circ {\psi(1,g)}^{-1}\right)\! \, = \, \phi(p_{\{g(x)\}}) \, = \, e_{\{g(x)\}} \, \perp \, e_{\{x\}} 
\end{displaymath} and so $\iota(1,g) \in \GLM(R)$ by~\ref{lemma:essential.embedding.1}.

\ref{lemma:essential.embedding.6} Let $u \in \I(\Sym(X))\setminus \{ 1\}$. Then $Y_{3} \defeq \Fix(u)$ is a proper subset of $X$. As $u$ is an involution, we find $Y_{1},Y_{2} \subseteq X$ such that $Y_{1} \cupdot Y_{2} = X\setminus Y_{3}$ and $u(Y_{1}) = Y_{2}$. Note that $(p_{Y_{1}},p_{Y_{2}},p_{Y_{3}}) \perp$ and $p_{Y_{1}} + p_{Y_{2}} + p_{Y_{3}} = 1$, which entails that $(e_{Y_{1}},e_{Y_{2}},e_{Y_{3}}) \perp$ and $e_{Y_{1}} + e_{Y_{2}} + e_{Y_{3}} = 1$. Necessarily, both $Y_{1}$ and $Y_{2}$ are proper, non-empty subsets of $X$, whence $e_{Y_{1}},e_{Y_{2}} \in \Emo(R)$ by~\ref{lemma:essential.embedding.1}. Moreover, either $Y_{3} = \emptyset$ and so $e_{Y_{3}} = 0$, or $Y_{3} \ne \emptyset$ and thus $e_{Y_{3}} \in \Emo(R)$ by~\ref{lemma:essential.embedding.1}. Note that $\iota(1,u) \in \I(R)$. One readily checks that $\psi(1,u) \circ p_{Y_{1}} \circ \psi(1,u) = p_{Y_{2}}$ and thus \begin{displaymath}
	\iota(1,u)e_{Y_{1}}\iota(1,u) \, = \, \phi(\psi(1,u) \circ p_{Y_{1}} \circ \psi(1,u)) \, = \, \phi(p_{Y_{2}}) \, = \, e_{Y_{2}} .
\end{displaymath} Also, $\psi(1,u) \circ p_{Y_{3}} = p_{Y_{3}} = p_{Y_{3}} \circ \psi(1,u)$ and hence $\iota(1,u)e_{Y_{3}} = e_{Y_{3}} = e_{Y_{3}}\iota(1,u)$, wherefore \begin{displaymath}
	\iota(1,u) \, \in \, \Gamma_{R}(1-e_{Y_{3}}) \, = \, \Gamma_{R}(e_{Y_{1}} + e_{Y_{2}}) \, = \, \Gamma_{R}(e_{Y_{1}} + \iota(1,u)e_{Y_{1}}\iota(1,u)) .
\end{displaymath} In particular, $\iota(1,u) \in \IME(R)$. Furthermore, the combination of Lemma~\ref{lemma:splits.over.moieties} and Lemma~\ref{lemma:splits.over.moieties.addon} with Lemma~\ref{lemma:conjugacy.of.moietious.involutions} implies that, for every $h \in \IM(R)$, \begin{equation}\label{eq:conjugation.square}
	\left. \left\{ ae_{Y_{1}} + \iota(1,u)a^{-1}\iota(1,u)e_{Y_{2}} + e_{Y_{3}} \, \right\vert a \in \GL(R), \, ae_{Y_{1}} = e_{Y_{1}}a \right\} \, \subseteq \, \Cl_{G}(h)^{2} .
\end{equation} To deduce the remaining claim, let $g \in G(u)$. Then $g_{\bullet}\vert_{Y_{3}} = 1$ and so $\psi(g) \circ p_{Y_{3}} = p_{Y_{3}}$, while also ${{g_{\bullet}} \circ u} = {g_{\bullet}^{-1}}$ and thus $\psi(g) = \psi(1,u) \circ \psi(g)^{-1} \circ \psi(1,u)$, wherefore \begin{align}
	\psi(g) \, &= \, (\psi(g) \circ p_{Y_{1}}) + (\psi(g) \circ p_{Y_{2}}) + (\psi(g) \circ p_{Y_{3}}) \notag \\
	& = \, (\psi(g) \circ p_{Y_{1}}) + \left(\psi(1,u) \circ \psi(g)^{-1} \circ \psi(1,u) \circ p_{Y_{2}}\right) + p_{Y_{3}} . \label{eq:decomposition}
\end{align} Moreover, from $g_{\circ} = 1$, we infer that $\psi(g) \circ p_{Y} = p_{Y} \circ \psi(g)$ and hence $\iota(g)e_{Y} = e_{Y}\iota(g)$ for every $Y \subseteq X$. Consequently, \begin{displaymath}
	\iota(g) \, \stackrel{\eqref{eq:decomposition}}{=} \, \iota(g)e_{Y_{1}} + \iota(1,u)\iota(g)^{-1}\iota(1,u)e_{Y_{2}} + e_{Y_{3}} \, \stackrel{\eqref{eq:conjugation.square}}{\in} \, \Cl_{G}(h)^{2} 
\end{displaymath} for every $h \in \IM(R)$, as desired.

\ref{lemma:essential.embedding.5} Let $x \in X$. By~\ref{lemma:essential.embedding.1} and Lemma~\ref{lemma:bounded.generation}, there exists $h \in \I(G)$ with $1-e_{\{ x\}} \leq he_{\{ x\}}h$ and $G = (\Gamma_{R}(e_{\{ x\}})h\Gamma_{R}(e_{\{ x\}})h)^{8}$. We conclude that \begin{align*}
	h\iota(G(X,X\setminus \{x\}))h \, &\stackrel{\ref{lemma:essential.embedding.3}}{\subseteq} \, h\Gamma_{R}(e_{X\setminus \{ x\}})h \, \stackrel{\eqref{eq:orthogonality}}{=} \, h\Gamma_{R}(1-e_{\{x\}})h \, \stackrel{\ref{lemma:corner.group}\ref{lemma:corner.group.3}}{=} \, \Gamma_{R}(h(1-e_{\{x\}})h) \\
	& \stackrel{\ref{lemma:corner.group}\ref{lemma:corner.group.1}}{\subseteq} \, \Gamma_{R}(e_{\{x\}}) \, \stackrel{\ref{lemma:essential.embedding.2}}{=} \, \iota(G(X,\{x\})) .
\end{align*}

\ref{lemma:essential.embedding.7} Suppose that $X$ is infinite. Let $x \in X$. According to Lemma~\ref{lemma:farrell.wagoner}, there exist $u,v \in \I(\Sym(X))$ such that $G(X,\{x\}) \subseteq G(u)G(v)$. Hence, for every $h \in \IM(R)$, \begin{displaymath} 
	\iota(G(X,\{x\})) \, \subseteq \, \iota(G(u)G(v)) \, = \, \iota(G(u))\iota(G(v)) \, \stackrel{\ref{lemma:essential.embedding.6}}{\subseteq} \, \Cl_{G}(h)^{4} . \qedhere
\end{displaymath} \end{proof}

We arrive at our first main result about unit groups of purely infinite rings.

\begin{thm}\label{theorem:universal.sequence} Let $R$ be a purely infinite, regular, right self-injective ring. Then $\GL(R)$ admits a universal sequence. In particular, $\GL(R)$ has strong distortion. \end{thm}

\begin{proof} This follows immediately from Theorem~\ref{theorem:purely.infinite}, Lemma~\ref{lemma:essential.embedding}\ref{lemma:essential.embedding.5}, Lemma~\ref{lemma:galvin}\ref{lemma:galvin.2}, and Remark~\ref{remark:strong.distortion}\ref{remark:strong.distortion.1}. \end{proof}

\begin{cor}\label{corollary:scf.for.rings} Let $R$ be a purely infinite, regular, right self-injective ring. Then $\scf(\GL(R)) \ne \aleph_{0}$. In particular, $\GL(R)$ has the Bergman property. \end{cor}

\begin{proof} The first assertion is due to Theorem~\ref{theorem:universal.sequence} and Remark~\ref{remark:strong.distortion.implies.uscf}\ref{remark:strong.distortion.implies.uscf.1}, while the final conclusion then follows by Remark~\ref{remark:strong.distortion.implies.uscf}\ref{remark:strong.distortion.implies.uscf.-1}. \end{proof}

\section{Bounded normal generation}\label{section:bng}

This section directly pertains to matters of bounded normal generation. The principal results are Theorem~\ref{theorem:generation.by.moietious.involutions} and Theorem~\ref{theorem:generation.by.moietious.units}. As immediate consequences, we deduce finite involution width (Corollary~\ref{corollary:generation.by.involutions}) and verbal simplicity (Theorem~\ref{theorem:verbally.elliptic}) of the considered groups.

\begin{lem}\label{lemma:generation.by.moietious.involutions} Let $R$ be a regular, right self-injective ring. If $e \in \Emo(R)$, then $\Gamma_{R}(e) \subseteq \Cl_{\GL(R)}(g)^{4}$ for every $g \in \IM(R)$. \end{lem}

\begin{proof} Let $e \in \Emo(R)$. Note that $R$ must be purely infinite due to Proposition~\ref{proposition:pure.infinity.via.moieties}. Fix any isomorphism $\eta_{0}\colon R_{R} \to eR_{R}$. Since $(1-e)R_{R} \cong R_{R}$, Theorem~\ref{theorem:purely.infinite} asserts the existence of an embedding $\eta_{1} \colon R_{R}^{\oplus \N\setminus \{0\}} \to (1-e)R_{R}$ with $\im(\eta_{1}) \less (1-e)R_{R}$. Then \begin{displaymath}
	\eta \colon \, R_{R}^{\oplus \N} \, \longrightarrow \, R_{R}, \quad z \, \longmapsto \, \eta_{0}(z(0)) + \eta_{1}\!\left(z\vert_{\N\setminus \{0\}}\right)
\end{displaymath} is an embedding and \begin{equation}\label{eq:fat}
	\im (\eta) \, = \, \im(\eta_{0}) \oplus \im(\eta_{1}) \, = \, eR \oplus \im(\eta_{1}) \, \less \, eR_{R} \oplus (1-e)R_{R} \, = \, R_{R}
\end{equation} by~\cite[Proposition~5.20(2)]{AndersonFuller}. As in Lemma~\ref{lemma:essential.embedding}, let $G \defeq \GL(R)$ and $M_{R} \defeq R_{R}^{\oplus \N}$, and consider the unital ring embedding\footnote{cf.~Lemma~\ref{lemma:essential.new}} \begin{displaymath}
	\phi \colon \, \End(M_{R}) \, \longrightarrow \, R, \quad a \, \longmapsto \, \overline{\eta \circ a \circ {\eta^{-1}}} ,
\end{displaymath} the group homomorphism $\psi \colon G(\N) \to \GL(M_{R})$ given by \begin{displaymath}
	\psi(g)z \, \defeq \, {g_{\bullet}} \cdot {\left(z \circ {g_{\circ}^{-1}}\right)} \qquad (g \in G(\N), \, z \in M) ,
\end{displaymath} the resulting homomorphism $\iota \defeq {\phi\vert_{\GL(M_{R})}} \circ \psi \colon G(\N) \to G$, and the idempotent endomorphism $p \colon M_{R} \to M_{R}$ defined by \begin{displaymath}
	p(z)(y) \, \defeq \, \begin{cases}
		\, z(0) & \text{if } y = 0 , \\
		\, 0 & \text{if } y \in \N \setminus \{ 0 \}
	\end{cases} \qquad (z \in M, \, y \in \N) .
\end{displaymath} We observe that \begin{displaymath}
	\phi(p)\eta(z) \, = \, \eta(p(z)) \, = \, \eta_{0}(z(0)) \, = \, e\eta(z)
\end{displaymath} for all $z \in M$. By~\eqref{eq:fat} and Lemma~\ref{lemma:essential.inclusion}\ref{lemma:essential.inclusion.2}, this implies that $\phi(p) = e$. We conclude that, for every $g \in \IM(R)$, \begin{displaymath}
	\Gamma_{R}(e) \, = \, \Gamma_{R}(\phi(p)) \, \stackrel{\ref{lemma:essential.embedding}\ref{lemma:essential.embedding.2}}{=} \, \iota(G(\N,\{0\})) \, \stackrel{\ref{lemma:essential.embedding}\ref{lemma:essential.embedding.7}}{\subseteq} \, \Cl_{G}(g)^{4} \, = \, \Cl_{\GL(R)}(g)^{4} . \qedhere	
\end{displaymath} \end{proof}

\begin{lem}\label{lemma:generation.by.moietious.involutions.extended} Let $R$ be a regular, right self-injective ring. If $e \in \Emo(R)$, then $\Gamma_{R}(e) \subseteq \Cl_{\GL(R)}(g)^{4}$ for every $g \in \IME(R)$. \end{lem}

\begin{proof} Suppose that $e \in \Emo(R)$. Necessarily, $R$ must be purely infinite according to Proposition~\ref{proposition:pure.infinity.via.moieties}. Let $g \in \IME(R)$. According to Lemma~\ref{lemma:characterization.of.ime}, there exist $f \in \E(R)$ and $g' \in \IM(fRf)$ such that ${R_{R}} \cong {fR_{R}}$ and $g = g'+1-f$. Then \begin{displaymath}
	R \, \stackrel{\ref{remark:cayley}\ref{remark:cayley.1}}{\cong} \, \End(R_{R}) \, \cong \, \End(fR_{R}) \, \stackrel{\ref{remark:cayley}\ref{remark:cayley.1}}{\cong} \, fRf \, \eqdef \, S .
\end{displaymath} In particular, $S$ is a purely infinite, regular, right self-injective ring. According to Proposition~\ref{proposition:pure.infinity.via.moieties}, we find some $f' \in \Emo(S)$. Using the group isomorphism \begin{displaymath}
	\iota \colon \, \GL(S) \, \longrightarrow \, \Gamma_{R}(f), \quad a \, \longmapsto \, a +1-f
\end{displaymath} from Lemma~\ref{lemma:corner.group}, we conclude that \begin{align*}
	\Gamma_{R}(f') \, &= \, \GL(f'Rf') + 1-f' \, = \, \GL(f'Sf') + f-f' + 1-f \, = \, \Gamma_{S}(f') + 1-f \\
	& = \, \iota(\Gamma_{S}(f')) \, \stackrel{\ref{lemma:generation.by.moietious.involutions}}{\subseteq} \, \iota\!\left(\Cl_{\GL(S)}(g')^{4}\right) \, = \, \Cl_{\Gamma_{R}(f)}(g)^{4} \, \subseteq \, \Cl_{\GL(R)}(g)^{4} .
\end{align*} Moreover, as $f' \in \Emo(S)$, we know that ${f'S_{S}} \cong {S_{S}} \cong {(f-f')S_{S}}$. While the former isomorphism entails that \begin{displaymath}
	{R_{R}} \, \cong \, {fR_{R}} \, \stackrel{\ref{lemma:ideal.isomorphism}}{\cong} \, {f'R_{R}}  ,
\end{displaymath} the latter implies that \begin{displaymath} 
	{R_{R}} \, \cong \, {fR_{R}} \, \stackrel{\ref{lemma:ideal.isomorphism}}{\cong} \, {(f-f')R_{R}} \, \leq \, {(1-f')R_{R}} \, \leq \, {R_{R}} 
\end{displaymath} and thus $R_{R} \cong (1-f')R_{R}$ by Remark~\ref{remark:summands.of.injective.modules} and Theorem~\ref{theorem:schroeder.bernstein}. Therefore, $f' \in \Emo(R)$. Consequently, if $e \in \Emo(R)$, then Lemma~\ref{lemma:conjugacy.of.moieties}\ref{lemma:conjugacy.of.moieties.1} asserts the existence of some $h \in \GL(R)$ such that $e = hf'h^{-1}$, whence \begin{displaymath}
	\Gamma_{R}(e) \, \stackrel{\ref{lemma:corner.group}\ref{lemma:corner.group.3}}{=} \, {h\Gamma_{R}(f')h^{-1}} \, \subseteq \, h\Cl_{\GL(R)}(g)^{4}h^{-1} \, = \, \Cl_{\GL(R)}(g)^{4} . \qedhere
\end{displaymath} \end{proof}

\begin{thm}\label{theorem:generation.by.moietious.involutions} Let $R$ be a regular, right self-injective ring. For every $g \in \IME(R)$, \begin{displaymath}
	\GL(R) \, = \, \Cl_{\GL(R)}(g)^{64} .
\end{displaymath} \end{thm}

\begin{proof} Due to Proposition~\ref{proposition:purely.infinite.involutions}, we may and will assume without loss of generality that $R$ is purely infinite. By Proposition~\ref{proposition:pure.infinity.via.moieties}, there exist $e_{1},e_{2},e_{3} \in \Emo(R)$ pairwise orthogonal such that $e_{1} + e_{2} + e_{3} = 1$. Then $e_{1}+e_{2} = 1-e_{3}$ and $e_{2}+e_{3} = 1-e_{1}$ are moieties in $R$, too. For every $g \in \IME(R)$, we conclude that \begin{displaymath}
	\GL(R) \, \stackrel{\ref{lemma:decomposition.into.moieties}}{=} \, (\Gamma_{R}(e_{1}+e_{2})\Gamma_{R}(e_{2}+e_{3}))^{8} \, \stackrel{\ref{lemma:generation.by.moietious.involutions.extended}}{\subseteq} \, \Cl_{\GL(R)}(g)^{64} . \qedhere
\end{displaymath} \end{proof}

\begin{cor}\label{corollary:generation.by.moietious.involutions} Let $R$ be a purely infinite, regular, right self-injective ring. Then there exists $g \in \I(R)$ such that $\GL(R) = \Cl_{\GL(R)}(g)^{64}$. \end{cor}

\begin{proof} This is a direct consequence of Proposition~\ref{proposition:purely.infinite.involutions} and Theorem~\ref{theorem:generation.by.moietious.involutions}. \end{proof}

\begin{cor}\label{corollary:generation.by.involutions} If $R$ is a purely infinite, regular, right self-injective ring, then \begin{displaymath}
	\GL(R) \, = \, \I(R)^{64} .
\end{displaymath} \end{cor}

\begin{proof} This follows immediately from Corollary~\ref{corollary:generation.by.moietious.involutions}. \end{proof}

We advance beyond involutions using the following lemma.

\begin{lem}\label{lemma:generating.moietious.involutions} Let $R$ be a regular, right self-injective ring. If $g \in \GLM(R)$, then \begin{displaymath}
	\IME(R) \cap \Cl_{\GL(R)}\!\left( \left\{ g,g^{-1}\right\}\right)^{2} \, \ne \, \emptyset .
\end{displaymath} \end{lem}

\begin{proof} Let $g \in \GLM(R)$. Choose any $e \in \E(R)$ such that ${R_{R}} \cong {eR_{R}}$ and $e \perp geg^{-1}$. Then \begin{displaymath}
	{R_{R}} \, \cong \, {eR_{R}} \, \cong \, {geR_{R}} \, = \, {geg^{-1}R_{R}} \, \leq \, {(1-e)R_{R}} \, \leq \, {R_{R}}
\end{displaymath} and hence ${R_{R}} \cong {(1-e)R_{R}}$ by Remark~\ref{remark:summands.of.injective.modules} and Theorem~\ref{theorem:schroeder.bernstein}, wherefore $e \in \Emo(R)$. By Lemma~\ref{lemma:splitting.moieties}, there exist $e_{0},e_{1} \in \Emo(R)$ such that ${e_{0}} \perp {e_{1}}$ and $e = e_{0} + e_{1}$. Since $e \perp geg^{-1}$, we know that \begin{displaymath}
	f_{0} \, \defeq \, e_{0} + ge_{0}g^{-1}, \qquad f_{1} \, \defeq \, e_{1} + ge_{1}g^{-1}
\end{displaymath} are elements of $\E(R)$ with ${f_{0}} \perp {f_{1}}$ and $f_{0} + f_{1} = e+geg^{-1}$. Also, we observe that ${R_{R}} \cong {e_{0}R_{R}} \leq {f_{0}R_{R}} \leq {R_{R}}$ and thus ${R_{R}} \cong {f_{0}R_{R}}$ by Remark~\ref{remark:summands.of.injective.modules} and Theorem~\ref{theorem:schroeder.bernstein}. Furthermore, according to Lemma~\ref{lemma:conjugacy.of.moieties}\ref{lemma:conjugacy.of.moieties.2}, we find $h \in \I(\Gamma_{R}(e))$ such that $he_{0}h = e_{1}$. As $h \in \Gamma_{R}(e)$ and therefore \begin{displaymath}
	ghg^{-1} \, \in \, g\Gamma_{R}(e)g^{-1} \, \stackrel{\ref{lemma:corner.group}\ref{lemma:corner.group.3}}{=} \, \Gamma_{R}\!\left( geg^{-1}\right) 
\end{displaymath} and moreover $e \perp geg^{-1}$, we see that \begin{equation}\label{eq:commute}
	hghg^{-1} \, \stackrel{\ref{lemma:corner.group}\ref{lemma:corner.group.2}}{=} \, ghg^{-1}h .
\end{equation} Since $h \in \I(R)$ and thus $ghg^{-1} \in \I(R)$, we infer from~\eqref{eq:commute} that $t \defeq hghg^{-1} \in \I(R)$. We also observe that \begin{align*}
	te_{0}t \, &\stackrel{\eqref{eq:commute}}{=} \, ghg^{-1}he_{0}hghg^{-1} \, = \, ghg^{-1}e_{1}ghg^{-1} \, \stackrel{e_{1} \perp geg^{-1}}{=} \, e_{1} , \\
	tge_{0}g^{-1}t \, &\stackrel{\eqref{eq:commute}}{=} \, hghg^{-1}ge_{0}g^{-1}ghg^{-1}h \, = \, hghe_{0}hg^{-1}h \, = \, hge_{1}g^{-1}h \stackrel{e \perp ge_{1}g^{-1}}{=} \, ge_{1}g^{-1} ,
\end{align*} which entails that \begin{equation}\label{eq:conjugate}
	tf_{0}t \, = \, te_{0}t + tge_{0}g^{-1}t \, = \, e_{1} + ge_{1}g^{-1} \, = \, f_{1} .
\end{equation} From~\eqref{eq:conjugate} and ${f_{0}} \perp {f_{1}}$, we infer that $f_{0} \perp tf_{0}t$. Moreover, \begin{displaymath}
	t \, \in \, \Gamma_{R}(e)\Gamma_{R}\!\left( geg^{-1}\right) \, \stackrel{\ref{lemma:corner.group}\ref{lemma:corner.group.1}}{\subseteq} \, \Gamma_{R}\!\left(e + geg^{-1}\right) \, = \, \Gamma_{R}(f_{0}+f_{1}) \, \stackrel{\eqref{eq:conjugate}}{=} \, \Gamma_{R}(f_{0}+tf_{0}t).
\end{displaymath} Hence, $t \in \IME(R)$. Finally, \begin{displaymath}
	t \, = \, ghg^{-1}h \, = \, ghg^{-1}h^{-1} \, \in \, \Cl_{\GL(R)}\!\left( \left\{ g,g^{-1}\right\}\right)^{2} 
\end{displaymath} and so $t \in \IME(R) \cap \Cl_{\GL(R)}\!\left( \left\{ g,g^{-1}\right\}\right)^{2}$. Hence, $\IME(R) \cap \Cl_{\GL(R)}\!\left( \left\{ g,g^{-1}\right\}\right)^{2} \ne \emptyset$. \end{proof}

\begin{thm}\label{theorem:generation.by.moietious.units} Let $R$ be a regular, right self-injective ring. For every $g \in \GLM(R)$, \begin{displaymath}
	\GL(R) \, = \, \Cl_{\GL(R)}\!\left( \left\{ g,g^{-1}\right\}\right)^{128} .
\end{displaymath} \end{thm}

\begin{proof} Let $g \in \GLM(R)$. By Lemma~\ref{lemma:generating.moietious.involutions}, we find $h \in \IME(R) \cap \Cl_{\GL(R)}\!\left( \left\{ g,g^{-1}\right\}\right)^{2}$. Hence, \begin{displaymath}
	\GL(R) \, \stackrel{\ref{theorem:generation.by.moietious.involutions}}{=} \, \Cl_{\GL(R)}(h)^{64} \, \subseteq \, \Cl_{\GL(R)}\!\left( \left\{ g,g^{-1}\right\}\right)^{128} .\qedhere
\end{displaymath} \end{proof}

We conclude this section by deducing verbal simplicity and verbal ellipticity of the considered groups. Our argument combines Theorem~\ref{theorem:generation.by.moietious.involutions} with the following result due to Larsen and Shalev, established in both~\cite{LarsenShalev08} and~\cite{LarsenShalev09} via independent methods.

\begin{thm}[{\cite[Theorem~1.19]{LarsenShalev08}}; {\cite[Theorem~1.5]{LarsenShalev09}}]\label{theorem:larsen.shalev} Let $w \in \free(\N)\setminus \{ 1 \}$. Then there exists $n \in \N_{>0}$ such that \begin{displaymath}
	\forall m \in \N_{\geq n} \colon \qquad w(\Alt(m))^{2} \, = \, \Alt(m) .
\end{displaymath} \end{thm}

The proof of the following result is similar to that of~\cite[Theorem~1.6]{Schneider25}.

\begin{thm}\label{theorem:verbally.elliptic} Let $R$ be a purely infinite, regular, right self-injective ring. For every $w \in \free (\N)\setminus \{ 1 \}$, \begin{displaymath}
	\GL(R) \, = \, w(\GL(R))^{128} .
\end{displaymath} In particular, $\GL(R)$ is verbally simple and verbally elliptic. \end{thm}

\begin{proof} Let $w \in \free(\N)\setminus \{ 1 \}$. Due to Theorem~\ref{theorem:larsen.shalev}, there exists $n \in \N$ such that \begin{equation}\label{eq:larsen.shalev}
	\forall m \in \N_{\geq n} \colon \qquad w(\Alt(m))^{2} \, = \, \Alt(m) .
\end{equation} Let $m \defeq \max \{ n, 4 \}$. According to Theorem~\ref{theorem:purely.infinite} and Lemma~\ref{lemma:essential.embedding}\ref{lemma:essential.embedding.6}, there exists a homomorphism $\iota \colon \Sym(m) \to \GL(R)$ such that $\iota(\I(\Sym(m))\setminus \{ 1\}) \subseteq \IME(R)$. Since $m \geq 4$, we know that $\I(\Alt(m))\setminus \{ 1\} \ne \emptyset$, whence Theorem~\ref{theorem:generation.by.moietious.involutions} entails that \begin{displaymath}
	\GL(R) \, = \, \Cl_{\GL(R)}(\iota(\Alt(m)))^{64} .
\end{displaymath} From this and \begin{displaymath}
	\Alt(m) \, \stackrel{\eqref{eq:larsen.shalev}}{=} \, w(\Alt(m))^{2} \, \stackrel{\ref{remark:word.images}\ref{remark:word.images.1}}{=} \, w^{\otimes 2}(\Alt(m)) ,
\end{displaymath} we conclude that \begin{displaymath}
	\GL(R) \, \stackrel{\ref{remark:word.images}\ref{remark:word.images.2}}{=} \, w^{\otimes 2}(\GL(R))^{64} \, \stackrel{\ref{remark:word.images}\ref{remark:word.images.1}}{=} \, w(\GL(R))^{128} . \qedhere
\end{displaymath} \end{proof}

\section{Goodearl--Boyle dimension}\label{section:goodearl.boyle.dimension}

This final section establishes several connections between structural properties of unit groups on the one hand and a notion of dimension suited for purely infinite, regular, right self-injective rings on the other. The latter is formulated on the basis of Theorem~\ref{theorem:purely.infinite}. According to~\cite[\S1.D, Proposition~1.27(a)]{GoodearlRingTheory}, if $R$ is a regular ring, then the module $R_{R}$ is \emph{non-singular} in the sense of~\cite[\S1.D, p.~30]{GoodearlRingTheory}, i.e., \begin{displaymath}
	\{ x \in R \mid \{ y \in R \mid xy = 0 \} \less {R_{R}} \} \, = \, \{ 0 \} .
\end{displaymath} This justifies the following specialization of the definition given in~\cite[XII, p.~92]{GoodearlBoyle}. %Alternative references: \cite[3, \S7A, (7.5)~Definition, p.~247]{lam} and~\cite[3, \S7D, (7.52)~Theorem, (1), p.~262]{lam}.

\begin{definition}[cf.~{\cite[XII, p.~92]{GoodearlBoyle}}]\label{definition:goodearl.boyle} Let $R$ be a regular, right self-injective ring. The \emph{Goodearl--Boyle dimension} of $R$, denoted by $\mu(R)$, is defined as follows: if $R = \{ 0 \}$, then $\mu(R) \defeq 0$; and if $R \ne \{ 0 \}$, then $\mu(R)$ is the smallest infinite cardinal $\alpha$ such that $R_{R}^{\oplus \alpha} \not\lesssim R$. \end{definition}

\begin{remark}[{\cite[XII, p.~92]{GoodearlBoyle}}]\label{remark:goodearl.boyle} Let $R$ be a non-zero, regular, right self-injective ring. Then $\mu(R)$ is the smallest infinite cardinal $\alpha$ such that $R_{R}^{\oplus \alpha}$ is not isomorphic to an essential submodule of $R_{R}$. \end{remark}

\begin{prop}[{\cite[XII, Proposition~12.1, p.~92]{GoodearlBoyle}}]\label{proposition:goodearl.boyle} A regular, right self-injective ring $R$ is purely infinite if and only if $\mu(R) \ne \aleph_{0}$. \end{prop}

Our first result involving the Goodearl--Boyle dimension concerns (a qualitative strengthening of) minimal almost periodicity.

\begin{prop}\label{proposition:uniformly.bounded.normal.generation.over.large.subgroup} Let $R$ be a regular, right self-injective ring. For every infinite cardinal $\alpha$ with $\alpha < \mu(R)$, the group $\GL(R)$ has uniformly bounded normal generation over some subgroup isomorphic to $\Sym(\alpha)$. \end{prop}

\begin{proof} This is an immediate consequence of Remark~\ref{remark:goodearl.boyle}, Lemma~\ref{lemma:essential.embedding}\ref{lemma:essential.embedding.4}, and Theorem~\ref{theorem:generation.by.moietious.units}. \end{proof}

\begin{cor}\label{corollary:unit.group.minap} If $R$ is a purely infinite, regular, right self-injective ring, then $\GL(R)$ is minimally almost periodic. \end{cor}

\begin{proof} If $R = \{ 0 \}$, then $\GL(R)$ is trivial, in particular minimally almost periodic. Otherwise, the desired conclusion follows from Proposition~\ref{proposition:goodearl.boyle}, Proposition~\ref{proposition:uniformly.bounded.normal.generation.over.large.subgroup}, and Proposition~\ref{proposition:rigidity.for.SIN.reps}\ref{proposition:rigidity.for.SIN.reps.1}. \end{proof}

The next result establishes a connection between the Goodearl--Boyle dimension of a ring and uniform strong cofinality (Definition~\ref{definition:uniform.strong.cofinality}) of its unit group. The proof involves an application of Theorem~\ref{theorem:abstract.2}.

\begin{thm}\label{theorem:uniform.bergman.unit.group} Let $R$ be a non-zero, purely infinite, regular, right self-injective ring. Then \begin{displaymath}
	\aleph_{0} \, < \, \mu(R) \, \leq \, \scf_{6144}(\GL(R)) .
\end{displaymath} \end{thm}

\begin{proof} Note that $\mu(R) > \aleph_{0}$ by Proposition~\ref{proposition:goodearl.boyle}. In particular, $\GL(R)$ is infinite due to Theorem~\ref{theorem:purely.infinite} and Lemma~\ref{lemma:essential.new}. Let $\alpha$ be an infinite cardinal with $\mu(R) > \alpha$. We show that $\scf_{6144}(\GL(R)) > \alpha$ by verifying the conditions specified in Theorem~\ref{theorem:abstract.2} for $m \defeq 64$, $\ell \defeq 32$, and $\mathcal{H} \defeq \{ \Gamma_{R}(e) \mid e \in \Emo(R) \}$.

\ref{theorem:abstract.2.a} Fix any set $X$ with $\vert X \vert = \alpha$. Let $M_{R} \defeq R_{R}^{\oplus X}$, $H \defeq \GL(R)^{X}$ and \begin{displaymath}
	H_{x} \, \defeq \, \{ h \in H \mid h\vert_{X\setminus \{ x\}} = 1 \} \qquad (x \in X) .
\end{displaymath} Consider the group embedding $\psi \colon H \to \GL(M_{R})$ given by \begin{displaymath}
	\psi(g)z \, \defeq \, g \cdot z \qquad (g \in H, \, z \in M) .
\end{displaymath} By Remark~\ref{remark:goodearl.boyle}, since $\mu(R) > \alpha = \vert X \vert$, there exists an embedding $\eta \colon M_{R} \to R_{R}$ with $\im(\eta) \less R_{R}$. Consider the unital ring embedding\footnote{cf.~Lemma~\ref{lemma:essential.new}} \begin{displaymath}
	\phi \colon \, \End(M_{R}) \, \longrightarrow \, R, \quad a \, \longmapsto \, \overline{\eta \circ a \circ {\eta^{-1}}} 
\end{displaymath} as well as the resulting group embedding $\iota \defeq {\phi\vert_{\GL(M_{R})}} \circ \psi \colon H \to \GL(R)$. According to Lemma~\ref{lemma:essential.embedding}\ref{lemma:essential.embedding.1}+\ref{lemma:essential.embedding.2}, for each $x \in X$, the idempotent endomorphism $p_{x} \colon M_{R} \to M_{R}$ defined by \begin{displaymath}
	p_{x}(z)(y) \, \defeq \, \begin{cases}
		\, z(y) & \text{if } y = x , \\
		\, 0 & \text{if } y \in X \setminus \{ x \}
	\end{cases} \qquad (z \in M, \, y \in X) 
\end{displaymath} gives rise to an element $e_{x} \defeq \phi(p_{x}) \in \Emo(R)$ with $\iota(H_{x}) = \Gamma_{R}(e_{x})$. As $H$ is full for $(H_{x})_{x \in X}$ thanks to Lemma~\ref{lemma:direct.products.are.full} and $\iota$ is a homomorphism, Lemma~\ref{lemma:full.homomorphic.images} asserts that $\iota (H)$ is full for $(\iota(H_{x}))_{x \in X} = (\Gamma_{R}(e_{x}))_{x \in X}$, whence $\GL(R)$ is full for $(\Gamma_{R}(e_{x}))_{x \in X} \in \mathcal{H}^{X}$ according to Remark~\ref{remark:full.persistence}.

\ref{theorem:abstract.2.b} Let $e \in \Emo(R)$. As $R_{R} \cong eR_{R}$, we see that \begin{displaymath}
	R \, \stackrel{\ref{remark:cayley}\ref{remark:cayley.1}}{\cong} \, \End(R_{R}) \, \cong \, \End(eR_{R}) \, \stackrel{\ref{remark:cayley}\ref{remark:cayley.1}}{\cong} \, eRe .
\end{displaymath} In particular, $eRe$ is a purely infinite, regular, right self-injective ring. Hence, by Corollary~\ref{corollary:generation.by.moietious.involutions}, there exists $g \in \I(eRe)$ such that $\GL(eRe) = \Cl_{\GL(eRe)}(g)^{64}$. As \begin{displaymath}
	\iota \colon \, \GL(eRe) \, \longrightarrow \, \Gamma_{R}(e), \quad a \, \longmapsto \, a+1-e
\end{displaymath} is a group isomorphism by Lemma~\ref{lemma:corner.group}, we conclude that $\Gamma_{R}(e) = \Cl_{\Gamma_{R}(e)}(\iota(g))^{64}$.

\ref{theorem:abstract.2.c} Let $e \in \Emo(R)$. By Lemma~\ref{lemma:bounded.generation}, there exists $g \in \I(R)$ such that \begin{displaymath}
	\GL(R) \, = \, (\Gamma_{R}(e)g\Gamma_{R}(e)g)^{8} \, \subseteq \, (\Gamma_{R}(e) \cup \{ g \})^{32} .
\end{displaymath}

By Theorem~\ref{theorem:abstract.2}, it follows that $\scf_{6144}(\GL(R)) > \alpha$. This completes the proof. \end{proof}

We obtain a qualitative improvement of Corollary~\ref{corollary:scf.for.rings}.

\begin{cor}\label{corollary:usc.for.rings} Let $R$ be a non-zero, purely infinite, regular, right self-injective ring. Then \begin{displaymath}
	\aleph_{0} \, < \, \mu(R) \, \leq \, \scf(\GL(R)) .
\end{displaymath} In particular, $\GL(R)$ has the Bergman property. \end{cor}

\begin{proof} The first assertion follows by Theorem~\ref{theorem:uniform.bergman.unit.group} and Lemma~\ref{lemma:uniform.strong.cofinality}\ref{lemma:uniform.strong.cofinality.3}. The final assertion follows by Remark~\ref{remark:strong.distortion.implies.uscf}\ref{remark:strong.distortion.implies.uscf.-1}. \end{proof}

We describe further representation-theoretic ramifications.

\begin{cor}\label{corollary:restricted.kazhdan.property} Let $R$ be a regular, right self-injective ring. Then $\GL(R)$ has property $(T_{\alpha})$ for every infinite cardinal $\alpha < \mu(R)$. In particular, if $R$ is non-zero and purely infinite, then $\GL(R)$ has property $(T_{\aleph_{0}})$, but not property $(T)$. \end{cor}

\begin{proof} This follows from Proposition~\ref{proposition:goodearl.boyle}, Theorem~\ref{theorem:uniform.bergman.unit.group}, and Proposition~\ref{proposition:restricted.kazhdan.property}. \end{proof}

\begin{cor}\label{corollary:homomorphism.rigidity} Let $R$ be a regular, right self-injective ring and let $\alpha$ be an infinite cardinal with $\alpha < \mu(R)$. For every non-trivial homomorphism $\pi \colon \GL(R) \to H$ to a Hausdorff topological group $H$ \begin{itemize}
	\item[---\,] having small invariant neighborhoods, or
	\item[---\,] being uniformly free from small subgroups\footnote{e.g., a Banach--Lie group (cf.~Remark~\ref{remark:banach.lie})},
\end{itemize} the image $\pi(\GL(R))$ contains a discrete subgroup isomorphic to $ \Sym(\alpha)$. \end{cor}

\begin{proof} This is a consequence of Proposition~\ref{proposition:uniformly.bounded.normal.generation.over.large.subgroup}, Lemma~\ref{lemma:rigidity.for.SIN.reps}, Proposition~\ref{proposition:goodearl.boyle}, Corollary~\ref{corollary:usc.for.rings} (or Corollary~\ref{corollary:scf.for.rings}), and Proposition~\ref{proposition:rigidity.for.BL.reps}. \end{proof}

We conclude by further examining Theorem~\ref{theorem:uniform.bergman.unit.group} with respect to Example~\ref{example:rings}\ref{example:rings.1}: if $V$ is an infinite-dimensional right vector space over a division ring, then $\mu(\End(V))$ equals the successor cardinal of $\dim(V)$, as noted in~\cite[XII, p.~92]{GoodearlBoyle} (see also~\cite[Chapter~12, p.~143]{GoodearlBook}), and hence Theorem~\ref{theorem:uniform.bergman.unit.group} yields that $\dim(V) < \scf_{6144}(\GL(V))$. Combining our Theorem~\ref{theorem:abstract.2} with the work of Tolstykh~\cite{tolstykh}, we obtain the following slight quantitative improvement for this particular family of examples.

\begin{thm}[cf.~{\cite[Theorem~2.1]{tolstykh}}]\label{theorem:tolstykh} Let $V$ be an infinite-dimensional right vector space over a division ring. Then \begin{displaymath}
	\dim(V) \, < \, \scf_{1764}(\GL(V)) .
\end{displaymath} In particular, $\GL(V)$ has property $(T_{\dim(V)})$. \end{thm}

\begin{proof} The proof is analogous to that of Theorem~\ref{theorem:uniform.bergman.unit.group} and will be included only for the sake of completeness. Let $E \defeq \{ e \in \E(\End(V)) \mid e(V) \cong V \cong (1-e)(V) \}$. Consider the group $G \defeq \GL(V) = \GL(\End(V))$ along with the set \begin{displaymath}
	\mathcal{H} \, \defeq \, \{ \Gamma_{\End(V)}(e) \mid e \in E \} \, \subseteq \, \Sub(G) .
\end{displaymath} We verify the conditions specified in Theorem~\ref{theorem:abstract.2} for this choice of objects with respect to $m \defeq 28$ and $\ell \defeq 21$.	
	
\ref{theorem:abstract.2.a} Let $X$ be a basis of $V$. In particular, $\vert X \vert = \dim(V)$. As $X$ is infinite, we may invoke the axiom of choice to find a family $(Y_{x})_{x \in X}$ of pairwise disjoint subsets of~$X$ such that $X = \bigcupdot_{x \in X} Y_{x}$ and $\vert Y_{x} \vert = \vert X \vert$ for every $x \in X$. For each $x \in X$, consider the unique element $e_{x} \in \E(\End(V))$ such that \begin{displaymath}
	e_{x}(y) \, = \, \begin{cases}
		\, y & \text{if } y \in Y_{x}, \\
		\, 0 & \text{if } y \in X\setminus Y_{x}.
	\end{cases}
\end{displaymath} Evidently, $\{ e_{x} \mid x \in X \} \subseteq E$ by the Schröder--Bernstein theorem. Consider the group $H \defeq \prod_{x \in X} \GL(e_{x}(V))$ along with its subgroups \begin{displaymath}
	H_{x} \, \defeq \, \{ h \in H \mid \forall y \in X \setminus \{ x \} \colon \, h_{y} = 1 \} \qquad (x \in X)
\end{displaymath} and the embedding $\iota \colon H \to G$ defined by \begin{displaymath}
	\iota(h)\vert_{e_{x}(V)} \, \defeq \, h_{x} \qquad (h \in H, \, x \in X) .
\end{displaymath} Since $H$ is full for $(H_{x})_{x \in X}$ according to Lemma~\ref{lemma:direct.products.are.full} and $\iota$ is a homomorphism, Lemma~\ref{lemma:full.homomorphic.images} asserts that $\iota (H)$ is full for $(\iota(H_{x}))_{x \in X} = (\Gamma_{\End(V)}(e_{x}))_{x \in X}$, whence $G$ is full for $(\Gamma_{\End(V)}(e_{x}))_{x \in X} \in \mathcal{H}^{X}$ thanks to Remark~\ref{remark:full.persistence}.	
	
\ref{theorem:abstract.2.b} Let $e \in E$. Since $e(V)$ is infinite-dimensional, \cite[Theorem~1.1]{tolstykh} asserts the existence of an involution $g \in \GL(e(V))$ such that $\GL(e(V)) = \Cl_{\GL(e(V))}(g)^{28}$. As $\GL(e(V)) \cong \Gamma_{\End(V)}(e)$, this entails the desired conclusion.
	
\ref{theorem:abstract.2.c} Let $e \in E$. Using the axiom of choice, we find $e_{0},e_{1} \in E$ such that $e_{0} \perp e_{1}$ and $e = e_{0} + e_{1}$. Moreover, let $e_{2} \defeq 1-e$. Since $1 = e_{0} + e_{1} + e_{2}$, $(e_{0},e_{1},e_{2})\perp$ and $e_{1}(V) \cong e_{2}(V)$, there exists an involution $g \in \GL(V)$ such that $ge_{0} = e_{0}g = e_{0}$ and $ge_{1} = e_{2}g$. Hence, $geg = ge_{0}g + ge_{1}g = e_{0}+e_{2}$ and therefore \begin{displaymath}
	\Gamma_{\End(V)}(e_{0}+e_{2}) \, = \, \Gamma_{\End(V)}(geg) \, = \, g\Gamma_{\End(V)}(e)g \, \subseteq \, (\Gamma_{\End(V)}(e) \cup \{ g \})^{3} .
\end{displaymath} As $\GL(V) = (\Gamma_{\End(V)}(e_{0}+e_{1}) \cup \Gamma_{\End(V)}(e_{0}+e_{2}))^{7}$ according to~\cite[Lemma~1.2]{tolstykh}, we conclude that $\GL(V) = (\Gamma_{\End(V)}(e) \cup \{ g \})^{21}$.
	
Thanks to Theorem~\ref{theorem:abstract.2}, this yields that $\dim(V) < \scf_{1764}(\GL(V))$. In particular, $\GL(V)$ has property $(T_{\dim(V)})$ by Proposition~\ref{proposition:restricted.kazhdan.property}. \end{proof}

\section*{Disclosure of AI use}

No AI tools were used to develop the ideas, arguments, results and proofs in this work, nor to write the original manuscript. OpenAI's ChatGPT (GPT‑6 Astra Pro) was used only during the subsequent proofreading and mathematical review of the manuscript, to identify possible linguistic, typographical, notational and mathematical errors and to suggest corrections, and to assist the author in writing this disclosure statement. All revisions were made by the author, who takes full responsibility for the accuracy and integrity of the final manuscript.

\end{document}